\documentclass [leqno, spanish, 10pt]{amsart}

\AtBeginDocument{%
   \def\MR#1{}
}

\AtBeginDocument{%
   \def\DOI#1{}
}
\AtBeginDocument{%
   \def\URL#1{}
   \def\ISBN#1{}
}

\usepackage[T1]{fontenc}
\usepackage[utf8]{inputenc}

\usepackage[headings, myheadings]{fullpage}
\usepackage{xcolor}
\newcommand{\revnew}[1]{\textcolor{black}{#1}}
\usepackage{lmodern}
\usepackage{microtype}

\usepackage[all]{xy}
\usepackage{amssymb, amsmath, amsfonts, amsthm, amsbsy, amscd, amsxtra, stmaryrd}
\usepackage{mathabx}
\usepackage[bbgreekl]{mathbbol} %
\usepackage[vcentermath]{youngtab} %
\usepackage{pgf,interval}
\usepackage{ytableau}
\usepackage[mathscr]{euscript}
\usepackage{upgreek}
\usepackage{mathtools}

\usepackage{tikz}
\usetikzlibrary{arrows}
\usetikzlibrary{decorations.pathreplacing}
\makeatletter
\def\underbrace#1{\@ifnextchar_{\tikz@@underbrace{#1}}{\tikz@@underbrace{#1}_{}}}
\def\tikz@@underbrace#1_#2{\tikz[baseline=(a.east)] {\node (a) {\(#1\)}; \draw[ultra thick,line cap=round,decorate,decoration={brace,amplitude=5pt}] (a.south east) -- node[below,inner sep=7pt] {\(\scriptstyle #2\)} (a.south west);}}
\makeatother

\usepackage[inline]{enumitem}%
\usepackage{subcaption}
\usepackage{url}
\usepackage{longtable}
\usepackage{hyperref}
\hypersetup{
	colorlinks=true, linktocpage=true, hidelinks
}

\usepackage{etoolbox}
\apptocmd{\sloppy}{\hbadness 10000\relax}{}{}

\usepackage{thmtools} %

\declaretheorem[
name = Theorem,
style = plain,
numberwithin = section,
]{theorem}

\declaretheorem[
name = Lemma,
style = plain,
numberwithin = section,
 sharenumber= theorem,
]{lemma}

\declaretheorem[
name = Corollary,
style = plain,
numberwithin = section,
sharenumber = theorem,
]{corollary}

\declaretheorem[
name = Definition,
style = plain,
numberwithin = section,
sharenumber = theorem,
]{definition}

\declaretheoremstyle[
  style=examplestyle,
  headfont=\normalfont\bfseries,
  numberwithin = section,
  sharenumber = theorem,
  bodyfont=\normalfont,
     qed = {\hbox{$\triangleleft$}}
]{examplestyle}

\declaretheorem[
  style=examplestyle,
  title=Example,
  numberwithin = section,
  sharenumber = theorem,
  refname={example,examples},
  Refname={Example,Examples},
]{example}

\declaretheorem[
  style=examplestyle,
  title=Remark,
  numberwithin = section,
 sharenumber = theorem,
  refname={remark, remarks},
  Refname={Remark, Remarks},
]{remark}

\numberwithin{equation}{section}

\makeatletter
\def\l@subsection{\@tocline{2}{0pt}{1pc}{4.6em}{}}
\renewcommand{\tocsubsection}[3]{%
  \indentlabel{\@ifnotempty{#2}{\hspace*{2.3em}\makebox[2.3em][l]{%
    \ignorespaces#1 #2.\hfill}}}#3}
\makeatother

\newcommand{\cng}{\ga}
\newcommand{\wts}{\mathscr{M}}
\newcommand{\wtse}{\mathscr{E}}
\newcommand{\bt}{\mathit{L}}
\newcommand{\ricform}{\rho}
\renewcommand{\div}{\operatorname{\delta}}
\newcommand{\symd}{\mathscr{S}}
\newcommand{\sraise}{\mathsf{X}_{+}}
\newcommand{\slower}{\mathsf{X}_{-}}
\newcommand{\sdeg}{\mathsf{Z}}
\newcommand{\kpc}{\square}

\newcommand{\cod}{\mathscr{C}}
\newcommand{\kcons}{\mathsf{K}}
\newcommand{\ccons}{\mathsf{L}}
\newcommand{\divcons}{\mathsf{D}}
\newcommand{\lie}{\mathfrak{L}}

\newcommand{\clie}{\mathscr{L}}
\newcommand{\klie}{\mathscr{K}}
\newcommand{\kliea}{\mathscr{K}^{\ast}}

\newcommand{\ilp}{\left(}
\newcommand{\irp}{\right)}%
\newcommand{\iln}{\lVert}
\newcommand{\irn}{\rVert}%
\newcommand{\cosmo}{\Lambda}
\newcommand{\stcod}{\mathscr{T}^{\text{Cod}}}
\newcommand{\stkill}{\mathscr{T}^{\text{Kill}}}
\newcommand{\stp}{\mathscr{T}^{+}}
\newcommand{\stm}{\mathscr{T}^{-}}
\newcommand{\stpm}{\mathscr{T}^{\pm}}
\def\op#1{\wideparen{#1}}

\newcommand{\stw}{\mathbb{W}}
\newcommand{\mcurv}{\mathscr{MC}}
\newcommand{\sphere}{\mathbb{S}}
\newcommand{\imt}{\iota}
\newcommand{\simt}{\mathsf{s}}
\newcommand{\quat}{\mathbb{H}}
\newcommand{\cayley}{\mathbb{O}}
\newcommand{\met}{\mathsf{h}}
\newcommand{\cubic}{\mathcal{C}}
\newcommand{\lich}{\diamond_{L}}
\newcommand{\culap}{\square}
\newcommand{\ih}{\mathsf{i}_{h}}
\newcommand{\tlie}{\mathscr{T}}
\newcommand{\sT}{\mathsf{T}}
\newcommand{\sN}{\mathsf{N}}

\renewcommand{\sprod}{\odot}
\newcommand{\sbl}{\mathsf{\si}}
\newcommand{\tf}{\operatorname{\mathsf{tf}}}
\newcommand{\symk}{S^{k}(\ctm)}
\newcommand{\precod}{\mathscr{C}}
\newcommand{\symkt}{S^{k}_{0}(\ctm)}
\newcommand{\symkpt}{S^{k+1}_{0}(\ctm)}
\newcommand{\symkmt}{S^{k-1}_{0}(\ctm)}
\newcommand{\symktv}{S^{k}_{0}(TM)}
\newcommand{\symkptv}{S^{k+1}_{0}(TM)}
\newcommand{\om}{\omega}
\newcommand{\mlt}{\circ}
\newcommand{\mt}{\mathcal{M}}

\def\dps#1{\begin{align}\begin{aligned}#1\end{aligned}\end{align}}

\newcommand{\Om}{\Omega}
\newcommand{\vol}{\operatorname{vol}}
\newcommand{\Pis}{\Pi^{\sharp}}
\newcommand{\tD}{\tilde{D}}
\newcommand{\ka}{\kappa}

\newcommand{\sR}{\mathscr{R}}
\newcommand{\sW}{\mathscr{W}}

\newcommand{\df}{\mathscr{D}}
\newcommand{\ef}{\mathscr{G}}
\newcommand{\efr}{\mathscr{G}}
\newcommand{\rictr}{\operatorname{\rho}}
\newcommand{\riem}{\mathit{Riem}}
\newcommand{\scal}{\mathit{S}}
\newcommand{\sric}{\mathit{Ric}}
\newcommand{\ric}{\operatorname{\mathit{Ric}}}
\newcommand{\statscal}{r}

\newcommand{\qR}{\mathscr{Q}_{\mathscr{R}}}
\newcommand{\qH}{\mathscr{Q}_{\mathscr{H}}}
\newcommand{\sH}{\mathscr{H}}
\newcommand{\qY}{\mathscr{Q}_{\mathscr{Y}}}

\newcommand{\sY}{\mathscr{Y}}

\newcommand{\T}{\mathcal{T}}
\newcommand{\Id}{\operatorname{Id}}
\newcommand{\sff}{\mathbb{\Pi}}
\newcommand{\dum}{\,\cdot\,\,}
\newcommand{\Ga}{\Gamma}

\newcommand{\lap}{\Delta}

\renewcommand{\j}{\mathsf{i}}
\newcommand{\la}{\lambda}
\newcommand{\ep}{\epsilon}
\newcommand{\ext}{\bigwedge}
\newcommand{\cinf}{C^{\infty}}
\newcommand{\ben}{[\Bar{\nabla}]}
\newcommand{\eno}{\operatorname{End}}
\newcommand{\si}{\sigma}
\newcommand{\pr}{\partial}
\newcommand{\ctm}{T^{\ast}M}
\newcommand{\aR}{R}
\newcommand{\bnabla}{\bar{\nabla}}

\def\mr#1{\tf(#1)}
\newcommand{\en}{[\nabla]}
\newcommand{\im}{\operatorname{Im}}
\newcommand{\aff}{\mathcal{A}}
\newcommand{\B}{\mathcal{B}}
\newcommand{\lb}{\langle}
\newcommand{\ra}{\rangle}
\newcommand{\ste}{\mathbb{V}}
\newcommand{\std}{\mathbb{V}^{\ast}}
\newcommand{\sted}{\mathbb{V}^{\ast}}
\newcommand{\A}{\mathcal{A}}
\newcommand{\al}{\alpha}
\newcommand{\be}{\beta}
\newcommand{\ga}{\gamma}
\newcommand{\frameb}{\mathcal{G}}
\newcommand{\W}{\mathscr{W}}
\newcommand{\hnabla}{\widehat{\nabla}}
\newcommand{\proj}{\mathbb{P}}
\newcommand{\g}{\mathfrak{g}}
\newcommand{\ad}{\text{ad}}
\newcommand{\tensor}{\otimes}
\newcommand{\rea}{\mathbb R}
\newcommand{\com}{\mathbb C}
\newcommand{\skmax}{\Lambda}
\newcommand{\skmin}{\lambda}

\newcommand{\tr}{\operatorname{\mathsf{tr}}}
\newcommand{\kwedge}{\owedge}

\let\oldtocsection=\tocsection
\let\oldtocsubsection=\tocsubsection
\renewcommand{\tocsection}[2]{\hspace{0em}\oldtocsection{#1}{#2}}
\renewcommand{\tocsubsection}[2]{\hspace{1em}\oldtocsubsection{#1}{#2}}

\begin{document}
\title{Curvature equations coupling symmetric tensors with a metric}

\author{Daniel J.~F. Fox}
\address{Departamento de Matemática Aplicada\\ Escuela Técnica Superior de Arquitectura\\ Universidad Politécnica de Madrid\\Av. Juan de Herrera 4 \\ 28040 Madrid España}
\email{daniel.fox@upm.es} 

\begin{abstract}
There are described hierarchies of equations coupling a metric with a trace-free tensor having prescribed symmetries and in the kernel of certain generalized gradients. These specialize, when the tensor vanishes identically, to the usual hierarchy of constant sectional curvature (projectively flat), Einstein, and constant scalar curvature. At the Ricci curvature level these equations are formal analogues of the Einstein-Maxwell and supergravity equations that couple differential forms with a metric. The particular cases coupling a metric with trace-free symmetric tensors satisfying the Codazzi or conformal Killing equations are studied in detail. Examples of solutions are obtained from mean curvature zero immersions, affine spheres, isoparametric hypersurfaces, and related algebraic constructions. The formalism yields a hierarchy of curvature equations for statistical structures. There are deduced constraints on the scalar curvature of the metric occurring in a solution that generalize classical results of Simons, for mean curvature zero hypersurfaces in spheres, and of Calabi, for hyperbolic affine spheres.
\end{abstract}

\maketitle

\setcounter{tocdepth}{2}  %

\begin{footnotesize}
\tableofcontents
\end{footnotesize}

\section{Introduction}\label{introsection}
The hierarchy of equations given by the conditions of constant sectional curvature, vanishing Einstein tensor, and constant scalar curvature is generalized to one coupling a metric with a tensor having prescribed symmetries and solving some partial differential conditions given by generalized gradient (Stein-Weiss) operators. The special case of trace-free symmetric tensors is studied in detail.

 Let $h$ be a pseudo-Riemannian metric and $\om$ a trace-free tensor with symmetries prescribed by some Young diagram. The hierarchy of equations \eqref{gs1}-\eqref{gs3} below for the pair $(h, \om)$ specializes to the usual constant curvature hierarchy when $\om \equiv 0$.  Each equation comprises two parts:
\begin{itemize}
\item A \emph{curvature equation} expressing the vanishing of a linear relation involving some part of the curvature of $h$ - such as the full curvature tensor, Einstein tensor, or scalar curvature - some expression quadratic in the auxiliary trace-free tensor $\om$ and some tensor built from $h$ itself;
\item Linear differential equations on the auxiliary tensor $\om$ given by the vanishing of some generalized gradients $\A_{\si}$, where $\si$ runs over some set $\wtse$ of weights of certain orthogonal group representations related with the symmetries of $\om$. The generalized gradients arise as orthogonal projections of $D\om$.
\end{itemize} 
The sign $\ep \in \{\pm 1\}$ is an auxiliary parameter, while $\ka$
denotes the (a priori not necessarily constant) function implicitly defined by each curvature equation below — Lemmas \ref{tracefreeconstantlemma} and \ref{projectivehiggslemma} show it is automatically constant under the stated auxiliary conditions, provided $n \geq 3$.
\begin{flalign}
\tag{GS1}\label{gs1} 
&\emph{Coupled projective flatness}:& &\A_{\si}(\om) = 0, \si \in \wtse,&&\riem - \ep \phi(\om, \om) = -\tfrac{\ka}{n(n-1)}h\kwedge h,  &
\\
\tag{GS2}\label{gs2}  
&\emph{Coupled Einstein equations}:& &\A_{\si}(\om) = 0,\si \in \wtse,&&
\begin{aligned}
&\ric - \ep \rictr(\phi(\om, \om))  = \tfrac{\ka}{n}h,
\,\,\text{or, equivalently,}\\
&\,\,\efr + \cosmo h = \ep \T(\om) ,\end{aligned} 
\\
\tag{GS3}\label{gs3} 
&\emph{Constraint equations}: & &\A_{\si}(\om) = 0, \si \in \wtse,&& \cosmo = \tfrac{n-2}{2n}\left(\scal + \ep \tfrac{n+2r}{r(n-2)}h(\om, \om)\right) \,\,\text{is constant}.
\end{flalign}
Here $\riem$, $\ric$, and $\scal$ are the Riemann, Ricci, and scalar curvatures, $\efr = \ric - \tfrac{1}{2}\scal h$ is the Einstein tensor, $\rictr$ is the Ricci trace of a tensor with metric curvature symmetries, $\phi$ is a bilinear map taking values in metric curvature tensors, and $r$ is an integer parameter determined by the symmetries of $\om$. 
Existence and behavior of solutions depend strongly on $\ep$.  Rescaling $\om$ by $c\neq 0$ rescales $\ka$ by $c^{2}$, which is positive. The specifications of the symmetries of $\om$, the map $\phi$, and the operators $\A_{\al}$ are all essential parts of the data. The scheme and the data constituting it is explained precisely in Section \ref{overviewsection}. Minimal conditions on such a scheme include that it be consistent and that it admit nontrivial solutions.

Special cases of such coupled equations at the full curvature level \eqref{gs1} include equations satisfied by the second fundamental forms of mean curvature zero submanifolds of space forms and mean curvature zero Lagrangian submanifolds of (para/pseudo)-Kähler space forms, while special cases of such coupled equations at the Ricci level \eqref{gs2} include the Einstein-Maxwell and supergravity equations.

The differential Bianchi identity imposes nontrivial consistency/integrability conditions at the level \eqref{gs1} and its trace does the same at the level \eqref{gs2}. These conditions are quadratic in $\om$. At the Ricci/Einstein level \eqref{gs2} there are many choices of data for which these integrability conditions are automatically satisfied in the sense that they are consequences of the vanishing of the specified generalized gradients, but at the full curvature level \eqref{gs1} such automatic consistency is rarer. For example, when $\om$ is a $2$-form and the generalized gradients are the exterior differential and codifferential, the Ricci/Einstein level is consistent with the standard Einstein-Maxwell coupling. However, while in this case the full curvature level does admit nontrivial solutions, the integrability condition on it derived from the differential Bianchi identity is not generally satisfied and is not an automatic consequence of the vanishing of these specific generalized gradients. On the other hand, as this example illustrates, even the case where only the level \eqref{gs2} is automatically consistent can be interesting.

Section \ref{overviewsection} describes the scheme for auxiliary tensors having symmetries determined by an arbitrary Young diagram. In this description some details that depend on the form of the Young diagram are omitted or left unchecked. %
The cases of exterior forms and trace-free symmetric tensors are the simplest, as they correspond to the extremal diagrams consisting of a single column or a single row. They are also the cases most directly interpretable geometrically. In particular, for symmetric tensors in certain degrees the corresponding equations have a clear geometric interpretation. This has the additional benefit of supplying, via nontrivial relations with known theorems, large classes of solutions to specific cases of the equations in contexts where direct solution may not appear tractable a priori. For example, trace-free cubic tensors are related with the study of hypersurfaces in flat affine space and an abundance of solutions to the proposed equations arise from existence theorems for affine spheres going back to work of Cheng-Yau solving an auxiliary Monge-Ampère equation.

Some effort is devoted in Section \ref{generalschemesection} to explaining that the general scheme has sense in that at least the datum $\phi$ generally is available. More generally it needs to be shown that there are choices of generalized gradients for which the kernels are nontrivial and moreover for which the equations of the general scheme admit solutions. For the coupled constraint equations this appears largely unproblematic and for the coupled Einstein equations one expects an abundance of solutions but this needs to be justified. However, for the coupled projectively flat equations it is not evident that solutions exist at all. In particular the consistency/integrability condition imposed by the differential Bianchi identity becomes clearly problematic in general. The bulk of the paper provides motivation for extending the project by presenting examples of solutions to these equations in the trace-free Codazzi setting, and describing some properties of solutions when they exist. Although the results presented are very far from complete, they indicate that in some particular cases solutions abound and admit interesting geometric interpretations.

The scheme is studied in detail in the special case where the coupled tensor is a trace-free symmetric $k$-tensor that is divergence-free and Codazzi, meaning its covariant derivative is trace-free and symmetric. This corresponds to the simplest Young diagrams that have only a single row of $k$ boxes. The resulting hierarch \eqref{projectivehiggsintro}-\eqref{constraintintro} is presented in Section \ref{couplingsection}.

The special case $k = 3$ yields a hierarchy of curvature equations for statistical structures and having links with the study of special hypersurfaces in flat affine space such as affine spheres that is described in Section \ref{statisticalsection} and is of independent interest. The equations corresponding with the similarly simple case of a Young diagram having only a single column of $k$ boxes, for which the auxiliary tensor is a $k$-form, are described as well, although they are not studied later in detail. In this case the generalized gradients are simply the exterior differential and codifferential.

For trace-free completely symmetric $k$-tensors and the generalized gradients which are the divergence and the trace-free Codazzi operator, the hierarchy is \emph{automatically consistent} in the sense that the integrability condition at the full curvature level is a consequence of the vanishing of the generalized gradients. This circumstance provides motivation for studying the full hierarchy in detail in this specific context.

For $k \in \{2, 3\}$ many solutions to the equations of this hierarchy are available from submanifold geometry. Solutions when $k = 2$ and $k = 3$ are obtained from mean curvature zero immersions in space forms and mean curvature zero Lagrangian immersions in (para)-Kähler space-forms. The case $k = 3$ also occurs in relation to the construction of affine spheres, and yields a full hierarchy for what are called statistical structures that includes as a special case the equations for affine spheres. On the one hand it is conceptually clarifying to see these problems in submanifold geometry as part of a general framework regarding projective flatness and Einstein type equations. On the other hand classical results about such submanifolds, such as growth estimates for the second fundamental form going back to Cheng and Yau and integral pinching conditions in the style of Simons, suggest qualitative properties of solutions in the more general setting. 

Section \ref{examplesection} describes various constructions of solutions to \eqref{stressenergyintro} and \eqref{projectivehiggsintro}.

Section \ref{geometricexamplesection} records examples of solutions of \eqref{stressenergyintro} and \eqref{projectivehiggsintro}. The $k = 2, 3$ special cases include the equations for mean curvature zero hypersurfaces in spheres, the equations for mean curvature zero Lagrangian submanifolds of (para/pseudo)-Kähler space forms, and the equations for affine spheres. The equations for a mean curvature zero nondegenerate immersion of a hypersurface in a pseudo-Riemannian space form and the equations for an affine sphere are special cases of the equations \eqref{projectivehiggsintro}, for a symmetric tensor $\om$ of rank $k = 3$. In these cases the tensor $\om$ is, respectively, the second fundamental form or the cubic form of the immersion. In both these contexts, a solution $(h, \om)$ to the equations \eqref{stressenergyintro} can be viewed as a more general geometric structure, not necessarily induced via an immersion. These examples show that, at least for $k \leq 3$, solutions to \eqref{projectivehiggsintro} abound. The case of affine spheres is particularly interesting as from it and results of Cheng-Yau it can be deduced that a properly convex flat real projective manifold carries a solution of the coupled projectively flat equations. This demonstrates that there are many solutions in Riemannian signature. %
On the other hand, this means that a direct solution of the $k = 3$ case of \eqref{projectivehiggsintro} would give an alternative manner of constructing affine spheres and the canonical Cheng-Yau metric on a properly convex flat real projective manifold. In fact this observation motivated the entire project described here, although its realization appears difficult.

In the $k = 3$ case the tensor $\om_{ij}\,^{k} = \om_{ijp}h^{kp}$ can be viewed as determining a structure of a commutative not necessarily associative algebra on each tangent space of $M$. In this case a solution $(h, w)$ of \eqref{projectivehiggsintro} can be seen as a curved generalization of a Frobenius manifold in the sense of \cite{Dubrovin, Manin-frobenius}. If the metric $h$ is flat and $k = 3$ then \eqref{projectivehiggsintro} becomes a projective version of the equations known as the \emph{associativity} or \emph{WDVV} equations \cite{Dubrovin-integrable, Dubrovin, Marshakov-Mironov-Morozov, Marshakov}, to which it specializes if moreover $\ka = 0$. 

The $k = 3$ case can be reformulated as yielding a hierarchy of curvature equations for statistical structures. The Ricci/Einstein level of this hierarchy was described in \cite[Section $9$]{Fox-conelike} in a form that goes back to \cite{Fox-ahs, Fox-2dahs, Fox-crm}, but it is useful to see this embedded as part of the full hierarchy. In this case the constant curvature level of the hierarchy is a generalization of the equations for affine spheres.

If $h$ is pseudo-Euclidean, so flat, the curvature terms in \eqref{stressenergyintro} vanish, and there remains a purely algebraic equation for the tensor $\om$. When $k = 3$, the complete polarization of a harmonic cubic polynomial solving the algebraic part of \eqref{stressenergyintro} determines the structure tensor of a metrized commutative algebra that is in general neither associative nor unital. On a flat Euclidean background, when $k = 3$, there is in each dimension $n \geq 2$ a unique solution of the coupled projectively flat equations \cite{Fox-crm, Fox-cubicpoly, Fox-simplicial}. On the other hand, in the case of completely symmetric trace-free $3$-tensors the coupled Einstein equations admit a rich abundance of solutions with flat background that correspond to a broad class of metrized commutative algebras that has been studied in \cite{Fox-ahs, Fox-crm, Fox-cubicpoly, Fox-simplicial, Fox-stsalgebra}. Related algebras are studied in \cite{Nadirashvili-Tkachev-Vladuts}. 

There are described also two algebraic constructions of solutions on a flat background that work for larger $k$. First, Theorem \ref{isoparametrictheorem} shows that all isoparametric polynomials yield solutions. Second, Theorem \ref{graphpolynomialtheorem} associates a solution with every $k$-regular graph for $k \geq 3$; this is related to a construction in the cubic ($k = 3$) case in \cite{Fox-cubicpoly, Fox-stsalgebra}.

Similarly, there can be sought left-invariant or biinvariant solutions on Lie groups or left-invariant solutions on reductive homogeneous spaces. In these cases it is unlikely that there are purely algebraic solutions of the coupled projectively flat equations, but the description of all solutions of the coupled Einstein equations on a compact simple real Lie group is a tractable problem in representation theory. The corresponding problem for other real forms of simple Lie groups should also be accessible, although more difficult. Theorem \ref{stressenergyexampletheorem} shows that on any compact simple Lie group solutions to \eqref{stressenergyintro} can be constructed from invariant polynomials on its Lie algebra. These solutions show that \eqref{stressenergyintro} admits nontrivial solutions in arbitrarily high dimensions for arbitrarily large $k$. 

Such results are motivated by and generalize results about immersed submanifolds that go back to Calabi \cite{Calabi}, in the context of affine spheres, Simons in the context of minimal immersions in spheres \cite{Simons}, and Cheng-Yau \cite{Cheng-Yau-maximalspacelike, Yau-cmcii} in the context of hypersurfaces of various kinds, as well as many others. 

The general pattern of such results is as follows. There is a Weitzenböck identity for the Laplacian of the squared-norm of a tensor satisfying some partial differential equation (here $\om$). Depending on the signs of some curvature terms there are two general classes of results. Refined Kato inequalities and sharp bounds on algebraic terms involving $\om$ yield a differential inequality that, via a sort of argumentation developed most prominently by Calabi and Cheng-Yau, yields an upper bound on $|\om|^2$ that can be interpreted as an upper bound on the scalar curvature of $h$ (working harder along the same lines one could obtain an upper bound on the Ricci curvature of $h$). The second follows arguments from Simons \cite{Simons}, and integrates the Weitzenböck formula to obtain integral bounds on $|\om|^{2}$. The results along these lines obtained here are reported in Section \ref{constraintsection} as Theorems \ref{scalarcurvaturetheorem} and \ref{simonstheorem}.

The proofs require bounds on $|\om \kwedge \om|^{2}$ that are presented in Section \ref{boundsection}. The bounds presented are sharp for $k\leq 3$ but almost certainly are not sharp for $k > 3$. 

The Weitzenböck identities are needed here for showing in full generality the equations \eqref{stressenergy} generalizing \eqref{stressenergyintro} that are considered here are well formulated, and for understanding when their hypotheses are nontrivial. This is described in Section \ref{couplingsection}. They are needed also to obtain the estimate leading to Theorem \ref{scalarcurvaturetheorem}. The proof of Theorem \ref{scalarcurvaturetheorem} uses the refined Kato inequalities for trace-free Codazzi and conformal Killing tensors described in Section \ref{katosection}. 

The general scheme \eqref{gs1}-\eqref{gs3} is described in a form so that it should be conceptually clear what form analogous Cheng-Yau and Simons type results might take for couplings with tensors having more general symmetry types. Most of the necessary machinery is available in full generality. In \cite{Hitchin-vanishing} Hitchin describes the machinery necessary for showing vanishing theorems for sections of kernels of generalized gradients. Related results in special cases are shown for Rarita-Schwinger operators in \cite{Branson-Hijazi, Homma-Semmelmann}. In \cite{Semmelmann-Weingart}, Semmelmann and Weingart describe the Weitzenböck machinery in the necessary generality. What is described here for symmetric tensors provides a roadmap for the more complicated general case. 

The main, interrelated objectives of the article can be summarized as follows:  
\begin{itemize}[leftmargin=11pt]
\item Describe a general scheme for pairs $(h, \om)$ comprising a pseudo-Riemannian metric and a tensor $\om$ having prescribed symmetries yielding a hierarchy of equations coupling the curvature of the metric with an expression quadratic in $\om$ and requiring that $\om$ satisfy some auxiliary partial differential equations expressed in terms of generalized Stein-Weiss operators.
\item Propose for statistical structures (and a sort of conformal generalization of statistical structures) a hierarchy of curvature equations generalizing the classical hierarchy of equations for a pseudo-Riemannian given by the conditions of constant sectional curvature, vanishing Einstein tensor, and constant scalar curvature.
\item Explain the relation between the previous two proposals and examine it in detail in the special case where $\om$ is a trace-free completely symmetric $k$-tensor and the auxiliary partial differential equations require that it be divergence-free and Codazzi. The $k = 3$ case recovers the hierarchy of curvature equations for statistical structures.
\item Provide examples of nontrivial solutions of the proposed systems.
\item For the case of Codazzi symmetric tensors, prove Cheng-Yau style growth estimates and Simons style integral pinching theorems. Independently of their own interest, these results, proved in Section \ref{constraintsection}, are illustrative of what can be expected in the general setting.
\end{itemize}

The main new contributions are the general scheme coupling a metric with a tensor of arbitrary Young type and its consistency conditions (Section \ref{overviewsection}); the hierarchies \eqref{projectivehiggsintro}-\eqref{constraintintro} for trace-free symmetric tensors and their correspondence with Einstein and projectively flat statistical structures (Sections \ref{couplingsection} and \ref{statisticalsection}); the families of solutions in Section \ref{examplesection}; and the Simons and Cheng-Yau type estimates of Section \ref{constraintsection}. The material on Codazzi and Killing tensors in Section \ref{differentialoperatorssection} and the vanishing theorems of Section \ref{katosection} are variations on classical or known material included for expository completeness.
   
\subsection{Notational and terminological conventions}\label{conventionssection}
Before proceeding it is useful to fix some notational and terminological conventions.
A \emph{pseudo-Euclidean} metric on a vector space $\ste$ means a pseudo-Riemannian metric whose Levi-Civita connection determines the flat affine structure on $\ste$. A \emph{metrized vector space} is a pair $(\ste, h)$, where $\ste$ is a finite-dimensional real vector space and $h\in S^{2}\sted$ is a pseudo-Euclidean metric.
A \emph{tensor module} means a $GL(\ste)$-submodule $\stw$ of $\tensor^{k}\std \tensor \tensor^{l}\ste$. Notations for tensors on vector spaces apply equally to sections of the corresponding vector bundles associated with the frame bundle. 
Indices are raised and lowered respecting horizontal position using $h_{ij}$ and the dual symmetric bivector $h^{ij}$ defined by $h^{ip}h_{pj} =\delta_{j}\,^{i}$. The abstract index conventions \cite[chapter $2$]{Penrose-Rindler} are used when helpful but indices are omitted when possible. With these conventions indices are labels indicating tensor type and symmetries and do not refer to any local frame. 
The metric on a tensor module is always that defined by complete contraction with $h_{ij}$, $h(\al, \be) = \lb \al, \be \ra = \al^{i_{1}\dots i_{k}}\be_{i_{1}\dots i_{k}}$, which generally differs from the metric induced by that on a tensor module $\stw$ by a constant factor that depends on the symmetries of the tensors considered. The notation $|\al|^{2}$ and $\lb \al, \be \ra$ mean $h(\al, \al)$ and $h(\al, \be)$, whatever the signature of $h$.

Given a vector space $\ste$ with dual $\std$, the tensor module
\begin{align}\label{mcurvdefined}
\mcurv(\std) = \{\sY_{ijkl} \in \tensor^{4}\std: \sY_{[ij]kl} = \sY_{ijkl} = \sY_{ij[kl]}, \sY_{[ijk]l} = 0\} \subset \tensor^{4}\std
\end{align}
comprises the covariant tensors $\sY_{ijkl}$ having \emph{metric curvature tensor type}. The \emph{Ricci trace} and \emph{scalar trace} are the linear maps $\rictr:\mcurv(\std) \to \tensor^{2}\std$ and $\scal: \mcurv(\std) \to \rea$ defined by $\rictr(\sY)_{ij} = \sY_{pij}\,^{p}$ and $\scal(\sY) = h^{ij}\rictr(\sY)_{ij}$. 

The curvature $\sR_{ijk}\,^{l}$ of a torsion-free connection $\nabla$ on the tangent bundle is defined by $2\nabla_{[i}\nabla_{j]}\om_{k} = -\sR_{ijk}\,^{p}\om_{p}$ for $\om_{i} \in \Ga(\ctm)$ and its \emph{Ricci} curvature is defined by $\sric_{ij} = \rictr(\sR)_{ij} = \sR_{pij}\,^{p}$. The Riemannian curvature tensor of a pseudo-Riemannian metric $h$ is obtained from the curvature of its Levi-Civita connection, $D$, by lowering the last index, $\riem_{ijkl} = \sR_{ijk}\,^{p}h_{pl} \in \Ga(\mcurv(\ctm))$. It takes values in the vector bundle $\mcurv(\ctm)$, and the Ricci and scalar curvatures of $h$ are defined by $\ric = \rictr(\riem)$ and $\scal = \scal(\riem)$. 

\subsection{Summary of notation}\label{notationsection}
Table \ref{notationtable} collects the notation used repeatedly and indicates where each symbol is defined. Here $n = \dim M$, unless stated otherwise, and $k$ is the rank of the symmetric tensor $\om$.

\begingroup
\small
\renewcommand{\arraystretch}{1.15}
\setlength{\LTleft}{0pt}
\setlength{\LTright}{0pt}
\newcommand{\nthead}[1]{\noalign{\smallskip}\multicolumn{3}{@{}l}{\textbf{#1}}\\\noalign{\smallskip}}
\begin{longtable}{@{}l@{\hspace{1em}}p{0.63\textwidth}@{\hspace{1em}}l@{}}
\multicolumn{3}{c}{\textsc{Table} \thetable.\enspace Summary of notation}\label{notationtable}\\\noalign{\medskip}
\hline\noalign{\smallskip}
\textbf{Symbol} & \textbf{Meaning} & \textbf{Reference}\\
\noalign{\smallskip}\hline
\endfirsthead
\hline\noalign{\smallskip}
\textbf{Symbol} & \textbf{Meaning} & \textbf{Reference}\\
\noalign{\smallskip}\hline
\endhead
\hline
\endfoot
\nthead{Metric and curvature}
$h$, $D$ & metric and its Levi-Civita connection & \S\ref{conventionssection}\\
$\lb \al, \be\ra$, $|\al|^{2}$ & complete contraction $\al^{i_{1}\dots i_{k}}\be_{i_{1}\dots i_{k}}$; $\lb \al, \al \ra$, whatever the signature of $h$ & \S\ref{conventionssection}\\
$\sR_{ijk}\,^{l}$ & curvature of a torsion-free connection, $2\nabla_{[i}\nabla_{j]}\om_{k} = -\sR_{ijk}\,^{p}\om_{p}$ & \S\ref{conventionssection}\\
$\riem$, $\ric$, $\scal$ & curvature tensor $\sR_{ijk}\,^{p}h_{pl}$, Ricci curvature $\sR_{pij}\,^{p}$, scalar curvature & \S\ref{conventionssection}\\
$\mcurv(\std)$, $\mcurv_{0}(\std)$ & tensors of metric curvature type; the trace-free ones & \eqref{mcurvdefined}\\
$\rictr(\sY)$, $\scal(\sY)$ & Ricci trace $\sY_{pij}\,^{p}$ and scalar trace of $\sY \in \mcurv(\std)$ & \S\ref{conventionssection}\\
$\df(\sY)$, $\ef(\sY)$, $\efr$ & divergence-free parts of $\sY$ and of $\rictr(\sY)$; Einstein tensor $\ef(\riem)$ & \eqref{divfreesy}\\
$\tf$, $\mr{\,\cdot\,}$ & orthogonal projection onto the completely trace-free part & \S\ref{generalschemesection}\\
$\sW$ & conformal Weyl tensor $\tf(\riem)$ & \eqref{tfweyl}\\
\nthead{Algebra of symmetric tensors}
$\symk$, $\symkt$ & symmetric and trace-free symmetric covariant $k$-tensors & \S\ref{algebrasection}\\
$\al \sprod \be$ & symmetrized tensor product & \S\ref{algebrasection}\\
$\imt(\al)\be$, $\imt(X)$ & contraction of $\al$ into the first indices of $\be$; interior multiplication & \eqref{imtdefined}\\
$\simt(X)$ & $\simt(X)\om = (k+1)X \sprod \om$ & \eqref{simtdefined}\\
$\tr$, $\met$ & trace $\imt(h)$; multiplication by $h$ & \eqref{trmetdefined}\\
$\sraise$, $\slower$, $\sdeg$ & $\mathfrak{sl}(2, \rea)$-triple built from $\met$, $\tr$, and the degree & \eqref{sl2ops}\\
$\al \kwedge \be$ & Kulkarni-Nomizu type product $S^{k}\std \times S^{k}\std \to \mcurv(\std)$ & \eqref{kwedgedefined}\\
$\al \cdot \be$ & curvature-valued product of exterior forms & \eqref{kformphi}\\
$\op{\sY}$ & curvature operator on $S(\std)$ determined by $\sY \in \mcurv(\std)$ & \eqref{syom}\\
$\qY(\om)$, $\qR(\om)$ & $\lb \om, \op{\sY}(\om)\ra$; the case $\sY = \sR$ & \S\ref{curvatureoperatorsection}\\
$\skmax_{n,k}$, $\skmin_{n,k}$ & extrema of $\tfrac{k-1}{2}|\om\kwedge\om|^{2} + |\rictr(\om\kwedge\om)|^{2}$ over $|\om| = 1$ & \eqref{skdefined}\\
\nthead{Differential operators}
$\A_{\si}$, $\wtse$ & generalized gradient; the set of generalized gradients imposed & \S\ref{generalschemesection}\\
$\sbl_{\A}(X)$ & symbol of the first order operator $\A$ at $X$ & \S\ref{generalschemesection}\\
$\div$ & divergence, $\div(\om)_{i_{1}\dots i_{k-1}} = D^{p}\om_{pi_{1}\dots i_{k-1}}$ & \eqref{divdefined}\\
$\cod$ & Codazzi operator, $\cod(\om)_{iji_{1}\dots i_{k-1}} = D_{[i}\om_{j]i_{1}\dots i_{k-1}}$ & \eqref{codazzidefined}\\
$\{\,\cdot\,, \,\cdot\,\}$, $\symd$ & Poisson bracket on $S(TM)$; symmetrized covariant derivative & \eqref{schoutendefined}\\
$\clie$ & conformal Killing operator & \eqref{cliedefined}\\
$\klie$, $\kliea$ & trace-free Codazzi operator $\tf \cod$; its formal adjoint & \eqref{kliedefined}\\
$\tlie$ & generalized gradient of Young type $(k, 1)$, equivalent to $\klie$ & \eqref{tliedefined}\\
$\ih$ & injection $\Ga(S^{k-1}_{0}\ctm) \to \Ga(\ctm\tensor\symkt)$ & \eqref{domkl}\\
$\precod^{k+1}(\std)$ & target module of $\klie$ & \eqref{precoddefined}\\
$\kcons$, $\ccons$, $\divcons$ & one-forms $\lb \om, \imt(X)\klie(\om)\ra$, $\lb \om, \imt(X)\clie(\om)\ra$, $\lb \imt(X)\om, \div(\om)\ra$ & \S\ref{couplingsection}\\
$\culap_{\al}$ & $\lap_{h} + \tfrac{\al}{k}\op{\sR}$ on $\Ga(\symkt)$ & \eqref{culapdefined}\\
\nthead{Coupled equations}
$\stp(\om)$, $\stm(\om)$ & stress-energy tensors associated with $\klie$ and $\clie$ & \eqref{stmpdefined}\\
$c$, $\ka$, $\ep$ & coupling constant; curvature function (constant if $n \geq 3$); sign & \S\ref{couplingsection}\\
\eqref{projectivehiggsintro}--\eqref{constraintintro} & coupled projectively flat, coupled Einstein, and constraint equations & \S\ref{couplingsection}\\
\eqref{gs1}--\eqref{gs3} & general scheme for tensors of arbitrary Young type & \S\ref{introsection}\\
\nthead{Statistical structures}
$(\nabla, h)$, $\bt$ & statistical structure; cubic form $\bt_{ijk} = \nabla_{i}h_{jk}$, $\nabla = D - \tfrac{1}{2}\bt_{ij}\,^{k}$ & \S\ref{statisticalsection}\\
$\bnabla$, $\ell$ & conjugate connection $\nabla + \bt_{ij}\,^{k}$; $\ell_{i} = \bt_{ip}\,^{p}$, zero if and only if special & \S\ref{statisticalsection}\\
$\nabla^{t}$, $R^{t}$ & $D - t\bt_{ij}\,^{k}$, with cubic form $2t\bt$; its curvature & \eqref{statt}\\
$R_{ij}$, $\bar{R}_{ij}$, $\statscal$ & Ricci curvatures of $\nabla$ and $\bnabla$; scalar curvature $h^{ij}R_{ij}$ & \S\ref{statisticalsection}\\
$S_{ij}$, $\bar{S}_{ij}$ & trace-free parts of the symmetrized Ricci tensors of $\nabla$, $\bnabla$ & \eqref{statisticalbianchi}\\
$B_{ijkl}$, $P_{ij}$ & projective Weyl and projective Schouten tensors of $\nabla$ & \eqref{stattproj}\\
\nthead{Submanifolds}
$\sff$, $\Pi$ & second fundamental form; its representative with respect to a transversal & \S\ref{geometricexamplesection}\\
$(G, A, \Om)$, $\hat{c}$ & (para-)Kähler structure; constant (para-)holomorphic curvature $4\hat{c}$ & \eqref{chsc}\\
$\Pi$, $\Pis$ & for a Lagrangian, $\Om(\hnabla_{X}Ti(Y), Ti(Z))$; $h(\Pis(X, Y), Z) = \Pi(X, Y, Z)$ & \eqref{pidefined}\\
$\sT$, $\sN$ & tangential and normal parts of the ambient curvature along a Lagrangian & \eqref{kic}\\
\nthead{Lie groups and flat examples}
$B$, $c_{ij}\,^{k}$ & Killing form, $h = -B$; structure constants of $\g$ & Thm.~\ref{stressenergyexampletheorem}\\
$X \mlt Y$, $\chi$, $L(X)$ & trace-free Jordan product on $\mathfrak{su}(m)$; $\chi = h(X\mlt Y, Z)$; $L(X)Y = X \mlt Y$ & \eqref{sunconnect}\\
$E$ & invariant quadratic polynomial determined by $h$; $E(x) = |x|^{2}$ on $\rea^{n}$ & \S\ref{examplesection}\\
$D^{(k)}F$ & $k$-fold covariant derivative of a function $F$ & Lem.~\ref{epequivalencelemma}\\
$P$, $g$, $m_{1}$, $m_{2}$ & Cartan-Münzner polynomial, its degree, and the multiplicities & \eqref{munznerequations}\\
\end{longtable}
\endgroup

\section{General scheme: hierarchies of curvature equations}\label{overviewsection}

This section describes in detail the general scheme and its specializations to the cases of completely antisymmetric and completely symmetric tensors.

\subsection{Classical hierarchy of curvature equations}
The study of a pseudo-Riemannian metric on an $n$-manifold can be framed in terms of the following hierarchy of increasingly weak curvature conditions.
\begin{enumerate}[label=(CCH\arabic*), ref=CCH\arabic*]
\item\label{cch1} \emph{Constant sectional curvature}: 
There is a smooth function $\ka$ such that
\begin{align}\label{constsec}
&\riem + \tfrac{\ka}{n(n-1)}(h\kwedge h) = 0, &
\end{align}
where $(h \kwedge h)_{ijkl} = 2h_{k[i}h_{j]l}$ is the Kulkarni-Nomizu square of $h$.
\item\label{cch2} \emph{Einstein equation}: 
There is a smooth function $\ka$ such that
\begin{align}\label{einstein}
&\sric -\tfrac{\ka}{n} h = 0.& 
\end{align}
\item\label{cch3} \emph{Constant scalar curvature}: The scalar curvature $\scal$ is constant. 
\end{enumerate}
Some observations about this hierarchy follow.
\begin{itemize}[leftmargin=11pt]
\item Tracing \eqref{constsec} shows that if $h$ solves \eqref{constsec} then it solves \eqref{einstein}.
By the traced differential Bianchi identity, if $h$ solves the Einstein equation \eqref{einstein}, then it has constant scalar curvature $\ka$, provided $n > 2$. On the other hand, when $n = 2$ the constancy of the scalar curvature is the basic equation of interest.
The same argument shows that when $n > 2$ the validity of either equation \eqref{constsec} or \eqref{einstein} for a smooth function $\ka$ already implies that $\ka$ must be constant, so it is sometimes better to think of these equations as requiring some curvature ($\riem$ or $\sric$) to be a multiple (a priori not necessarily constant) of the metric ($h\kwedge h$ or $h$). 

\item By the classical Beltrami theorem, the Levi-Civita connection of a pseudo-Riemannian metric is projectively flat if and only if the metric has constant sectional curvature. This means that the constant sectional curvature equation \eqref{constsec} also can be regarded as a condition of projective flatness.
\end{itemize}
 
For symmetric $2$-tensors $\al, \be \in \Ga(S^{2}\ctm)$, the symmetric, bilinear, curvature-tensor valued generalized Kulkarni-Nomiuz product $\kwedge$ is defined by 
\begin{align}
(\al\kwedge \be)_{ijkl} = \al_{k[i}\be_{j]l} - \al_{l[i}\be_{j]k}.
\end{align}
The \emph{divergence operator} $\div:\Ga(\tensor^{k}\ctm) \to \Ga(\tensor^{k-1}\ctm)$ is defined by
\begin{align}\label{divdefined}
&\div(\om)_{i_{1}\dots i_{k-1}} = D^{p}\om_{pi_{1}\dots i_{k-1}},
\end{align} 
and the \emph{Codazzi operator} $\cod$ on $\Ga(\symk)$ is defined by 
\begin{align}\label{codazzidefined}
\cod(\om)_{iji_{1}\dots i_{k-1}} = D_{[i}\om_{j]i_{1}\dots i_{k-1}}.
\end{align}
The hierarchy \eqref{cch1}-\eqref{cch3} is equivalent, except in specific dimensions, to a hierarchy of equations expressing the vanishing of certain divergence-free parts of curvature tensors. 
 Lemma \ref{divergencefreecurvaturelemma} shows that certain standard consequences of the differential Bianchi identity for a tensor having the symmetries of a metric curvature tensor depend only on this identity and not on whether the tensor is the curvature of a metric. It will also yield integrability conditions for the coupled equations discussed later.

\begin{lemma}\label{divergencefreecurvaturelemma}
On pseudo-Riemannian manifold $(M, h)$, a tensor $\sY \in \Ga(\mcurv(\ctm))$ \emph{satisfies the differential Bianchi identity} if $D_{[m}\sY_{ij]kl} = 0$. In this case there hold
\begin{align}\label{trdby}
&\div\sY= - 2\cod(\rictr(\sY)), &&2 \div\rictr(\sY) = d\scal(\sY),
\end{align}
and the tensors 
\begin{align}
\label{divfreesy}
&\df(\sY) = \sY + 2\rictr(\sY)\kwedge h - \tfrac{1}{2}\scal(\sY)h \kwedge h, &&\ef(\sY)  = \rictr(\sY) - \tfrac{1}{2}\scal(\sY)h,
\end{align}
are divergence free and satisfy
\begin{align}\label{trdf}
&\rictr(\df(\sY)) = (3-n)\ef(\sY), && 2\scal(\df(\sY)) =(n-2)(n-3)\scal(\sY).
\end{align}
\end{lemma}
\begin{proof}
Tracing the differential Bianchi identity in $ml$ yields the first equality of \eqref{trdby} and tracing again in $jk$ yields the second equality of \eqref{trdby}. The point is that it is not necessary that $\sY$ be the curvature tensor of $h$. Straightfoward calculations using \eqref{trdby} show that \eqref{divfreesy} are divergence free. Tracing \eqref{trdby} shows \eqref{trdf}.
\end{proof}

Both $\df$ and $\ef$ are linear operators for which the metric tensor is an eigenvector:
\begin{align}\label{dfefhh}
&\df(h\wedge h) = \tfrac{(n-2)(n-3)}{2}h\wedge h, && \ef(h) = \tfrac{(n-1)(n-2)}{2}h.
\end{align}
The divergence-free part of the Riemann tensor, $\df(\riem)$, figures in the definition of asymptotic mass in \cite{Ge-Wang-Wu, Ge-Wang-Wu-Xia} because $h(\df(\riem),\riem)$ is the Gauss-Bonnet curvature integrand in dimension $4$.

The classical hierarchy of curvature equations is equivalent to the following hierarchy, except for certain particular values of the dimension $n$, as indicated below: 
\begin{enumerate}[label=(DFH\arabic*), ref=DFH\arabic*]
\item\label{dfh1} \emph{Divergence-free curvature equation}: 
There is a smooth function $\ka$ such that 
\begin{align}\label{preeinsteinfield}
\df(\riem) + \tfrac{(n-2)(n-3)}{2n(n-1)}\ka (h \kwedge h) = 0.
\end{align}
\item\label{dfh2} \emph{Source-free Einstein field equations}: 
There is a smooth function $\ka$ such that 
\begin{align}\label{vacuumeinsteinfield}
\efr + \tfrac{n-2}{2n}\ka h = 0,
\end{align}
where $\efr = \ef(\riem) = \sric - \tfrac{1}{2}\scal h$ is the \emph{Einstein tensor}. 

\item\label{dfh3} \emph{Constant scalar curvature}: The scalar curvature $\scal$ is constant. 
\end{enumerate}

Some observations about this hierarchy follow.
\begin{itemize}[leftmargin=11pt]
\item The equation \eqref{vacuumeinsteinfield} is the source free Einstein field equation with cosmological constant $\cosmo = \tfrac{n-2}{2n}\ka$.

\item When $n = 2$, it follows from $\riem = (\scal(\riem)/2)h\kwedge h$ and $\rictr(\riem) = (\scal(\riem)/2)h$ that $\df(\riem) = 0$ and $\ef(\riem) =0$, so the equations \eqref{preeinsteinfield} and \eqref{vacuumeinsteinfield} are both vacuous in this case.

\item When $n = 3$, $\df(\riem)$ equals the conformal Weyl tensor (the trace-free part of the Riemann tensor), which vanishes when $n = 3$ by \cite[Theorem $5.7.A$]{Weyl}, so equation \eqref{preeinsteinfield} is vacuous in this case.

\item When $n > 3$, if $h$ solves \eqref{preeinsteinfield} then it solves the source-free Einstein field equations \eqref{vacuumeinsteinfield}, for, by \eqref{trdf}, taking the Ricci trace of \eqref{preeinsteinfield} shows that $0 = (n-3)(\ef(\riem) + \tfrac{n-2}{2n}\ka h)$.

\item When $n > 2$, tracing \eqref{vacuumeinsteinfield} shows that it implies that $\scal(\riem) = \ka$ is constant and there holds the Einstein equation \eqref{einstein}.

\item The constant sectional curvature equation \eqref{constsec} with constant $\ka$ implies the equation \eqref{preeinsteinfield} with the same constant $\ka$, while when $n > 3$ the converse holds. For all $n$, applying $\df$ to the constant sectional curvature equation \eqref{constsec} and using \eqref{dfefhh} shows that $h$ solving \eqref{constsec} solves \eqref{preeinsteinfield}. 
When $n > 3$, if $h$ solves \eqref{preeinsteinfield}, then it solves \eqref{constsec}, as can be seen as follows. Applying $\scal$ and $\rictr$ to \eqref{preeinsteinfield} and using \eqref{trdf} yield
$(n-2)(n-3)(\scal(\riem) - \ka) =0$ respectively $(n-3)(\ef(\riem) + \tfrac{n-2}{2n}\ka h) = 0$, so that $\scal = \ka$ and $\ric = (\ka/n)h$, provided $n > 3$. Using this and \eqref{divfreesy} to rewrite \eqref{preeinsteinfield} in terms of $\riem$ yields \eqref{constsec}.

\item When $n > 2$, either of the equations \eqref{einstein} and \eqref{vacuumeinsteinfield} implies the other with the same value of $\ka$ and moreover in either case the scalar curvature must be constant and equal to $\ka$: if there holds \eqref{einstein}, then, by \eqref{trdf}, taking its trace shows $(n-2)D_{i}\scal(\riem) = 0$, while if there holds \eqref{vacuumeinsteinfield}, takings its trace shows directly $(n-2)(\scal(\riem) - \ka) = 0$.
\end{itemize}

The two equivalent ways of writing the hierarchy of equations reflect different perspectives. The first, more geometric, asserts that some curvature tensor is a multiple of the metric or some tensor build from it, while the second, more physical, asserts that the divergence-free part of the same curvature tensor Ricci tensor is a multiple of the metric or some tensor built from it. The two formulations are equivalent except possibly in dimension $2$ or $3$.

The proposed hierarchies of equations \eqref{gs1}-\eqref{gs3} coupling a pseudo-Riemannian metric to an auxiliary tensor having prescribed symmetries specialize to \eqref{cch1}-\eqref{cch3} when the coupled tensor vanishes identically. From the divergence-free formulation of the original hierarchy it is clear that such a coupling must imply, as an integrability condition, some partial differential equations for the coupled tensor. In general such a condition is more general than (that is, weaker than) that imposed by requiring the vanishing of some generalized gradients, but such conditions are more tractable and understandable, if only because they are linear.

The forms of the couplings at the Ricci level of the hierarchy generalize the special case in which the auxiliary tensor is a $2$-form, regarded as the electromagnetic $2$-form, coupled with a metric as in the Einstein-Maxwell equations. In order to motivate what form the general hierarchy might take, it is convenient to describe first a hierarchy that includes the usual Einstein-Maxwell equations, before describing the general scheme that includes those mentioned as special cases.

The physical source free Einstein field equations are the special case of the more general \emph{Einstein field equations with energy momentum tensor $\T$ and cosmological constant $\cosmo$},
\begin{align}\label{einsteinT}
 & \efr + \cosmo h = \T,& 
\end{align}
in which $\efr = \ef(\riem)$ and the divergence-free symmetric two-tensor $\T$ is the \emph{stress-energy} tensor encoding the matter fields interacting with the gravitational field. In many cases, $\T$ is a quadratic expression in the components of some auxiliary tensor with given symmetries. 

For example, if $F_{ij}$ is an antisymmetric two-form, $\T(F) = -F\circ F - \tfrac{1}{4}|F|^{2}h$  is symmetric, where $(F\circ F)_{ij} = F_{i}\,^{p}F_{pj}$. 
It is divergence-free if $F$ is both closed and coclosed and in this case the source-free \emph{Einstein-Maxwell} equations are
\begin{align}\label{einsteinmaxwellT}
&dF = 0, && d\star F = 0,& &\efr + \cosmo h = \T, && \cosmo = \tfrac{2-n}{2n}\left(\scal - \tfrac{n-4}{2(n-2)}|F|^{2}\right).
\end{align}
or, in geometric form,
\begin{align}\label{einsteinmaxwell}
&dF = 0, && d\star F = 0,& &\sric - F\circ F  = \tfrac{1}{n}(\scal - |F|^{2})h. 
\end{align}
Although equivalent except when $n = 2$, the forms \eqref{einsteinmaxwellT} and \eqref{einsteinmaxwell} of the equations reflect respectively a physical or a geometric perspective. \emph{Source-free} refers to the vanishing of $dF$ and $d\star F$. Relaxations of these requirements will be discussed a bit further on in the context of the supergravity equations. Because the Einstein tensor is divergence free, that there hold $\efr + \cosmo h = \T$ implies some nonlinear partial differential equation for $F$ in which $F$ figures quadratically. The source-free Einstein-Maxwell equations and their relaxations described later are a priori stronger, because they require $F$ to satisfy linear or quasilinear partial differential equations.

It is reasonable to ask if there are equations at the full curvature level \eqref{gs1} analogous to \eqref{constsec} whose solutions necessarily solve \eqref{einsteinmaxwellT}. 

For $p \geq 2$, define a symmetric bilinear map $\cdot:\ext^{p}\std \times \ext^{p}\std \to \mcurv(\std)$ by
\begin{align}\label{kformphi}
(\al \cdot \be)_{ijkl} = \tfrac{1}{3}\left(\al_{k[i}\,^{A}\be_{j]lA} - \al_{l[i}\,^{A}\be_{j]kA} - \al_{ij}\,^{A}\be_{klA} - \be_{ij}\,^{A}\al_{klA}\right).
\end{align}
in which the single index $A$ abbreviates $a_{1}\dots a_{p-2}$ and its repetition indicates $(p-2)$fold contraction. There hold
\begin{align}
&\rictr(\al\cdot \be)_{ij} = \al_{(i}\,^{a_{1}\dots a_{p-1}}\be_{j)a_{1}\dots a_{p-1}},& &\tr_{h}\rictr(\al\cdot \be) = h(\al, \be).
\end{align}
Associate with $F \in \Ga(\ext^{p}\ctm)$ the \emph{stress-energy} tensor defined by
\begin{align}
\T(F) = \rictr(F\cdot F) - \tfrac{1}{2p}h(F, F)h.
\end{align}
In \cite{Baird}, Baird showed that, for $X \in \Ga(TM)$,
\begin{align}\label{bairdidentity}
\begin{aligned}
h(X, \div(\T)) &= h(\imt(X)F, \div F) - \tfrac{1}{p}h(\imt(X)dF, F) = h(\sbl_{\div}(X)(F), \div F) - \tfrac{1}{p}h(dF, \sbl_{d}(X)(F)).
\end{aligned}
\end{align}
where $\imt(X)$ is interior multiplication in the first slot, $\div$ is the divergence operator defined in \eqref{divdefined}, and $\sbl_{d}(X)(F)$ and $\sbl_{\div}(X)(F)$ denote the symbols of $d$ and $\div$ applied to $X$ and $F$.

By \eqref{bairdidentity}, if $F$ is closed and coclosed then $\T(F)$ is divergence free. More generally, $\div \T(F) = 0$ if each of $dF$ and $\div F$ is orthogonal to its symbol, $\sbl_{d}(X)(F)$ and $\sbl_{\div}(X)(F)$, for all $X$ (for an example of this nature see \eqref{supergravity} below). 

For $\ep \in \{\pm 1\}$, the following hierarchy of equations specializes to the Einstein-Maxwell equations when $\ep = 1$ and specializes to the usual constant curvature hierarchy when $F \equiv 0$:
\begin{flalign}
\tag{EM1}\label{coupledEMp} 
&\emph{Coupled projective flatness}:& &dF = 0, &&d\star F = 0,&&\riem - \ep (F\cdot F) = -\tfrac{\ka}{n(n-1)}h\kwedge h,  &
\\
\tag{EM2}\label{einsteinmaxwellp}  
&\emph{Coupled Einstein equations}:& &dF = 0, &&d\star F = 0,&&\efr + \cosmo h = \ep \T \iff 
\sric - \ep \rictr(F\cdot F)  = \tfrac{\ka}{n}h ,
\\
\tag{EM3}\label{scalarEMp} 
&\emph{Constraint equations}: & &dF = 0, &&d\star F = 0,& & \cosmo = \tfrac{n-2}{2n}\left(\scal - \ep \tfrac{n-2p}{p(n-2)}h(F, F)\right) \,\,\text{is constant}.
\end{flalign}
Because the coupling is quadratic in $F$, rescaling $F$ cannot change the sign $\ep$, which is the reason for including this additional parameter. 

The identity \eqref{bairdidentity} yields a weak consistency condition on the curvature equation at the Ricci/Einstein level \eqref{einsteinmaxwellp}. Whatever auxiliary equations are imposed on $F$ must also satisfy this condition. A similar tautological consistency condition can be deduced from the full curvature equation \eqref{coupledEMp}, which forces that $F\cdot F$ satisfy the differential Bianchi identity. This is not implied by the equations $dF = 0$ and $d\star F = 0$ and means that \eqref{coupledEMp} is sufficiently overdetermined that it is not self-evident that solutions with nonvanishing $F$ exist. However, if $F$ is parallel, then $F\cdot F$ satisfies the differential Bianchi identity and such solutions are exhibited below.  

For exterior $2$-forms there are three generalized gradients, $d$, $\div$ (which is a multiple of $\star d \star$), and $\T(F)_{ijk}$ equal to a multiple of $D_{[i}F_{j]k} - D_{[i}F_{jk]}$, as indicated schematically in \eqref{ext2grad}:
\begin{align}\label{ext2grad}
\begin{aligned}
\underset{D\om}{\underbrace{\ydiagram{1}\tensor\ydiagram{1,1}}} = \underset{\text{$dF$}}{\underbrace{\ydiagram{1, 1, 1}}} + \underset{\text{$\T(F)$}}{\underbrace{\ydiagram{2, 1}}} + \underset{\text{$\div(F)$}}{\underbrace{\ydiagram{1}}}
\end{aligned}
\end{align}
The vanishing of the generalized gradients $d$ and $d\star$ alone is insufficient to imply automatic consistency; there must vanish the third generalized gradient corresponding to the Young diagram of the partition $3 = 2 + 1$.

For Lorentzian signature $h$ the Einstein-Maxwell equations are physical and admit a great diversity of much studied solutions. When $h$ is Riemannian it is an observation due to \cite{Flaherty} in the scalar curvature zero case and \cite{Lebrun-einsteinmaxwell} in the general constant scalar curvature case that if $(M, h, J)$ is a $4$-dimensional Kähler manifold with Kähler form $\om(\dum, \dum) = h(J\dum, \dum)$ constant scalar curvature $\scal$ and Ricci-form $\ricform(\dum, \dum) = \sric(J\dum, \dum)$, then $(h, F)$ solves \eqref{einsteinmaxwell} where 
\begin{align}
F = \om + \tfrac{1}{2}(\ricform - \tfrac{1}{4}\scal \om)
\end{align}
is the sum of the Kähler form and half the primitive part of the Ricci form.

In a similar manner it can be shown that there exist nontrivial solutions to the more restrictive equations \eqref{coupledEMp}. By definition, a Kähler structure $(h, J)$ has constant holomorphic sectional curvature $4\ka$ if and only if
\begin{align}
&\riem - 3\ka (\om \cdot \om) = -\ka (h \kwedge h).& %
\end{align}
In this case $\ka = \scal/24$ and:
\begin{itemize}
\item If $\ka > 0$ (positive constant holomorphic sectional curvature) then, for $F = \sqrt{3\ka}\om$, $(h, F)$ solves the coupled projectively flat equations \eqref{coupledEMp} with $\ep = 1$;
\item If $\ka < 0$ (negative constant holomorphic sectional curvature) then, for $F = \sqrt{-3\ka}\om$, $(h, F)$ solves the coupled projectively flat equations \eqref{coupledEMp} with $\ep = -1$.
\end{itemize}
Although it is not the objective here to discuss in detail the equations for $p$-forms \eqref{coupledEMp}-\eqref{scalarEMp}, it would be interesting to know what are all the solutions of \eqref{coupledEMp} or \eqref{einsteinmaxwellp}, in particular whether in Riemannian signature when $p = 2$ there are any solutions to \eqref{einsteinmaxwellp} besides the Flaherty-LeBrun solutions.

As mentioned, in some cases the stress-energy tensor remains divergence-free under a relaxation of the coclosed condition. For example, if $n+1 = 3p$ and $c$ is a constant the equations
\begin{align}\label{supergravity}
&dF = 0, &&d\star F = c F\wedge F,&&\efr + \cosmo h = \ep \T(F) , \iff 
\sric - \ep \rictr(F\cdot F)  = \tfrac{\ka}{n}h ,
\end{align}
imply that $\T$ is divergence free by \eqref{bairdidentity}. The vanishing of the term $h(\imt(X)F, \div(F))$ reduces to $\imt(X)(F \wedge F \wedge F) = 0$, which is true because $3p > n$. Replacing $d\star F = 0$ by $d \star F = c F \wedge F$ in \eqref{coupledEMp} and \eqref{scalarEMp} yields a hierarchy of equations generalizing the hierarchy \eqref{coupledEMp}-\eqref{scalarEMp} and specializing to it when $c = 0$.

The equations \eqref{supergravity} are a variant of the usual \emph{supergravity equations} \cite{Alekseevsky-Chrysikos-Taghavi-Chabert, Cremmer-Julia-Scherk, Figueroa-O'Farrill-Santi, Santi, Figueroa-O'Farrill-Santi-sasakian}. The cited references use different conventions for the norms and exterior products than those used here, and such differences are not geometrically or physically meaningful, but given such choices fixing a value of $c$ fixes a scaling that is consequential. More precisely, for a particular choice of the constant $c$ the equations \eqref{supergravity} in the case $p = 4$ recover the $n = 11$ supergravity equations as in \cite{Figueroa-O'Farrill-Santi, Santi}. This particular value of $c$ is determined by the requirement of invariance with respect to the Poincare superalgebra; precisely it is necessary for the cocycle condition in \cite[Proposition $7$]{Figueroa-O'Farrill-Santi}. On the other hand, in the discussion surrounding the proof of \cite[Theorem $4.1$]{Figueroa-O'Farrill-Santi-sasakian} it is shown that certain $11$-dimensional Lorentzian Sasakian manifolds with $F$ an appropriate multiple of the pullback of the square of the Kähler form on the base of the characteristic fibration yield a solution of \eqref{supergravity}, although for a particular constant $c$ in general having numerical value and sign different from those of the usual supergravity equations as in \cite{Figueroa-O'Farrill-Santi}. More generally still, \cite{Figueroa-O'Farrill-Santi-sasakian} considers the modified equation $d \star F - c F \wedge F = p$ in which $p$ is a closed $8$-form built from the first and second Pontryagin forms of a Lorentzian Sasakian manifold. Under a technical hypothesis on the base of the characteristic fibration the form $p$ is a multiple of $F \wedge F$, and together \cite[Proposition $4.3$ and Theorem $4.5$]{Figueroa-O'Farrill-Santi-sasakian} yield more solutions of \eqref{supergravity}.

In the just mentioned generalized supergravity setting, as for the Einstein-Maxwell hierarchy or for the hierarchy for trace-free Codazzi tensors described below, the scaling of the auxiliary tensor coupled to the metric can be fixed up to the sign here denoted by $\ep$ (although in a setting where the auxiliary tensor has some physical meaning such a scaling may already be regarded as inadmissible), but this sign depends on underlying geometric or physical considerations, and properties and existence of solutions depend on it in essential ways.

\subsection{Ingredients of the general scheme}\label{generalschemesection}
At each level \eqref{gs1}-\eqref{gs3} the general scheme consists of a curvature equation and a set of generalized gradients.

The generalized gradients arise as orthogonal projections of $D\om$.
Those that appear (the set $\wtse$) cannot be chosen freely because consistency of the couplings forces that they depend on the choice of the quadratic expression $\Phi(\om, \om)$. Formally, taking the derivative of $\Phi(\om, \om)$ yields 
\begin{align}\label{formalgrad}
D\Phi(\om, \om) = 2\Phi(\om, D\om) = 2 \Phi(\om, \sum_{\si}\A_{\si}(\om))
\end{align}
and the curvature equation together with the differential Bianchi identity forces the vanishing of some expression derived from this formal sum, in which there will figure nontrivially certain of the generalized gradient operators. Generally there will be considered a hierarchy of curvature equations that can be deemed full curvature, Einstein, and constraint equations, each obtained formally from the preceding by taking traces. The consistency conditions for the auxiliary equations on $\om$ arise as tautological consequences of the differential Bianchi identities applied to the curvature equations. For example, at the Einstein level the consistency equations derive from the vanishing of the divergence of the Einstein tensor of $h$, and will take a form generalizing the identity \eqref{bairdidentity}. Which generalized gradients can appear may be different depending on the level in the hierarchy. Consistency is easier for the Einstein and constraint equations than for the full curvature equations - as at these levels the consistency equations depend only on the traced differential Bianchi identities - and correspondingly a greater variety of auxiliary constraints are possible at these levels. On the other hand, this remark suggests that the most interesting choices are those that yield consistency at all levels. In these cases the fundamental equation is that at the top level, expressing the projective Weyl tensor of the metric in terms of some quadratic expression $\phi(\om, \om)$ and the other equations are derived from it by taking traces. The main focus of this article is such a hierarchy for trace-free Codazzi symmetric tensors. On the other hand, for trace-free Killing tensors, as for exterior differential forms, consistency at the top level is not guaranteed simply by the vanishing of certain generalized gradients, although a consistent formulation at the Einstein level is available. %

For simplicity the general scheme is described for a Riemannian $n$-manifold $M$ with $O(n)$-frame bundle $\frameb \to M$. With minor modifications most of the discussion can be adapted for oriented manifolds or spin manifolds, in the latter cases replacing $O(n)$ by $SO(n)$ or $Spin(n)$. Likewise it could be adapted to pseudo-Riemannian manifolds, replacing $O(n)$ by an orthogonal group of appropriate signature. In these cases some of the representation theory used in the definite signature case becomes more complicated, with corresponding complications for the general scheme described.

The following background is repeated from Branson's \cite{Branson-steinweiss} where it is explained lucidly.
The associated bundle $(\tau, \ste)$ of the standard representation of $O(n)$ is $T^{\ast}M = \frameb \times_{\tau}\ste$. Because each irreducible $O(n)$-module $(\rho, \stw)$ appears as a summand of some tensor power of the standard representation, the associated bundle $\W(\rho) = \frameb\times_{\rho}\stw$ is a subbundle of some tensor bundle whose sections are trace-free tensors. (Were $M$ oriented and were $O(n)$ replaced by $SO(n)$, in some specific cases, e.g. forms of middle dimension, it would be a subbundle of such a bundle.) The covariant derivative $D$ maps a section $\om$ of $\W(\rho)$ to a section of $T^{\ast}M\tensor \W(\rho)$, which is the bundle associated to the tensor product $O(n)$-module $(\tau \tensor \rho, \ste \tensor \stw)$. This tensor product decomposes into a direct sum $\oplus_{\si \in \wts(\rho)} \W(\si)$ of vector bundles associated with irreducible $O(n)$-modules each appearing with multiplicity one (see \cite{Branson-steinweiss}; this follows from the Pieri rule). The orthogonal projection $\A_{\si}(\om)$ of $D\om$ on one of the summands is a \emph{generalized gradient} or \emph{Stein-Weiss operator} in the sense of \cite{Branson-steinweiss, Calderbank-Gauduchon-Herzlich, Stein-Weiss}. The generalized gradients corresponding to a given Young diagram are indexed by the Young diagrams obtained from the given one by adding to it a single box or deleting from it a single box in such a way as to obtain a new Young diagram. (While this is correct for $O(n)$-modules, for $SO(n)$-modules there are further possibilities because on a manifold of even dimension $n$, the module of exterior differential forms of rank $n/2$ splits as a sum of self-dual and anti-self-dual forms.)

In order to analyze what conditions can be imposed in terms of generalized gradients, it is necessary to examine the possibilities for the quadratic map $\Phi$ appearing in \eqref{formalgrad}. This is possible separately at the different levels in the hierarchy. Here it will be studied at the most complicated top level, the Ricci curvature level being similar but simpler.

Suppose $\phi:\stw \times \stw \to \mcurv(\std)$ is a symmetric bilinear $O(n)$-module map. Because $\tr_{h} \rictr(\phi(\om, \om))$ is an $O(n)$-invariant quadratic form on the irreducible $O(n)$-module $\stw$, it must be constant multiple of $|\om|^{2}$. It is convenient to normalize $\phi$ by requiring that $\tr_{h} \rictr(\phi(\om, \om)) = h(\om, \om)$ for all $\om \in \Ga(\W)$. 

Given a pseudo-Euclidean metric $h$ on a tensor module $\stw$, $\stw_{0}$ denotes the $O(h)$-submodule comprising elements of $\stw$ for which there vanish all traces and contractions with $h_{ij}$ and $h^{ij}$. The convention here is that $\ste_{0}$ equals $\ste$. For any tensor module $\stw$ with determined symmetries and valency, $\tf \in \eno(\stw)$ denotes the $h$-orthogonal projection onto the submodule $\stw_{0} \subset \stw$ comprising completely trace-free tensors of the same symmetries and valency. 

The choice of $\phi$ is an important piece of data for the equations to be described. In general the equations will be most interesting when the trace-free part $\tf \phi(\om, \om)$ need not vanish, so yields a nontrivial map $\tf \phi: \stw \times \stw \to \mcurv_{0}(\std)$ into the irreducible (if $n > 4$) $O(n)$-module $\mcurv_{0}(\std) = \mcurv(\std)\cap \ker \rictr$ of Weyl curvature tensors (trace-free curvature tensors). To know that there are nontrivial possibilities it needs to be that the dimension of the space of $O(n)$-equivariant symmetric bilinear maps $\stw \times \stw \to \mcurv_{0}(\std)$ is positive (at the Einstein level one would consider instead $O(n)$-equivariant symmetric bilinear maps $\stw \times \stw \to S^{2}_{0}(\std)$). The problem of calculating this dimension is a special case of the problem of calculating what are called the triple multiplicities for the orthogonal group $O(n)$, namely the dimensions of the invariant subspaces of threefold tensor products of irreducible $O(n)$-modules. There is a huge literature on this and related problems, for example \cite{Berenstein-Zelevinsky, Littelmann-pathsroots}, which yields methods in principle adequate for calculating the dimension in the specific case of interest here, but which remain complicated. As even in the case of successful realization their development would require considerable space, and the objective here is only to argue that the data $\phi$ appearing in the equations described later is in general nontrivial, only some special cases are described. On the other hand, to implement the scheme there is needed a basis of the space of such equivariant maps, not only the dimension of the space, and it is not clear how effective is the available general machinery in this regard.

In general, if $\om_{i_{1}\dots, i_{k}}$ has some specified symmetries, contracting it with itself on some subset of $k-2$ indices and applying a linear combination of these two maps yields a map $\phi$ of the desired sort, although proving its nontriviality may require some work. A precise statement follows. Let $\ste$ be the standard representation of $O(n)$ and let $\stw \subset \tensor^{k}\std$ be a nontrivial irreducible $O(n)$-module. If $k \geq 2$, there is a nontrivial $O(n)$-equivariant symmetric bilinear map $\phi:\stw \times \stw \to \mcurv(\std)$. The elements of $\stw$ can be identified with covariant tensors $\om_{i_{1}\dots i_{k}}$, trace-free with respect to the $O(n)$ invariant metric $h$ on $\std$, having the symmetries encoded by a filling by the indices $i_{1}, \dots, i_{k}$ of a Young diagram with nonincreasing row lengths given by a partition of $k$. The tensors of $\stw$ are antisymmetric in the indices filling any given column of the Young diagram, while there vanishes their antisymmetrization over a set of indices comprising the indices in a column and any index in a box to the right of the given column. If $\stw$ corresponds to a Young diagram with more than one row, then, after an isomorphism given by reordering the indices, it can be supposed that its constituent tensors are antisymmetric in their first two indices, $\om_{iji_{1}\dots i_{k-2}} = \om_{[ij]i_{1}\dots i_{k-2}}$. The product
\begin{align}\label{alcdotbedefined}
(\al \cdot \be)_{ijkl} = \tfrac{1}{3}\left(\al_{k[i}\,^{A}\be_{j]lA} - \al_{l[i}\,^{A}\be_{j]kA} - \al_{ij}\,^{A}\be_{klA} - \be_{ij}\,^{A}\al_{klA}\right),
\end{align}
in which $A$ abbreviates $a_{1}\dots a_{k-2}$, satisfies $\rictr(\al \cdot \be)_{ij} = \al_{(i}\,^{p_{1}\dots p_{k-1}}\be_{j)p_{1}\dots p_{k-1}}$ and $\tr_{h}\rictr(\al \cdot \be) = h(\al, \be)$. By the nondegeneracy of $h$ this shows that if $\al \neq 0$ there exists $\be \in \stw$ such that $\al \cdot \be \neq 0$. It is not obvious, and this argument does not show, that $\tf \phi$ is nontrivial, but in many cases this is in fact true provided $k \geq 2$ and $n \geq 4$. 

The preceding argument omits the case most relevant here, $\stw = S^{k}_{0}\std$. For $k \geq 2$, define a symmetric bilinear \emph{generalized Kulkarni-Nomizu product} $\kwedge:S^{k}(\std) \times S^{k}(\std) \to \mcurv(\std)$ by 
\begin{align}\label{kwedgedefined}
(\al \kwedge \be)_{ijkl} = \al_{k[i}\,^{A}\be_{j]lA} - \al_{l[i}\,^{A}\be_{j]kA}.
\end{align}
in which $A$ abbreviates $a_{1}\dots a_{k-2}$ as in \eqref{alcdotbedefined}. (When $k = 2$ the map $\kwedge$ is half what is usually called the Kulkarni-Nomizu product.)
It is immediate from the definition that $(\al \kwedge \be)_{[ijk]l} = 0$, so that $\al \kwedge \be \in \mcurv(\std)$. There hold
\begin{align}\label{rictralbe}
\begin{aligned}
\rictr(\al \kwedge \be)_{ij} &= \al_{(i}\,^{p_{1}\dots p_{k-1}}\be_{j)p_{1}\dots p_{k-1}}
- \tfrac{1}{2}\al_{ij}\,^{p_{1}\dots p_{k-2}}\tr(\be)_{p_{1}\dots p_{k-2}} - \tfrac{1}{2}\be_{ij}\,^{p_{1}\dots p_{k-2}}\tr(\al)_{p_{1}\dots p_{k-2}},\\
\scal(\al \kwedge \be) & = \lb \al, \be \ra - \lb \tr\al, \tr \be \ra,
\end{aligned}
\end{align}
which yield the corresponding quadratic expressions needed for the lower levels of the hierarchy.
If $k \geq 2$, the restricted product $\kwedge: S^{k}_{0}\std \times S^{k}_{0}\std \to \mcurv(\std)$ can be taken as $\phi$ and by \eqref{rictralbe} satisfies the normalization $\tr_{h} \rictr(\om \kwedge \om) = h(\om, \om)$.
Lemmas \ref{tfkwedgenontriviallemma} and \ref{kwedgelemma} show that if $\dim \std \geq 4$, then the orthogonal projection of $\om \kwedge \om$ onto the space of trace-free curvature tensors is nontrivial and any such $O(n)$-equivariant product is built from this map and its traces. These lemmas give an alternative and more concrete proof of the more abstract Lemma \ref{tensormultlemma} which, while far from an optimal statement, serves to illustrate what needs to be proved in general and how it might be proved.

\begin{lemma}\label{tensormultlemma} 
Let $(\ste, h)$ be an $n$-dimensional Euclidean vector space. Let $\stw_{\la} \subset \tensor^{k}\std$ be an irreducible $O(n)$-module of $h$-trace-free tensors with symmetries corresponding with a filling of a Young diagram whose row lengths are given by an integer partition $\la = (\la_{1}^{m_{1}} \dots \la_{r}^{m_{r}})$ of length $r = |\la|$. If $2 \leq k$ and $r \leq n/4$, there is a nontrivial $O(n)$-equivariant symmetric bilinear map $\phi: \stw_{\la}\times \stw_{\la}\to \mcurv_{0}(\std)$, and such a map is unique up to scale in the case of a Young diagram of one row corresponding with $\stw_{(k)} = S^{k}_{0}\std$.
\end{lemma}

\begin{proof}
By \cite[Theorem $2.5.3$]{Koike-Terada} or \cite[Section A.$6$]{Littelmann-littlewoodrichardson} (which refer to $SO(n, \com)$ modules) and that standard relation between irreducibles of a simple complex Lie group and its compact form, the multiplicity of $\mcurv_{0}(\std)$ as a summand in the decomposition of $\stw_{\la}\times \stw_{\la}$ into $O(n)$-irreducibles is given by the Newell-Littlewood number $N_{\la, \la, (2, 2)}$ in the \emph{stable range} where $2r = 2|\la| \leq n/2$. For partitions $\mu, \nu, \tau$ of $n$, the Newell-Littlewood number $N_{\mu, \nu, \tau}$ is defined by 
\begin{align}\label{nlformula}
N_{\mu, \nu, \tau} = \sum_{\al, \be, \ga}c_{\al, \be}^{\mu}c_{\be, \ga}^{\nu}c_{\ga, \al}^{\tau}, 
\end{align}
where the sum runs over partitions of $n$ and $c_{\al, \be}^{\mu}$ is the usual Littlewood-Richardson number (the structure constants for the ring of symmetric functions with respect to the Schur basis) \cite{Gao-Gidon-Yong, Gao-Gidon-Ressayre-Yong}. In private correspondence, Alexander Yong explained to me a proof, using the formula \eqref{nlformula}, the nonnegativity of the Littlewood-Richardson numbers, properties of skew-Schur functions, and the usual Pieri rule, that $N_{\nu, \nu, (2, 2)} \geq 1$ if and only if $\nu \neq (1)$ (the partition with one box) with equality if and only if $\nu$ is a row or a column. Applied in the current setting this shows that $N_{\la, \la, (2, 2)} \geq 1$ with equality if and only if $\la = (k)$ or $\la = (1^{k})$. Because $|(k)| = 1$, this gives the desired conclusion when $|\la| \leq n/4$. By \cite[Theorem $2.5.3$]{Koike-Terada} this restriction is sufficient to conclude that the desired multiplicities are given by the Newell-Richardson numbers. 
\end{proof}

There is a large literature calculating the triple multiplicities for threefold tensor products of irreducibles for classical groups as in \cite{Berenstein-Zelevinsky}. Outside the stable range there are formulas that in principle serve to calculate the desired multiplicities available in \cite{Koike-Terada, Littelmann-littlewoodrichardson, Littelmann-pathsroots, Newell} and other sources, but developing them here sufficiently to prove the desired claim would require too much space. The basic issue is that outside the stable range partitions do not uniquely determine $O(n)$-irreducibles - some columns are longer than $n/2$, and applying the Hodge star over the antisymmetric indices in a such columns yields via the diagram folding procedure described precisely in \cite{Koike-Terada} an equivalent irreducible. Such equivalencies can result in hidden cancellations that complicate calculating multiplicities. In indefinite signature there are additional complications generated by nonpositivity. Nonetheless, it seems likely that the claim of Lemma \ref{tensormultlemma} remains true with only very minor conditions provided the diagram of the partition $\la$ contains at least two boxes.

Qualifications aside, the preceding serves on the one hand to support that nontrivial choices of $\phi$ exist quite generally, and on the other hand to show that for $\stw = S^{k}_{0}\std$, the normalization $\tr_{h} \rictr(\om \kwedge \om) = h(\om, \om)$ determines the trace-free part of $\phi$ uniquely as the trace-free part of $\om \kwedge \om$.

To write down the hierarchy of equations to be considered, there are needed, in addition to $\phi$:
\begin{itemize}
\item a set $\wtse \subset \wts(\rho)$ parametrizing the generalized gradients and depending on $(\rho, \stw)$ and $\phi$; and
\item a stress-energy tensor associated with $\stw$ and $\phi$ and having the form
\begin{align}
\T(\om) = \rictr(\phi(\om, \om)) + \tfrac{1}{2r}h(\om, \om)h,
\end{align}
where the nonzero constant $r$ is determined by $(\rho, \stw)$ and the choice of $\phi$ via the requirement that the divergence of $\T(\om)$ vanish. 
\end{itemize}
The sign and value of the nonzero integer $r$ depend on the number of rows and columns in the Young diagram corresponding with $\stw$ and also the set $\wtse$ specifying summands in the decomposition of $\ste \tensor \stw$ into irreducibles. Initially it can seem surprising that for $\om$ having a particular symmetry type there may be more than one choice for $\T(\om)$, but this occurs already for trace-free symmetric tensors, for which there are two possible stress-energy tensors, one corresponding with the Codazzi operator, the other with the Killing operator. See \eqref{stmpdefined} for the precise values of $r$ when $\stw = S^{k}_{0}\std$ and $\phi = \kwedge$ for the two corresponding sets $\wtse$.

The hierarchy of equations \eqref{gs1}-\eqref{gs3} specializes to the usual constant curvature hierarchy when $\om \equiv 0$. Each equation comprises two parts, a \emph{curvature equation} and equations on the auxiliary tensor $\om$.
The hierarchy is written so that formally tracing the curvature equation at each level yields the curvature equation at the subsequent level. This means that any choice at one level fixes the choices at lower levels. It is possible as well to consider the two-step truncated hierarchy beginning at the Einstein level; in this case the data that needs to be chosen is not $\phi$ but an $O(n)$-equivariant symmetric bilinear map $\stw \times \stw \to S^{2}_{0}(\std)$. 

In each equation of the hierarchy, that $(h, \om)$ satisfy the curvature equation need not imply that $\om$ satisfy the equations $\A_{\si}(\om) = 0$, $\si \in \wtse$, but in general it does imply some weaker equations, in the Einstein case via an identity like \eqref{generalizedbaird} derived from the condition $\div \T(\om) = 0$ and in the full curvature case via the differential Bianchi identity for the Levi-Civita connection. This means that the consistency of the curvature equations and the equations for $\om$ is not automatic and forces some conditions on the set $\wtse$. The form of such conditions is most easily seen at the Einstein level.

At the full curvature level, the curvature equation in \eqref{gs1} implies that $\phi(\om, \om)$ is in the kernel of the differential Bianchi operator. Because $\phi(\om, \om)$ is quadratic in $\om$ and the differential Bianchi operator is first order, working out what it means for the former to be in the kernel of the latter yields via \eqref{formalgrad} the vanishing of a sum of generalized gradients of $\om$ paired in various ways with $\om$.

As shown by Lemma \ref{consistencylemma} below, in the particular case of completely symmetric trace-free tensors and the Codazzi type generalized gradients the differential Bianchi identity follows from the vanishing of the generalized gradients so that the coupled equations are automatically consistent even at the level of the full curvature tensor. This automatic consistency is a distinguishing feature of the completely symmetric case.%

Likewise, for the consistency of the coupled Einstein equations, there must be a set $\wtse \subset \wts(\rho)$, depending on $(\rho, \stw)$ and possibly $\phi$, such that, for some constants $\be_{\si} \in \rea$ (that also depend on $(\rho, \stw)$),
\begin{align}\label{generalizedbaird}
&h(X, \div \T(\om)) =\sum_{\si \in \wtse} \be_{\si} h(\sbl_{\A_{\si}}(X)(\om), \A_{\si}(\om)).
\end{align}
in which $\sbl_{\A}(X)(\om)$ is the symbol of the first-order differential operator $\A$ in the direction $X \in \Ga(TM)$.

As is apparent from \eqref{generalizedbaird}, for \eqref{gs1} and \eqref{gs2} to imply \eqref{gs3} it is not necessary that $\A_{\si}(\om)$ vanish for $\si \in \wtse$, rather only that some linear combination of $h(\sbl_{\A_{\si}}(X)(\om), \A_{\si}(\om))$ vanish for $\si \in \wtse$. The previously mentioned example of the supergravity equations is an interesting example in which this extra generality applies. With such a condition in place of the vanishing of the $\A_{\si}(\om)$ the auxiliary equations of the hierarchy will be said to be \emph{relaxed}.

All of the preceding can be formulated in more generality. First, the orthogonal group can be replaced by the orthogonal group of a metric of indefinite signature. Second, it can be supposed that $M$ has a spin structure and there can be used sections of the bundles associated with irreducible $Spin(n)$-modules and the associated generalized gradients (these include, for example, the Rarita-Schwinger operators). Third, equations of the form \eqref{gs1}-\eqref{gs3} can be considered in which in place of $\ep\phi(\om, \om)$ there appears a sum of terms of this same form but for tensors of various symmetry types. 

Apart from consistency considerations there are also solvability considerations. For a general metric $h$, there may be no $\om$ solving the equations $\A_{\si}(\om) = 0, \si \in \wtse$. For differential forms, with $\A$ being $d$ or $d^{\ast}$, the exterior differential and its adjoint, the space of solutions is described by Hodge theory. Although for other symmetries there is no comparably powerful general theory yielding existence there are general vanishing theorems that describe obstructions to solvability. Such results at least suggest where to look for nontrivial existence results.

When the metric of a solution is Riemannian signature, strong results on the behavior of nontrivial solutions will be available when the sum of the Laplacians $\A_{\mu}^{\ast}\A_{\mu}$ as $\A_{\mu}$ runs over the generalized gradients \emph{not} occurring in $\wtse$ is elliptic. In this case qualitative properties of solutions can be deduced via the Bochner/Weitzenböck/Kato machinery that goes back to Calabi and Yau in particular. In Lorentzian signature the same calculations are generally valid, but the resulting partial differential equations are hyperbolic and their behavior and solvability has an entirely different character that, although worth exploring, escapes the purview of the current paper.

The Bochner Laplacian $D^{\ast}D$ on $\Ga(\W(\rho))$ is a sum of Laplacian operators of the form $\A^{\ast}_{\si}\A_{\si}$ as $\si$ runs over all the set $\wts(\rho)$ associated with $(\rho, \stw)$. Which sums of operators of the form $\A^{\ast}_{\si}\A_{\si}$ are elliptic is described by Branson in \cite[Theorem $1.4$]{Branson-steinweiss}.

In the cases where this general hierarchy described above has been seen to be interesting, the complementary operator $\sum_{\si \in \wts(\rho)\setminus \wtse}\A^{\ast}_{\si}\A_{\si} = D^{\ast}D - \oplus_{\si\in \wtse}\A^{\ast}_{\si}\A_{\si}$ is elliptic. This means that a solution of \eqref{gs1} satisfies an elliptic second order PDE with zeroth order terms quadratic in the curvature tensor of the metric. This is exactly a Weitzenböck formula in the sense of Semmelmann-Weingart \cite{Semmelmann-Weingart}, where such Weitzenböck formulas are studied in the generality necessary for this general scheme. The equations \eqref{gs1} means that the zeroth order terms quadratic in the curvature tensor are reexpressible in terms of $\phi$ as an expression quartic in $\om$. Depending on the sign $\ep$ and the value of $\cosmo$, such a formula can either be integrated, yielding Simons style integral estimates, or it can be bounded using refined Kato inequalities in the sense of \cite{Branson-kato, Calderbank-Gauduchon-Herzlich} to yield a differential inequality that can be used to obtain bounds on the growth of $\om$ in the style of Calabi-Cheng-Yau.

A special feature of the exterior differential on differential forms is its diffeomorphism equivariance. This feature fails for other choices of generalized gradients.

For the general scheme, the coupled projective flatness equations formally resemble the usual Hitchin equations for Higgs fields. The conditions $\A_{\si}(\om) = 0, \si \in \wtse$ are something like holomorphicity of the auxiliary field, while the equation $\riem - \ep \phi(\om, \om) = -\tfrac{\ka}{n(n-1)}h\kwedge h$ says that a quadratic modification of the curvature determined by the auxiliary field is a multiple of the metric (the identity). In the special case of trace-free symmetric tensors on surfaces, that the symmetric $k$-tensor $\om$ be trace-free and Codazzi means it is the real part of a holomorphic $k$-differential \cite[Lemma $3.5$]{Fox-2dahs}, and the coupled projective flatness equations (which collapse to the coupled Einstein equations on a surface) are in fact a special case of the abelian Higgs equations, as is explained (using somewhat different terminology) in \cite[Section $8$]{Fox-2dahs}.

\section{Completely symmetric tensors}
This section describes the particularization of the general scheme of Section \ref{generalschemesection} for trace-free symmetric tensors. The simplest instantiations of the general scheme occur when the number of generalized gradients is minimal, namely in the cases corresponding to antisymmetric tensor and symmetric tensors (Young diagrams with one column or row). The first case leads to the generalized Einstein-Maxwell hierarchy for $p$-forms already described. The second case leads to hierarchies for conformal Killing and trace-free Codazzi tensors. This is presented in Section \ref{couplingsection} after some background material is described in Sections \ref{algebrasection} and \ref{differentialoperatorssection}.

There is a large literature related to conformal Killing and Codazzi tensors scattered over many years and contexts but still not well assimilated into the general literature. The material in Sections \ref{algebrasection} and \ref{differentialoperatorssection} is presented to record what is needed here in a notation and form adapted to the use made of it in later sections. 

\subsection{Algebra of symmetric tensors and curvature operators}\label{algebrasection}
Fix a metrized vector space $(\ste, h)$. The graded vector space $S(\sted) = \oplus_{k \geq 0}S^{k}(\sted)$ of finite linear combinations of completely symmetric covariant tensors on $\ste$ is an algebra graded by degree with the product given by the symmetrized tensor product $(\al \sprod \be)_{i_{1}\dots i_{k+l}} = \al_{(i_{1}\dots i_{k}}\be_{i_{k+1}\dots i_{k+l})}$ for $\al \in S^{k}(\sted)$ and $\be \in S^{l}(\sted)$. Using $h$, $S(\ste)$ and $S(\std)$ are identified by index raising and lowering, where for $X \in \Ga(S^{k}\ste)$, $X^{\flat}_{i_{1}\dots i_{k}} = X^{j_{1}\dots j_{k}}h_{i_{1}j_{1}}\dots h_{i_{k}j_{k}} \in \Ga(S^{k}\std)$, and, for $\om \in \Ga(S^{k}\std)$, $\om^{\sharp} \in \Ga(S^{k}\ste)$ is defined dually.

Define a bilinear pairing $\imt:S(\sted)\times S(\sted) \to S(\sted)$ by
\begin{align}\label{imtdefined}
\imt(\al)\be = \al_{i_{1}\dots i_{k}}\be^{i_{1}\dots i_{k}}\,_{i_{k+1}\dots i_{l}},
\end{align}
for $\al \in S^{k}\std$ and $\be \in S^{l}\std$ with $k \leq l$.
For $X \in \std$ the degree $-1$ and $1$ operators $\imt(X)$ and $\simt(X) \in \eno(S(\sted))$ defined by 
\begin{align}\label{simtdefined}
\simt(X)\om = (k+1)X\sprod \om
\end{align}
for $\om \in S^{k}\std$ satisfy the relation 
\begin{align}\label{imtsimt0}
[\imt(X), \simt(Y)] = h(X, Y).
\end{align}
Graded linear operators $\tr, \met \in \eno(S(\sted))$ of degrees $-2$ and $2$ are defined by 
\begin{align}\label{trmetdefined}
&\tr(\om) =  \imt(h)\om,& 
& \met(\om) = h \sprod \om , && \om \in S^{k}(\sted).
\end{align}
By convention, $\tr(\om) = 0$ when $k = 1$. These operators are adjoints, meaning that $\lb \al, \met(\be) \ra = \lb \tr(\al), \be\ra$ for $\al \in S^{k}(\ste)$ and $\be \in S^{k-2}(\ste)$. In particular, the orthogonal complement in $S(\sted)$ of the image of $\met$ is the subspace $S_{0}(\sted) =\oplus_{k \geq 0}S^{k}_{0}\std$ of trace-free elements of $S(\sted)$, where, by convention, $S^{1}_{0}\std = \std$. 

\begin{lemma}\label{sl2lemma}
The degree $2$, $-2$, and $0$ graded linear operators $\sraise, \slower, \sdeg \in \eno(S(\sted))$ defined for $\om \in S^{k}(\sted)$ by
\begin{align}\label{sl2ops}
\sraise(\om) = \tbinom{k+2}{2}h\sprod \om, & &\slower(\om) = -\tfrac{1}{2}\tr \om, & &\sdeg(\om) = \tfrac{n+2k}{2}\om, %
\end{align}
satisfy
\begin{align}\label{sl2rels}
&[\sraise, \slower] = \sdeg, &&[\sdeg, \sraise] = 2\sraise, & &[\sdeg, \slower] = -2\slower,
\end{align}
so generate an action of $\mathfrak{sl}(2, \rea)$ on $S(\sted)$. 
The operators $\imt$ and $\simt$ satisfy the additional relations
\begin{align}
\label{imtsimt}&[\sraise, \imt(X)] = -2\simt(X),& &[\slower, \simt(X)] = -\imt(X),& &[\slower, \imt(X)] = 0,& &[\sraise, \simt(X)] = 0.&
\end{align}
\end{lemma}
\begin{proof}
The identities \eqref{sl2rels} and \eqref{imtsimt} follows straightforwardly from  \eqref{imtdefined}-\eqref{trmetdefined}.
\end{proof}

\begin{lemma}
Let $(\ste, h)$ be a Euclidean vector space of dimension $n \geq 4$. Suppose $k \geq 3$. For $\al \in S^{k-2}\std$ and $\be \in S^{k}\std$,
\begin{align}\label{slkw}
\begin{aligned}
\sraise(\al)\kwedge \be + (k-2)(k-3)\al \kwedge \slower(\be) &= h \kwedge \imt(\al)\be,\\
\rictr(\sraise(\al)\kwedge \be) + (k-2)(k-3)\rictr(\al \kwedge \slower(\be)) &= \tfrac{2 - n}{2}\imt(\al)\be + (n-2) \lb\al, \slower(\be)\ra h,\\
\scal(\sraise(\al)\kwedge \be) + (k-2)(k-3)\scal(\al \kwedge \slower(\be)) &=2(n-1) \lb\al, \slower(\be)\ra.
\end{aligned}
\end{align}
(The term $\al \kwedge \slower(\be)$ is zero if $k = 3$.)
\end{lemma}

\begin{proof}
A straightforward computation shows
\begin{align}\label{kw1}
2\tbinom{k}{2}(h\sprod \al)_{kiA} \be_{jl}\,^{A}=\tbinom{k-2}{2}\al_{kiB}\tr(\be)_{jl}\,^{B} + (k-2)(\al_{kB}\be_{ijl}\,^{B} + \al_{iB}\be_{jkl}\,^{B}) + h_{ki}\al_{B}\be_{jl}\,^{B},
\end{align}
in which the uppercase Latin letters $A$ and $B$ denote the contracted multiindices. Antisymmetrizing \eqref{kw1} in $ij$ and in $kl$ yields
\begin{align}\label{kw2}
\tbinom{k}{2}(h\sprod \al)\kwedge \be = \tbinom{k-2}{2}\al \kwedge \tr(\be) + h \kwedge \imt(\al)\be.
\end{align}
Rewriting \eqref{kw2} yields the first equation of \eqref{slkw}. For $\si \in S^{2}(\std)$ there hold 
\begin{align}\label{rictrsih}
&\rictr(\si \kwedge h) = \tfrac{2 - n}{2}\left(\si + \tfrac{1}{n-2}\tr(\si)h\right), &&\scal(\si \kwedge h) = (1-n)\tr \si.
\end{align}
Applied to the first equation of \eqref{slkw} with $\si = \imt(\al)\be$, the identities \eqref{rictrsih} yield the remaining equations of \eqref{slkw}.
\end{proof}

\begin{corollary}
Suppose $k \geq 3$. For $\al \in S^{k-2}\std$ and $\be \in S^{k}_{0}\std$,
\begin{align}\label{tfkw}
\tf(\sraise(\al)\kwedge \be) = \sraise(\al)\kwedge \be - h \kwedge \imt(\al)\be.
\end{align}
\end{corollary}
\begin{proof}
From \eqref{rictrsih} it follows that the trace-free part $\tf(\sY)$ of $\sY \in \mcurv(\std)$ is given by 
\begin{align}\label{tfweyl}
\begin{aligned}
\tf(\sY)&= \sY + \tfrac{2}{n-2}\rictr(\sY)\kwedge h - \tfrac{1}{(n-2)(n-1)}\scal(\sY)h\kwedge h = \sY + \tfrac{2}{n-2}\mr{\rictr(\sY)}\kwedge h + \tfrac{1}{n(n-1)}\scal(\sY)h\kwedge h.
\end{aligned}
\end{align}

Because $\slower(\be) = 0$, combining \eqref{tfweyl} and \eqref{slkw} yields \eqref{tfkw}.
\end{proof}

\begin{lemma}\label{tfkwedgenontriviallemma}
Let $(\ste, h)$ be a Euclidean vector space of dimension $n$. 
\begin{enumerate}
\item\label{tfk1} If $n \geq 2$, for all $k \geq 2$ there are $\al, \be \in S^{k}_{0}\std$ such that $\tf\rictr(\al \kwedge \be) \in S^{2}_{0}\std$ is nonzero.
\item\label{tfk2} If $n \geq 4$, for all $k \geq 2$ there are $\al, \be \in S^{k}_{0}\std$ such that $\tf(\al \kwedge \be) \in \mcurv_{0}(\std) = \mcurv(\std)\cap \ker \rictr$ is nonzero.
\end{enumerate}
\end{lemma}
\begin{proof}
Choose nonzero $X \in \std$ and nonzero $\si \in S^{k-1}_{0}\std$ such that $\imt(X)\si = 0$ (for example, $\si$ can be taken as the extension to $\std$ of a trace-free symmetric $(k-1)$-tensor on the orthocomplement of $X$). By \eqref{imtsimt}, $\simt(X)\si \in S^{k}_{0}\std$. Because $\imt(X)\si = 0$, there holds $\rictr((\simt(X)\si)\kwedge(\simt(X)\si)) = |\si|^{2}X \tensor X$, which has nonzero trace-free part. This shows \eqref{tfk1}.

For $X, Y, Z, W \in \std$ there hold
\begin{align}\label{xyzw}
\begin{aligned}
8&(X\sprod Y)\kwedge(Z \sprod W) \\
&= (Y\wedge Z) \tensor (X\wedge W) + (X \wedge W)\tensor (Y \wedge Z) +  (Y\wedge W) \tensor (X\wedge Z) + (X \wedge Z)\tensor (Y \wedge W),\\
4&\rictr((X\sprod Y)\kwedge(Z \sprod W) )  = - 2h(X, Y)Z\sprod W - 2h(Z, W)X\sprod Y\\
&\quad h(X, Z)Y\sprod W + h(X, W)Y\sprod Z + h(Y, Z)X\sprod W + h(Y, W)X\sprod Z.
\end{aligned}
\end{align}
Supose $\{X, Y, Z, W\}$ is an orthonormal set (this requires $n \geq 4$). By \eqref{xyzw}, $(X\sprod Y)\kwedge(Z \sprod W)$ is nonzero and is contained in $\mcurv_{0}(\std)$. If $k = 2$ there is nothing more to show, so suppose $k > 2$. Because $n \geq 4$ there is nonzero $\si \in S^{k}_{0}\std$ such that $\imt(X)\si = 0 = \imt(Y)\si$ (extend any trace-free symmetric tensor on the orthocomplement of the span of $X$ and $Y$). Because $\imt(X)\si = 0 = \imt(Y)\si$, 
\begin{align}
\tbinom{k+2}{2}^{2}(X\sprod Y\sprod \si)\kwedge (Z\sprod W \sprod \si)= |\si|^{2}(X\sprod Y)\kwedge(Z \sprod W)
\end{align}
is a nonzero element of $\mcurv_{0}(\std)$. However, $Z\sprod W \sprod \si$ need not be trace-free.
The tensors $\al = \tbinom{k+2}{2}\tf(Z\sprod W \sprod \si)$ and $\be = \tbinom{k+2}{2} X\sprod Y\sprod \si$ are in $S^{k}_{0}\std$. There is $\mu \in S^{k-2}\std$ such that $Z\sprod W \sprod \si = \al + \sraise(\mu)$. 
By \eqref{tfkw},
\begin{align}
\begin{aligned}
|\si|^{2}(X\sprod Y)\kwedge(Z \sprod W)& = \tf \left(|\si|^{2}(X\sprod Y)\kwedge(Z \sprod W) \right)  = \tf\left((\al + \sraise(\mu))\kwedge \be \right)\\
&= \tf(\al \kwedge \be)+ \tf\left(h \kwedge \imt(\mu)\be\right)= \tf(\al \kwedge \be),
\end{aligned}
\end{align}
the last equality by \eqref{tfweyl}. This shows that $\tf(\al \kwedge \be)$ is nonzero, so shows \eqref{tfk2}.
\end{proof}

\begin{lemma}\label{kwedgelemma}
Let $(\ste, h)$ be a Euclidean vector space of dimension $n \geq 4$. Let $k \geq 2$. For any $a, b, c\in \rea$ not all zero, the $O(n)$-equivariant bilinear map $\phi:\symkt \times \symkt \to \mcurv(\std)$ defined by
\begin{align}\label{abcmap}
\phi(\al, \be) = a\tf(\al \kwedge \be) + b (\tf \rictr(\al \kwedge \be) )\kwedge h+ c h(\al, \be) h \kwedge h,
\end{align}
is nontrivial and every nontrivial such map has this form for some $a, b, c \in \rea$ not all zero.
\end{lemma}
\begin{proof}
By \cite[Lemma $4.3$]{Bettiol-Mendes} the $O(n)$-module $S^{2}(\symkt)$ has exactly one irreducible factor isomorphic to each of $\mcurv_{0}(\std)$, $S^{2}_{0}(\std)$, and $\rea$. Since a map of the form \eqref{abcmap} is $O(n)$-equivariant, it suffices to prove that such a map is nontrivial provided at least one of $a$, $b$, and $c$ is nonzero. For nonzero $a$ or $b$ this follows from Lemma \ref{tfkwedgenontriviallemma}, while for nonzero $c$ it is simply the nondegeneracy of $h$.
\end{proof}

Lemma \ref{kwedgelemma} justifies that for symmetric tensors there is a nontrivial coupling $\phi$ as in \eqref{gs1}. Lemma \ref{consistencylemma} shows that there is a choice of such a coupling that is automatically consistent.

\subsection{Generalized gradients for symmetric tensors: trace-free Codazzi and conformal Killing tensors}\label{differentialoperatorssection}

This section describes the generalized gradients for trace-free completely symmetric tensors. The operators $\clie$ and $\klie$ and their Weitzenböck identities were introduced by the author in \cite[section $6$]{Fox-ahs} in 2009; the resulting vanishing theorem was established there only in the rank two case, via \cite{Berger-Ebin}, and, for all ranks but only on surfaces, in \cite[section $3$]{Fox-2dahs}. The corresponding vanishing theorem for all ranks $k$ on manifolds of every dimension was obtained for compact manifolds by \cite{Stepanov}, extending the rank two result of \cite{Berger-Ebin}, and for complete manifolds by \cite{Shandra-Stepanov-Mikes} in 2019; see Corollary \ref{bochnercorollary} and Theorem \ref{completecodazzitheorem} below, the latter of which slightly sharpens the complete case using the refined Kato inequality of Section \ref{katosection}.
This material has since been treated independently, in various equivalent forms, in \cite{Hadfield, Heil-Moroianu-Semmelmann, Stepanov-Tsyganok}; see the survey \cite{Stepanov-Aleksandrova-Tsyganok}.

For $\symkt$ there are three generalized gradients, $\clie$, $\tlie$, and $\div$ as indicated schematically in \eqref{gengrad} for the case $k = 4$.
\begin{align}\label{gengrad}
\begin{aligned}
\underset{D\om}{\underbrace{\ydiagram{1}\tensor\ydiagram{4}}} = \underset{\text{$\clie(\om)$}}{\underbrace{\ydiagram{5}}} + \underset{\text{$\tlie(\om)$}}{\underbrace{\ydiagram{4, 1}}} + \underset{\text{$\div(\om)$}}{\underbrace{\ydiagram{3}}}
\end{aligned}
\end{align}
These generalized gradients are defined in what follows.

Because the fibers of $TM$ and $\ctm$ carry canonically dual flat centroaffine structures, constructions, such as those of Sectin \ref{algebrasection}, applicable to dual vector spaces $\ste$ and $\std$ apply fiberwise to $TM$ and $\ctm$ without change. 
The graded vector space $S(TM)= \oplus_{k\geq 0}\Ga(S^{k}(TM))$ (respectively $S_{0}(TM)= \oplus_{k\geq 0}\Ga(S^{k}_{0}(TM)))$ comprising finite linear combinations of (trace-free) completely symmetric tensors is a graded Poisson algebra with the fiberwise symmetrized tensor product $\sprod$ and the Poisson bracket defined by 
\begin{align}\label{schoutendefined}
\{X, Y\}^{i_{1}\dots i_{k+l-1}} = k X^{p(i_{1}\dots i_{k-1}}\nabla_{p}Y^{i_{k}\dots i_{k+l-1})} - l Y^{p(i_{1}\dots i_{l-1}}\nabla_{p}X^{i_{l}\dots i_{k+l-1})},
\end{align}
for $X \in \Ga(S^{k}(TM))$ and $Y \in \Ga(S^{l}(TM))$ and any torsion-free affine connection $\nabla$. Since index raising and lowering using $h$ induce commutative algebra isomorphisms $\flat: S(TM) \to S(\ctm)$ and $\sharp:S(\ctm) \to S(TM)$, there results on $S(\ctm)$ the Poisson bracket $\{\al, \be\}$ defined by $\{\al, \be\} = \{\al^{\sharp}, \be^{\sharp}\}^{\flat}$. 

Given a metric $h$, the connection in \eqref{schoutendefined} may be taken to be its Levi-Civita connection $D$ and a degree one operator $\symd: S(\ctm) \to S(\ctm)$ is defined as a multiple of the symmetrized covariant derivative by
\begin{align}
&\symd(\om) _{i_{1}\dots i_{k+1}}= \tfrac{k+1}{2}\{h, \om\}_{i_{1}\dots i_{k+1}} = (k+1)D_{(i_{1}}\om_{i_{2}\dots i_{k+1})}, &&\om \in \Ga(S^{k}(\ctm)).
\end{align}
The formal adjoint of $\om \to -\tfrac{1}{2}\{h, \om\}$ is the divergence operator $\div:S(\ctm) \to S(\ctm)$ defined by \eqref{divdefined}.
Straightforward computations show $\symd$ and $\div$ span a two-dimensional irreducible $\mathfrak{sl}(2, \rea)$-module, meaning
\begin{align}
&[\sraise, \symd] = 0,&&[\slower, \symd] = -\div,&&[\sraise, \div] = -\symd, &&[\slower, \div] = 0,& &[\sdeg, \symd] = \symd,& &[\sdeg, \div] = -\div.
\end{align}
A \emph{rank $k$ Killing tensor} is a symmetric tensor $X \in \Ga(S^{k}(TM))$ such that $\{h, X^{\flat}\} = 0$; equivalently, $\symd(X) = 0$.

Recall that $\symkt \subset \symk$ is the subbundle of trace-free tensors. Let $\tf:\Ga(\symk) \to \Ga(\symkt)$ be the $h$-orthogonal projection onto the completely trace-free part and define the degree one \emph{conformal Killing} operator $\clie: S_{0}(\ctm) \to S_{0}(\ctm)$ by 
\begin{align}
\label{cliedefined}
&\begin{aligned}
\clie(\om) &= \tfrac{1}{2}\tf \{h, \om\}= \tfrac{1}{2}\{h, \om\} - \tfrac{k}{n + 2(k-1)} h\sprod \div(\om) = \tfrac{1}{k+1}\left(\symd - \tfrac{2}{n+2(k-1)}\sraise \div\right)(\om), \end{aligned}.%
\end{align}
More explicitly, the conformal Killing operator $\clie$ is given by
\begin{align}\label{cliedefined0}
&\clie(\om)_{i_{1}\dots i_{k+1}}  = D_{(i_{1}}\om_{i_{2}\dots i_{k+1})} - \tfrac{k}{n + 2(k-1)} h_{(i_{1}i_{2}}\div(\om)_{i_{3}\dots i_{k+1})}.
\end{align}
If $X \in \Ga(TM)$, then $\clie(X^{\flat})$ is the Lie derivative of the conformal structure $[h]$ along $X$, for
\begin{align}\label{cliex}
2\clie(X^{\flat})_{ij} = 2\mr{DX^{\flat}} = 2D_{(i}X_{j)} - \tfrac{2}{n}D_{p}X^{p} h_{ij} =\tf(\lie_{X}h)_{ij} = (\lie_{X}[h])_{ij}.
\end{align}
The identity \eqref{cliex} motivates use of the notation, $\clie$, resembling that for the Lie derivative. 

The \emph{trace-free Codazzi operator} $\klie$ on $\Ga(\symk)$ is defined to be the trace-free part of the Codazzi operator $\cod(\om)$ defined in \eqref{codazzidefined}. For trace-free $\om \in \Ga(\symkt)$,
\begin{align}\label{kliedefined}
\begin{aligned}
\klie&(\om)_{ij i_{1}\dots i_{k-1}} = \cod(\om)_{iji_{1}\dots i_{k-1}} - \tfrac{1}{n+k-3}\sum_{s = 1}^{k-1}h_{i_{s}[i}\div(\om)_{j]i_{1}\dots \hat{i}_{s}\dots i_{k-1}}.\\
\end{aligned}
\end{align}
By definition, 
\begin{align}
&\klie(\om)_{iji_{1}\dots i_{k-1}} = \klie(\om)_{[ij]i_{1}\dots i_{k-1}}, && 
&\klie(\om)_{iji_{1}\dots i_{k-1}} = \klie(\om)_{ij(i_{1}\dots i_{k-1})}, && \klie(\om)_{[ijk]i_{1}\dots i_{k-2}} = 0. 
\end{align}
Note that, for $\om \in \Ga(\ctm) = \Ga(S^{1}_{0}(\ctm))$, $\klie$ is half the exterior differential, $\klie(\om)_{ij} = D_{[i}\om_{j]} = \tfrac{1}{2}d\om_{ij}$.

A \emph{rank $k$ Codazzi tensor} is a completely symmetric tensor $\om \in \Ga(S^{k}(\ctm)) \cap \ker \cod$. Equivalently, $D\om$ is also completely symmetric. The divergence of a trace-free Codazzi tensor vanishes, for $D^{p}\om_{pi_{1}\dots i_{k-1}} = D_{i_{1}}\om^{p}\,_{pi_{2}\dots i_{k-1}} = 0$, so the space of rank $k$ trace-free Codazzi tensors is $\Ga(\symkt)\cap \ker \klie \cap \ker \div = \Ga(\symkt)\cap \cod$. Properties of trace-free Codazzi tensors have been studied in \cite{Shandra-Stepanov-Mikes, Stepanov, Stepanov-Aleksandrova-Tsyganok, Stepanov-Tsyganok}.

A \emph{rank $k$ conformal Codazzi tensor} is a tensor $\om \in \Ga(S^{k}(\ctm))$ such that $\klie(\tf \om) = 0$. By definition a symmetric tensor is conformal Codazzi if and only if its trace-free part is conformal Codazzi. Trace-free Codazzi tensors and divergence-free conformal Codazzi tensors are the same thing.

\begin{example}
Let $h$ be a metric on a manifold of dimension $n \geq 3$ and suppose $\sY \in \mcurv(\ctm)$ satisfyies the differential Bianchi identity. Define $\si(\sY) = \rictr(\sY) - \tfrac{1}{2(n-1)}\scal(\sY)h$, so that $\si(\riem)$ is the Schouten tensor of $h$. Straightforward computations using \eqref{tfweyl} and Lemma \ref{divergencefreecurvaturelemma} show
\begin{align}
&\div \tf(\rictr(\sY))  = \tfrac{n-2}{2n}d\scal(\sY),&&
\div \tf(\sY)  =  \tfrac{2(3-n)}{n-2}\cod(\si(\sY)) =  \tfrac{2(3-n)}{n-2}\klie(\tf\rictr(\sY)).
\end{align}
It follows that if $n > 3$, then the following are equivalent: $\rictr(\sY)$ is conformal Codazzi; $\si(\sY)$ is Codazzi; and $\sY$ is divergence free. In particular, if $h$ is conformally flat and $n \geq 3$, then the Schouten tensor $\si(\riem)$ is a Codazzi tensor, and if $n \geq 4$, a metric $h$ has harmonic Weyl tensor (meaning $\div \tf\riem = 0$) if and only if its Ricci curvature is conformal Codazzi.
\end{example}

For $\om \in \Ga(\symkt)$ define 
\begin{align}\label{tliedefined}
\tlie(\om)_{ii_{1}\dots i_{k}} = \tfrac{2k}{k+1}\klie(\om)_{i(i_{1}\dots i_{k})},
\end{align}
which is completely trace-free and satisfies $\tlie(\om)_{i(i_{1}\dots i_{k})} = \tlie(\om)_{ii_{1}\dots i_{k}}$ and $\tlie(\om)_{(i_{1}\dots i_{k+1})} = 0$. Using 
\begin{align}
\klie(\om)_{[iji_{1}]i_{2}\dots i_{k-1}} = 0
\end{align}
it can be checked that $\tlie(\om)_{[ij]i_{1}\dots i_{k-1}} = \klie(\om)_{iji_{1}\dots i_{k-1}}$, so that each of $\klie$ and $\tlie$ determines the other on $\symkt$. The preceding implies that the kernels of $\klie$ and $\tlie$ on $\Ga(\symkt)$ are the same. 

The linear map $\ih:\Ga(\symkmt)\to \Ga(\ctm \tensor \symkt)$ defined by
\begin{align}
\ih(\om)_{ii_{1}\dots i_{k}} = \tfrac{k(n+2(k-2))}{(n+k-3)(n+2(k-1))}h_{i(i_{1}}\om_{i_{2}\dots i_{k})} + \tfrac{k(1-k)}{(n+k-3)(n+2(k-1))}h_{(i_{1}i_{2}}\om_{i_{3}\dots i_{k})i},
\end{align}
(for $f \in \cinf(M) = \Ga(S^{0}_{0}(\ctm))$, $\ih(f)_{ij} = \tfrac{1}{n}fh_{ij}$)
is characterized by the properties that its image is contained in $\Ga(\ctm \tensor \symkt)$ and that the nontrivial traces of $\ih(\om)$ equal $\om$. In particular, $\ih$ is injective.  

Recall the definition \eqref{tliedefined} of the operator $\tlie$.
The identity \eqref{normdom} of Lemma \ref{domdecompositionlemma} is a special case of \cite[Lemma $2.24$]{Bourguignon-Hijazi-Milhorat-Moroianu-Moroianu}. 

\begin{lemma}\label{domdecompositionlemma}
 For $\om_{i_{1}\dots i_{k}} \in \Ga(\symkt)$ there hold
\begin{align}\label{domkl}
D\om & = \clie(\om) + \tlie(\om) + \ih(\div(\om)),\\
\label{normdom}
|D\om|^{2} & - |\clie(\om)|^{2} - \tfrac{k(n+2(k-2))}{(n+k-3)(n+2(k-1))}|\div(\om)|^{2} =  |\tlie(\om)|^{2}  =  \tfrac{2k}{k+1}|\klie(\om)|^{2}.
\end{align}
\end{lemma}
\begin{proof}
Substituting the definitions of $\clie$, $\tlie$, and $\klie$ into $D_{i}\om_{i_{1}\dots i_{k}} = D_{(i}\om_{i_{1}\dots i_{k})} + \tfrac{2}{k+1}\sum_{s = 1}^{k}D_{[i}\om_{i_{s}]i_{1}\dots \hat{i}_{s}\dots i_{k}}$ and simplifying the trace terms yields \eqref{domkl}.  Pairing \eqref{domkl} with itself yields \eqref{normdom}.
\end{proof}
Symmetrizing \eqref{domkl} recovers \eqref{cliedefined}.
It is immediate from \eqref{normdom} that $\ker D \cap \Ga(S^{k}_{0}(\ctm)) = \ker \clie\cap \ker \klie \cap \ker \div \cap \Ga(S^{k}_{0}(\ctm))$. 

The identity \eqref{domkl} of Lemma \ref{domdecompositionlemma} justifies that $\clie$, $\tlie$, and $\div$ are the generalized gradients on $\symkt$. Together with the equivalence of $\klie$ and $\tlie$ as in previous paragraph this justifies regarding $\klie$ as a generalized gradient or, at any rate, using it instead of $\tlie$.

\begin{lemma}
On an $n$-dimensional pseudo-Riemannian manifold $(M, h)$, the operators $\clie$, $\klie$, and $\div$ are conformally invariant in the sense that for $0 < f \in \cinf(M)$ there hold
\begin{align}\label{lieconf}
&\clie_{fh}(f^{k}\om) = f^{k}\clie_{h}(\om), && \klie_{fh}(f^{(k-1)/2}\om) = f^{(k-1)/2}\klie_{h}(\om),&
&f\div_{fh}(f^{(2-n)/2}\om) = f^{(2-n)/2}\div_{h}(\om).
\end{align}
\end{lemma}
\begin{proof}
The Levi-Civita connections $\tD$ and $D$ of conformally related pseudo-Riemannian metrics $fh_{ij}$ and $h_{ij}$ are related by $\tD - D =   2\si_{(i}\delta_{j)}\,^{k} - h_{ij}h^{kp}\si_{p}$ with $2\si_{i} = d\log f_{i}$ and $\si^{i} = h^{ip}\si_{p}$. Write $\sbl_{\clie}(Z)(\phi)$ for the symbol of $\clie$ applied to $Z \in \Ga(TM)$ and $\phi \in \Ga(\symkt)$, and similarly for $\klie$ and $\div$. 
For $\om \in \Ga(\symkt)$, $0 < f \in \cinf(M)$, and $\al \in \rea$, there hold
\begin{align}
\begin{aligned}
&\clie_{fh}(f^{\al}\om) = f^{\al}\left(\clie_{h}(\om) + 2(\al -k)\sbl_{\clie}(\si^{\sharp})(\om)\right),&\\
&\klie_{fh}(f^{\al}\om) = f^{\al}\left(\klie_{h}(\om) + (2\al + 1-k)\sbl_{\klie}(\si^{\sharp})(\om)\right),\\
&f\div_{fh}(f^{\al}\om) = f^{\al}\left(\div_{h}(\om) + (n-2 + 2\al)\imt(\si^{\sharp})\om\right),
\end{aligned}
\end{align}
from which \eqref{lieconf} follows.
\end{proof}
Define $\clie^{\sharp}:\Ga(\symktv) \to \Ga(\symkptv)$ by $\clie^{\sharp}(X) = \clie(X^{\flat})^{\sharp}$. Then \eqref{lieconf} implies the invariance $f\clie^{\sharp}_{\tilde{h}}(X) = \clie^{\sharp}_{h}(X)$, so that while $\lie^{\sharp}$ depends on $h$, the subspace $\ker \clie^{\sharp} \cap \Ga(\symktv)$ does not. 

\begin{lemma}\label{cklemma}
For $X \in \Ga(S^{k}(TM))$, $\clie^{\sharp}(\tf X) = 0$ if and only if $\{h, X^{\flat}\} \in \im \sraise$.
\end{lemma}
\begin{proof}
Write $X^{\flat} = \tf(X^{\flat}) + h \sprod \al = \tf(X)^{\flat} + h \sprod \al$ for $\al \in \Ga(S^{k-2}(\ctm))$. By \eqref{cliedefined},
\begin{align}\label{ckidentity}
\{h, X^{\flat}\} & = \{h, \tf(X)^{\flat}\} + \{h,h \sprod\al\} = 2\clie(\tf(X)^{\flat}) + h \sprod\left( \tfrac{2k}{n+2(k-1)}\div(\tf(X)^{\flat}) + \{h, \al\}\right),
\end{align} 
from which the claimed equivalence is apparent.
\end{proof}

 A \emph{rank $k$ conformal Killing tensor} is a symmetric tensor $X \in \Ga(S^{k}(TM)))$ satisfying the equivalent conditions of Lemma \ref{cklemma}. The definition means that the complete symmetrization of the covariant derivative of $X$ is contained in the Poisson module generated by the bivector dual to the metric. By definition, a trace-free symmetric tensor is conformal Killing if and only if it is in $\ker \clie^{\sharp}$. 
 Via metric duality conformal Killing tensors are identified with elements of $\Ga(\symk)$ whose trace-free part is contained in $\ker \clie$. Lemma \ref{cklemma} recovers the observation from \cite{Heil-Moroianu-Semmelmann} that a symmetric tensor is conformal Killing if and only if its trace-free part is conformal Killing. By \eqref{ckidentity}, a trace-free conformal Killing tensor is a Killing tensor if and only if it is also divergence free. For explicit constructions of Killing tensors of rank higher than two see \cite{Cariglia-Gibbons-vanHolten-Horvathy-Kosinski-Zhang, Gibbons-Houri-Kubiznak-Warnick}.
 
\begin{remark}
For a Riemann surface, conformal Killing and trace-free Codazzi tensors are exactly the real parts of holomorphic sections of powers of the canonical line bundle and for a compact Riemann surface their existence is controlled by the genus via the Riemann-Roch theorem and linked to curvature via the Gauss-Bonnet theorem \cite[section $3$]{Fox-2dahs}. This remark provides one justification for treating such tensors simultaneously. Another is that when the covariant derivative of a trace-free symmetric tensor is decomposed by symmetries, the generalized gradients that appear are the divergence operator and the Codazzi and conformal Killing operators.
\end{remark}

\subsection{The hierarchy of curvature equations for trace-free symmetric tensors}\label{couplingsection}

Lemma \ref{consistencylemma} shows that for a Codazzi tensor $\om$ the curvature tensor $\om \kwedge \om$ automatically satisfies the differential Bianchi identity.
\begin{lemma}\label{consistencylemma}
Let $(M, h)$ be an $n$-dimensional pseudo-Riemannian manifold and suppose $k \geq 2$. For any $\om \in \Ga(S^{k}\ctm)$ such that $\cod(\om) = 0$ there holds $D_{[m}(\om \kwedge \om)_{ij]kl} = 0$.
\end{lemma}
\begin{proof}
A straightforward computation shows that for $\om \in \Ga(S^{k}\ctm)$,
\begin{align}
\begin{aligned}
D_{[m}(\om \kwedge \om)_{ij]kl} &= 2\om_{k[i}\,^{A}D_{m}\om_{j]lA} -2\om_{l[i}\,^{A}D_{m}\om_{j]kA}  = 2\om_{k[i}\,^{A}\cod(\om)_{mj]lA} -2\om_{l[i}\,^{A}\cod(\om)_{mj]kA} ,
\end{aligned}
\end{align}
where $A$ is shorthand for $a_{1}\dots a_{k-2}$. The claim follows.
\end{proof}

\begin{corollary}\label{uniqueconsistencycorollary}
Let $(M, h)$ be an $n$-dimensional pseudo-Riemannian manifold and suppose $k \geq 2$. The map $\phi: \Ga(\symkt) \times \Ga(\symkt) \to \mcurv(\ctm)$ defined by
\begin{align}\label{symphi}
\phi(\om, \om) = A\om \kwedge \om + B \rictr(\om \kwedge \om) \kwedge h + C \scal(\om \kwedge \om)h \kwedge h
\end{align}
satisfies
\begin{align}\label{consistentsym}
\begin{aligned}
D_{[m}\phi(\om , \om)_{ij]kl} &= A\left(2 \om_{k[i}\,^{A}\cod(\om)_{mj]lA} -2\om_{l[i}\,^{A}\cod(\om)_{mj]kA}\right) \\
&\quad + B \left(h_{k[i}\cod(\rictr(\om\kwedge \om))_{mj]l} -  h_{l[i}\cod(\rictr(\om\kwedge \om))_{mj]k}\right) + C d_{[m}\scal(\om \kwedge \om)(h\kwedge h)_{ij]kl},
\end{aligned}
\end{align}
for all $\om \in \Ga(\symkt)$. In particular $D_{[m}\phi(\om , \om)_{ij]kl} = 0$ for all $\om \in \ker \cod \cap \Ga(\symkt)$ if $B = C = 0$.
\end{corollary}
\begin{proof}
The identity \eqref{consistentsym} follows from Lemma \ref{consistencylemma} via a straightforward computation. By \eqref{consistentsym} if $B = C = 0$ and $\cod(\om) = 0$, then $D_{[m}\phi(\om , \om)_{ij]kl} = 0$. 
\end{proof}
Together with Lemma \ref{kwedgelemma} and Corollary \ref{uniqueconsistencycorollary} imply that any multiple of the coupling $\phi(\om, \om) = \om \kwedge \om$ has nontrivial trace-free part and is automatically consistent. Although the details are omitted because the claim is not needed, with a bit more work it can be shown that for $\phi$ as in \eqref{symphi}, $D_{[m}\phi(\om , \om)_{ij]kl} = 0$ for all $\om \in \ker \cod \cap \Ga(\symkt)$ if and only if $B = C = 0$, so that up to scale $\phi(\om, \om) = \om \kwedge \om$ is the unique choice of coupling that is automatically consistent.

By Lemmas \ref{divergencefreecurvaturelemma} and \ref{consistencylemma}, if $\om \in \Ga(S^{k}\ctm)$ is Codazzi, then $\ef(\om \kwedge \om)$ is divergence-free. The quadratic map $\stp:\Ga(\symkt) \to \Ga(S^{2}\ctm)$ defined by 
\begin{align}\label{stpdefined}
\begin{aligned}
\stp(\om) &= \begin{cases}
(\om \tensor \om) - \tfrac{1}{2}|\om|^{2}h & k = 1,\\
\ef(\om \kwedge \om) = \rictr(\om \kwedge \om) - \tfrac{1}{2}|\om|^{2}h & k \geq 2,
\end{cases}
\end{aligned}
\end{align}
determines the stress-energy tensor associated with a trace-free Codazzi tensor. This is justified by the identity
\begin{align}
\label{divrictrkbairdintro}
 \begin{aligned} h(X, \div(\stp(\om)))&= -2h(\klie(\om), \sbl_{\klie}(X)(\om)) + \tfrac{n-2}{n+k-3}h(\sbl_{\div}(X)(\om), \div(\om)),
\end{aligned}
\end{align}
having the form \eqref{generalizedbaird} and proved in Theorem \ref{divrictrtheorem} (in \eqref{divrictrkbairdintro}, $\sbl$ denotes the symbol of a differential operator).

When $\om$ vanishes identically, the following hierarchy of equations for a pair $(h, \om)$, a smooth function $\ka$, and a constant $c \in \rea$ comprising a pseudo-Riemannian metric $h$ and a trace-free symmetric tensor $\om \in \Ga(\symkt)$ specializes to the usual hierarchy of constant sectional curvature (projectively flat), Einstein, and constant scalar curvature:
\begin{flalign}
\tag{TFC1}\label{projectivehiggsintro} 
&\emph{Coupled projective flatness}:& &\klie(\om) = 0,& &\div(\om) = 0,&& \riem - c(\om \kwedge \om) = -\tfrac{\ka}{n(n-1)}h\kwedge h,
\\
\tag{TFC2}\label{stressenergyintro}  
&\emph{Coupled Einstein equations}:& &\klie(\om) = 0,& &\div(\om) = 0,&& \ric - \tfrac{1}{2}\scal h + \tfrac{n-2}{2n}\ka h =c\stp(\om),
\\
\tag{TFC3}\label{constraintintro} 
&\emph{Constraint equations}: & &\klie(\om) = 0,& &\div(\om) = 0,&& D_{i}(\scal - c|\om|^{2}) = 0.&
\end{flalign}

\begin{itemize}[leftmargin=11pt]
\item Because $\om$ is trace-free the pair of equations $\klie(\om) = 0$ and $\div(\om) = 0$ is equivalent to the single equation $\cod(\om) = 0$, but it seems better to write the pair because it conforms with the general scheme and because it makes clearer that each of the two equations can be relaxed individually.

\item Because $\om$ is a trace-free Codazzi tensor, from the differential Bianchi identity it follows that for a solution of \eqref{stressenergyintro} the function $\ka = \scal - c|\om|^{2}$ is constant (provided $n > 2$). Consequently, for a trace-free Codazzi tensor $\om$, \eqref{stressenergyintro} is equivalent to
\begin{align}\tag{TFC2b}\label{stressenergyintrob}
&\klie(\om) = 0,& &\div(\om) = 0,&&&\ric-  c\rictr(\om\kwedge \om) =\tfrac{\ka}{n} h,
\end{align}
from which it is evident that a solution of \eqref{projectivehiggsintro} solves \eqref{stressenergyintro}.

\item In all three systems \eqref{projectivehiggsintro}-\eqref{constraintintro}, the absolute value of $c$ can always be absorbed into $\om$, but its sign cannot.
Both the existence and behavior of solutions depend strongly on the sign of $c$, as well as the signature of $h$, as is illustrated partly by Theorems \ref{scalarcurvaturetheorem} and \ref{simonstheorem}.

\item The special case of \eqref{constraintintro} with $k = 2$ and $c$ positive also includes the special case of the constraint equations of general relativity for a mean curvature zero spacelike hypersurface in a spacetime solving the Einstein field equations \cite{Carlotto}, and this motivates calling \eqref{constraintintro} the \emph{coupled constraint} equations.	

\item If $c$ has the wrong sign, the solutions of \eqref{stressenergy} may not admit a physical interpretation, as physical considerations focus attention on solutions of the Einstein field equations for which the energy momentum tensor satisfies some energy condition, such as the \emph{weak energy condition} that $x^{i}x^{j}\T_{ij} \geq 0$ for all timelike vector fields $x^{i}$, where $x^{i}$ is \emph{timelike} if $h(x, x) < 0$. (Such a condition is vacuous if $h$ is Riemannian.) Changing the sign of $c$ destroys such a condition. On the other hand, examples coming from the study of submanifolds show that the equations \eqref{projectivehiggsintro}-\eqref{constraintintro} have mathematical interest even without such an energy condition (for example, see Example \ref{constraintequationsexample}). As was already remarked, similar observations apply to the Einstein-Maxwell hierarchy and its supergravity generalizations.

\item By Lemma \ref{consistencylemma} the generalized gradient equations of \eqref{projectivehiggsintro} imply the consistency of the curvature equation of \eqref{projectivehiggsintro}. This is what was meant earlier by the automatic consistency of \eqref{projectivehiggsintro}. Note that the analogous claim for \eqref{coupledEMp} is not true. Such a claim depends on the Young type of the symmetries of the coupled tensor and the particular generalized gradients considered.
\end{itemize}

The equations \eqref{projectivehiggsintro} and \eqref{stressenergyintro} are special cases of more general equations \eqref{projectivehiggs} and \eqref{stressenergy} defined in Section \ref{couplingsection}, in which the conditions stating that $\om$ be trace-free Codazzi, $\klie(\om) = 0$ and $\div(\om) = 0$, are replaced by the weaker condition, analogous to that for exterior forms in the supergravity-like equations \eqref{supergravity}, that there vanish the right-hand side of \eqref{divrictrkbairdintro}. 

In the rest of this section, the corresponding equations coupling a metric with a trace-free Killing tensor is described as well, although only at the Einstein/Ricci level. Because conformal Killing tensors generally do not exist on negatively curved manifold, these equations are likely of most interest when the background metric has nonnegative curvature in some sense. These equations can be seen also as a natural generalization of conditions coupling a metric on an oriented surface to holomorphic sections of powers of the canonical line bundle as in \cite[Section $8$]{Fox-2dahs}. 

Define quadratic maps $\stpm:\Ga(\symkt) \to \Ga(S^{2}\ctm)$ by 
\begin{align}\label{stmpdefined}
\begin{aligned}
\stp(\om) &= \begin{cases}
\om \tensor \om - \tfrac{1}{2}|\om|^{2}h & k = 1,\\
\rictr(\om \kwedge \om) - \tfrac{1}{2}|\om|^{2}h & k \geq 2,
\end{cases}&&
\stm(\om) &= \begin{cases}
(\om \tensor \om) + \tfrac{1}{2}|\om|^{2}h & k = 1,\\
\rictr(\om \kwedge \om) + \tfrac{1}{2k}|\om|^{2}h & k \geq 2.
\end{cases}
\end{aligned}
\end{align}
Note the asymmetry in the definitions; there is a $k$ in the definition of $\stm$ absent in the definition of $\stp$. 
The $\pm$ in the notation is motivated by thinking of sections of $S^{k}_{0}TM$ as sections of $S^{-k}_{0}\ctm$, so corresponding to negative integers.  Although this has no sense in general, it makes sense on oriented surfaces, for which conformal Killing and trace-free Codazzi tensors are exactly the real parts of holomorphic differentials \cite[Section $3$]{Fox-2dahs}.

By definition,
\begin{align}\label{trstmp}
&\tr \stp(\om) = \tfrac{2-n}{2}|\om|^{2}, & &\tr \stm(\om) = \tfrac{n+2k}{2k}|\om|^{2}.
\end{align}
The operators $\stpm$ depend on $h$ and, when necessary, this dependence is indicated by a subscript, $\stpm_{h}$. For $\om \in \Ga(\symkt)$ and $0 < f \in \cinf(M)$, 
\begin{align}
&\om\kwedge_{fh}\om = f^{2-k}\om\kwedge \om,& &\rictr_{fh}(\om\kwedge_{fh}\om) = f^{1-k}\rictr_{h}(\om\kwedge \om), & & |\om|_{fh}^{2} = f^{-k}|\om|^{2}.
\end{align}
It follows that the operators $\stpm$ are conformally invariant in the sense that
\begin{align}
&\stpm_{fh}(f^{\al}\om) = f^{2\al + 1-k}\stpm_{h}(\om), & &\om \in \Ga(\symkt).
\end{align}
For $\om \in \Ga(\symkt)$ define one-forms by
\begin{align}
&\lb \kcons_{h}(\om), X\ra = \lb \om, \imt(X)\klie_{h}(\om)\ra ,& 
&\lb \ccons_{h}(\om), X\ra = \lb \om, \imt(X)\clie_{h}(\om)\ra ,& 
&\lb \divcons_{h}(\om), X\ra = \lb \imt(X)\om, \div_{h}(\om)\ra,& 
\end{align}
for $X \in \Ga(\ctm)$.
By \eqref{lieconf}, with $\si = \tfrac{1}{2}d\log f$,
\begin{align}\label{conformalscaling1}
\begin{aligned}
f\kcons_{fh}(f^{\al}\om) & = f^{2\al + 1 -k}\left( \kcons_{h}(\om) + \tfrac{2\al + 1 - k}{2}\left(|\om|^{2}\si_{i} - \tfrac{n+2(k-2)}{n+k-3}\imt(\si^{\sharp})\rictr(\om \kwedge \om) \right)\right),\\
f\ccons_{fh}(f^{\al}\om) & = f^{2\al + 1 -k}\left( \ccons_{h}(\om) + \tfrac{2(\al  - k)k}{k+1}\left(|\om|^{2}\si_{i} + \tfrac{n+2(k-2)}{n+2(k-1)}\imt(\si^{\sharp})\rictr(\om \kwedge \om) \right)\right),\\
f\divcons_{fh}(f^{\al}\om) & = f^{2\al + 1-k}\left( \divcons_{h}(\om) + (n -2 + 2\al)\imt(\si^{\sharp})\rictr(\om \kwedge \om) \right).
\end{aligned}
\end{align}

Theorem \ref{divrictrtheorem} describes the identities of the form \eqref{generalizedbaird} for trace-free symmetric tensors.
\begin{theorem}\label{divrictrtheorem}
Let $M$ be a manifold of dimension $n \geq 2$ and let $h$ be a pseudo-Riemannian metric. For $\om \in \Ga(\symkt)$ and $X \in \Ga(\ctm)$ there hold
\begin{align}
\label{divrictrkbaird}
 \begin{aligned} h(X, \div(\stp(\om)))&=
 -2h(X, \kcons(\om)) + \tfrac{n-2}{n+k-3}h(X, \divcons(\om)),\\
&=  -2h(\imt(X)\klie(\om), \om) + \tfrac{n-2}{n+k-3}h(\imt(X)\om, \div(\om)),\\
&= -2h(\klie(\om), \sbl_{\klie}(X)(\om)) + \tfrac{n-2}{n+k-3}h(\sbl_{\div}(X)(\om), \div(\om)),
\end{aligned}\\
\label{divrictrcbaird}
 \begin{aligned} h(X, \div(\stm(\om))) &  
 = \tfrac{k+1}{k}h(X, \ccons(\om))  + \tfrac{n+ 2k}{n+2(k-1)}h(X, \divcons(\om))\\
 &= \tfrac{k+1}{k}h(\imt(X)\clie(\om),\om) + \tfrac{n+ 2k}{n+2(k-1)}h(\imt(X)\om, \div(\om))\\
&= \tfrac{k+1}{k}h(\clie(\om),\sbl_{\clie}(X)(\om)) + \tfrac{n+ 2k}{n+2(k-1)}h(\sbl_{\div}(X)(\om), \div(\om))).
\end{aligned}
\end{align}
\end{theorem}
\begin{proof}
Suppose $k \geq 2$ and let $\om \in \Ga(\symkt)$.
Using \eqref{domkl} yields
\begin{align}
\begin{aligned}
\tfrac{1}{2}D_{i}|\om|^{2} & = \om^{i_{1}\dots i_{k}}D_{i}\om_{i_{1}\dots i_{k}}  =  \tfrac{2k}{k+1}\kcons(\om)_{i} +\ccons(\om)_{i} + \tfrac{k(n+2(k-2))}{(n+k-3)(n+2(k-1))}\divcons(\om)_{i},
\end{aligned}
\end{align}
which shows
\begin{align}
\label{divrictr}
\begin{aligned}
\div(\rictr(\om \kwedge \om))_{i} &  = -\tfrac{2}{k+1}\kcons(\om)_{i} + \ccons(\om)_{i} + \left(1 +  \tfrac{n-2}{(n+k-3)(n + 2(k-1))}\right)\divcons(\om)_{i}\\
&  = -\tfrac{2}{k+1}\kcons(\om)_{i} + \ccons(\om)_{i} + \tfrac{(n+ k-2)(n+ 2(k-2)) + 2(n-2)}{(n+k-3)(n + 2(k-1))}\divcons(\om)_{i},
\end{aligned}
\end{align}
Contracting $\om_{i_{1}\dots i_{k}}$ with $D_{i_{1}}\om_{i_{2}\dots i_{k}i}$ and using \eqref{domkl} yields
\begin{align}\label{divnormpre}
\begin{aligned}
\om^{i_{1}\dots i_{k}}D_{i_{1}}\om_{i_{2}\dots i_{k}i} & =  - \tfrac{2}{k+1}\kcons(\om)_{i} + \ccons(\om)_{i} + \tfrac{2(1-k)}{(n+k-3)(n+2(k-1))}\divcons(\om)_{i}.
\end{aligned}
\end{align}
Differentiating $\rictr(\om \kwedge \om)_{ij} = \om_{ia_{1}\dots a_{k-1}}\om_{j}\,^{a_{1}\dots a_{k-1}}$ yields
\begin{align}\label{divrictrom}
\begin{aligned}
 \om^{i_{1}\dots i_{k}}D_{i_{1}}\om_{i_{2}\dots i_{k}i} & %
= \div(\rictr(\om \kwedge \om))_{i} - \divcons(\om)_{i}.
\end{aligned}
\end{align}
Substituting \eqref{divrictrom} in \eqref{divnormpre} yields  
\begin{align}
\begin{aligned}
\label{divnorm}\tfrac{1}{2}\div(|\om|^{2}h)_{i} & = \tfrac{1}{2}D_{i}|\om|^{2}  = \tfrac{2k}{k+1}\kcons(\om)_{i} + \ccons(\om)_{i} + \tfrac{k(n+2(k-2))}{(n+k-3)(n+2(k-1))}\divcons(\om)_{i}.
\end{aligned}
\end{align}
Taking linear combinations of \eqref{divrictr} and \eqref{divnorm} in different ways and substituting the result in \eqref{stmpdefined} yields 
\begin{align}
\label{divrictrk}
\div(\stp(\om))_{i} &= \div(\rictr(\om \kwedge \om))_{i}  - \tfrac{1}{2}D_{i}|\om|^{2}  = -2\kcons(\om)_{i} + \tfrac{n-2}{n+k-3}\divcons(\om)_{i},\\
\label{divrictrc}\div(\stm(\om))_{i} & = \div(\rictr(\om \kwedge \om))_{i}  + \tfrac{1}{2k}D_{i}|\om|^{2} = \tfrac{k+1}{k}\ccons(\om)_{i} + \tfrac{n+ 2k}{n+2(k-1)}\divcons(\om)_{i},
\end{align}
where, when $k = 1$, \eqref{divrictrk}-\eqref{divrictrc} have sense if $\rictr(\om \kwedge \om)_{ij}$ is interpreted as $\om_{i}\om_{j}$.
The $k = 1$ cases of \eqref{divrictrk} and \eqref{divrictrc} follow from 
\begin{align}
\begin{aligned}
\div(\ga \tensor \ga - \tfrac{1}{2}|\ga|^{2}h)_{i} & =   \div(\ga)\ga_{i} -  2\ga^{p}D_{[i}\ga_{p]}  = \div(\ga)\ga_{i} - 2\ga^{p}\klie(\ga)_{ip},\\
\div(\ga \tensor \ga + \tfrac{1}{2}|\ga|^{2}h)_{i} & =   \div(\ga)\ga_{i} +  2\ga^{p}D_{(i}\ga_{p)}  = \tfrac{n+2}{2}\div(\ga)\ga_{i} + 2\ga^{p}\clie(\ga)_{ip},
\end{aligned}
\end{align}
valid for any $\ga \in \Ga(\ctm)$. 
Pairing $X$ with \eqref{divrictrk} and \eqref{divrictrc} yields \eqref{divrictrkbaird} and \eqref{divrictrcbaird}
\end{proof}

\begin{remark}
Among the possible linear combinations of \eqref{divrictr} and \eqref{divnorm}, the combinations on the right-hand sides of \eqref{divrictrk} and \eqref{divrictrc} are distinguished by the absence of $\ccons(\om)$ and $\kcons(\om)$. This is relevant because, by \eqref{conformalscaling1},  the conformal scaling of $\kcons(\om)$, $\ccons(\om)$, and $\divcons(\om)$ is different. The first two rescale in a way dependent on $k$, while the last rescales in a way dependent only on dimension. This means that given a conformal structure $[h]$ it makes sense to impose that either $\kcons(\om)$ or $\ccons(\om)$ vanish, but does not make sense to require that both vanish. The vanishing of $\divcons(\om)$ can be treated as a condition selecting within the conformal class. 
\end{remark}

\begin{lemma}\label{tracefreeconstantlemma}
Let $(M, h)$ be a pseudo-Riemannian manifold $M$ of dimension $n \geq 3$. Consider sequences of tensors $\{\om(k)\in\Ga(S^{k}_{0}(\ctm)): k \geq 1\}$ and $\{\ga(k)\in\Ga(S^{k}_{0}(\ctm)): k \geq 1\}$ such that
\begin{itemize}
\item $\om^{(k)} \in\Ga(S^{k}_{0}(\ctm)) \cap \ker \divcons \cap \ker \kcons$,
\item $\ga^{(k)} \in\Ga(S^{k}_{0}(\ctm)) \cap \ker \divcons \cap \ker \ccons$, and
\item the tensors $\stcod =  \sum_{k \geq 1}a_{k}\stp(\om^{(k)})$ and $\stkill = \sum_{k \geq 1}b_{k}\stm(\ga^{(k)})$ converge pointwise to smooth sections of $S^{2}(\ctm)$ for some sequences $\{a_{k} \in \rea: k \geq 1\}$ and $\{b_{k} \in \rea: k \geq 1\}$.
\end{itemize}
Then the equations
\begin{align}\label{stressenergy2}
\sric - \tfrac{1}{2}\scal h + \tfrac{n-2}{2n}\ka h = \stcod + \stkill 
\end{align}
are consistent and 
\begin{align}\label{stressenergyconstant}
\ka =  \scal - \sum_{k \geq 1}a_{k}|\om^{(k)}|^{2} + \sum_{k \geq 1}\tfrac{n+2k}{k(n-2)}b_{k}|\ga^{(k)}|^{2} 
\end{align}
is a constant. 
If $n > 2$ the equations \eqref{stressenergy2} hold for some $\ka \in \rea$ if and only if 
\begin{align}\label{tracefreericcondition}
\sric - a_{1}\om^{(1)}\tensor\om^{(1)} - \sum_{k \geq 2}a_{k}\rictr(\om^{(k)}\kwedge \om^{(k)}) - b_{1}\ga^{(1)}\tensor\ga^{(1)} - \sum_{k \geq 2}b_{k}\rictr(\ga^{(k)}\kwedge \ga^{(k)}) = \tfrac{\ka}{n} h,
\end{align}
in which case $\ka$ has the form \eqref{stressenergyconstant}.
\end{lemma}

\begin{proof}
By Theorem \ref{divrictrtheorem}, $\stcod$ and $\stkill$ are divergence free.  Taking the divergence of both sides of \eqref{stressenergy2} shows that $\ka$ must be constant if \eqref{stressenergy2} is to admit solutions. The form \eqref{stressenergyconstant} for $\ka$ follows by tracing both sides of \eqref{stressenergy2}. Evidently \eqref{stressenergy2} implies \eqref{tracefreericcondition}. If there holds \eqref{tracefreericcondition}, taking the divergence of \eqref{tracefreericcondition} and using the traced differential Bianchi identity $2\div(\sric) = d\scal$ and Theorem \ref{divrictrtheorem} shows
\begin{align}
0 = \tfrac{n-2}{n}\left( \scal - \sum_{k \geq 1}a_{k}|\om^{(k)}|^{2} + \sum_{k \geq 1}\tfrac{n+2k}{k(n-2)}b_{k}|\ga^{(k)}|^{2}  \right),
\end{align}
so that, since $n > 2$, $\ka$ as in \eqref{stressenergyconstant} is constant. With \eqref{tracefreericcondition} this implies \eqref{stressenergy2}.
\end{proof}

\begin{remark}
The hypotheses of Lemma \ref{tracefreeconstantlemma} are satisfied provided $\om^{(k)}$ is a rank $k$ trace-free Codazzi tensor and $\ga^{(k)}$ is a rank $k$ Killing tensor.
\end{remark}

\begin{remark}
The equations \eqref{stressenergy2} make sense for $h$ of any signature, although for indefinite signature $h$ the notation $|\om^{(k)}|^{2}$ is misleading as this expression need not be positive. 
\end{remark}

\begin{corollary}
On an $n$-manifold $M$ a pair $(h, \om)$ comprising a pseudo-Riemannian metric $h$ and a tensor $\om \in \Ga(\symkt)$ solves the system
\begin{align}\label{stressenergy}
\begin{aligned}
0 & = \div(\stp(\om))  =-2\kcons(\om) + \tfrac{(n-2)}{(n+k-3)}\divcons(\om) ,\\
\sric &- \tfrac{1}{2}\scal h + \tfrac{n-2}{2n}\ka h = c\left(\rictr(\om\kwedge \om) -  \tfrac{1}{2}|\om|^{2}h\right) = c\stp(\om),
\end{aligned}
\end{align}
for some real constants $\ka$ and $c$ if and only if $h$ solves the Einstein field equations with energy momentum tensor $c\stp(\om)$ and cosmological constant $\cosmo = \tfrac{n-2}{2n}\ka = \tfrac{n-2}{2n}\left(\scal - c|\om|^{2}\right)$. 
\end{corollary}

The system \eqref{stressenergy} has been written so as to make readily apparent its formal resemblance to the Einstein-Maxwell system for a metric and a two-form.

\begin{example}
In the case $k = 1$, for a solution $(h, \om)$ of \eqref{stressenergy}, the one-form $\om$ is $h$-harmonic. 
If, moreover, $\om = d\phi$ for some $\phi \in \cinf(M)$, then $\phi$ solves the wave equation $\lap_{h}\phi = 0$, and in Lorentzian signature can be interpreted as a massless scalar field, so the pair $(h, d\phi)$ satisfies the Einstein scalar field equations. In the Riemannian case, were $M$ compact, then $\phi$ would be harmonic, so constant, and $\om$ would vanish identically, but if $\om$ is not required to be exact, there can still be interesting solutions in Riemannian signature.
\end{example}

Thinking of $\om \in \Ga(\symkt)$ as a $S^{k-1}_{0}(\ctm)$-valued one-form, $(\om \kwedge \om)_{ijkl}$ can be viewed as a curvature term. 
Lemma \ref{projectivehiggslemma} shows that a metric and $\om \in \Ga(\symkt)\cap \ker \div \cap \ker \klie$ such that the modified curvature $\riem - \tfrac{1}{4}(\om \kwedge \om)$ is projectively flat, meaning it is a multiple of $h\kwedge h$, yield a solution of \eqref{stressenergy}.

\begin{lemma}\label{projectivehiggslemma}
Let $h$ be a pseudo-Riemannian metric on a manifold $M$ of dimension $n \geq 3$. Suppose $\om \in\Ga(\symkt) \cap \ker \div$ satisfies $\kcons(\om) = 0$ and that there is $\ka \in \cinf(M)$ such that the curvature of $h$ satisfies
\begin{align}\label{projectivehiggs}
\riem - c(\om \kwedge \om) = -\tfrac{\ka}{n(n-1)}h\kwedge h,
\end{align}
for some $c \in \rea$. Then $h$ and $\om$ solve the equations \eqref{stressenergy}. In particular, $\ka$ is a constant. 

If $n = 3$, then $(h, \om)$ solves \eqref{projectivehiggs} for $c \in \rea$ if and only if it solves \eqref{stressenergy}.
\end{lemma}

\begin{proof}
By \eqref{rictralbe}, tracing \eqref{projectivehiggs} yields $\ric = \tfrac{\ka}{n}h+ c\rictr(\om \kwedge \om)$. Tracing this yields $\scal = \ka + c|\om|^{2}$, and calculating the trace-free part of $\ric$ yields the last equation of \eqref{stressenergy}. From Lemma \ref{tracefreeconstantlemma} it follows that $\ka$ is a constant. 
If $(h, \om)$ solves \eqref{stressenergy}, then, by \eqref{tfweyl} and \eqref{stressenergy}, 
\begin{align}\label{stressenergyinverse}
\begin{aligned}
&\riem - c(\om \kwedge \om)= \sW - c\tf(\om \kwedge \om) -\tfrac{2}{n-2}\tf\left( \ric - c\rictr(\om \kwedge \om)\right) \kwedge h - \tfrac{1}{n(n-1)}\left( \scal - c|\om|^{2}\right) h \kwedge h \\
&= \sW - c\tf(\om \kwedge \om)  - \tfrac{\ka}{n(n-1)} h \kwedge h,
\end{aligned}
\end{align}
where the conformal Weyl tensor $\sW$ of $h$ is the trace-free part of $\riem$. If $n = 3$, then there vanish $\sW$ and $\tf(\om \kwedge \om)$, and \eqref{stressenergyinverse} shows that $h$ and $\om$ solve \eqref{projectivehiggs}.
\end{proof}

\section{Examples of solutions to the coupled projectively flat and coupled Einstein equations}\label{examplesection}
This section records examples of solutions of the coupled projectively flat equations \eqref{projectivehiggsintro} and coupled Einstein equations \eqref{stressenergyintro}. There is no pretense at completeness. Rather, the goals are to illustrate the abundance of solutions and the diverse nature of possible constructions. The different subsections are essentially independent, and also this entire section is independent of Section \ref{propertiessection}.

\subsection{Examples related with submanifold geometry}\label{geometricexamplesection}
This section records examples of solutions of the coupled projectively flat equations \eqref{projectivehiggsintro} coming from minimal immersions. These example arise from consideration of classes of hypersurfaces in special ambient spaces and characterized by conditions on their second fundamental forms. 

The \emph{second fundamental form} of an immersion $i:M \to N$ of a manifold $M$ in a manifold $N$ equipped with a projective structure $[\hnabla]$ is a normal bundle valued symmetric two-tensor on $M$, $\sff \in \Ga(S^{2}(T^{\ast}M)\tensor i^{\ast}(TN)/TM)$. By definition $\sff(X, Y)$ is the projection of $\hnabla_{X}Y$ onto the normal bundle for $X, Y \in \Ga(TM)$ and any torsion-free representative $\hnabla \in [\hnabla]$. Each nonvanishing section $W$ over $U \subset M$ of the normal bundle of a hypersurface immersion determines an identification of the second fundamental form with a symmetric tensor $\Pi$ on $U$ defined by $\sff(X, Y) = \Pi(X, Y)W$. Different such sections determine conformally related tensors. The tensor $\Pi$ is said to be the second fundamental form with respect to the transversal $W$ and $\hnabla$. 

\subsubsection{Examples from mean curvature zero immersions in space forms}\label{constraintequationsexample}
Theorem \ref{hypersurfacetheorem} yields solutions of the coupled projectively flat equations \eqref{projectivehiggsintro} with $c$ having either sign. 
\begin{theorem}\label{hypersurfacetheorem}
Let $(N, g)$ be an $(n+1)$-dimensional pseudo-Riemannian space form with scalar curvature $\aR_{g}$ and let $i:M \to N$ be an immersion of an $n$-dimensional hypersurface $M$ such that $h = i^{\ast}(g)$ is nondegenerate. Let $\Pi$ be the second fundamental form of the immersion $i$ defined with respect to a unimodular transverse vector field $Z$ orthogonal to $i(M)$ and satisfying $\ep = |Z|^{2}_{g} \in \{\pm 1\}$. 
If $g$ has constant curvature equal to $-\tfrac{\aR_{g}}{n(n+1)} g\kwedge g$, and $i(M)$ has mean curvature zero, then $(h, \Pi)$ solves the coupled projectively flat equations \eqref{projectivehiggsintro} with $c = -\ep$ and $\ka = \scal + \ep |\Pi|^{2} = \tfrac{n-1}{n+1}\aR_{g}$, where $\scal$ is the scalar curvature of $h$.
\end{theorem}
\begin{proof}
By definition, for $X, Y \in \Ga(TM)$, $\Pi(X, Y) = \ep h(SX, Y)$, where $S$ is the shape operator with respect to the normal $Z$. The Gauss-Codazzi equations and the assumption that $g$ have constant sectional curvature imply that $\cod(\Pi) = 0$ and that the curvature tensor of $g$ satisfies $\riem + \ep(\Pi \kwedge \Pi) = - \tfrac{\aR_{g}}{n(n+1)}(h \kwedge h)$. If the immersion has mean curvature zero, then $\tr_{h}(\Pi) = 0$, so $\div(\Pi) = 0$ and $\Pi \in \Ga(S^{2}_{0}(\ctm))\cap \ker \div \cap \ker \klie$. 
\end{proof}

In the case where $M$ is a spacelike mean curvature zero hypersurface in a $4$-dimensional Lorentzian spacetime solving the vacuum Einstein field equations, the equations $\div(\Pi) = 0$ and $\ka = \scal - |\Pi|^{2} = 0$ are the \emph{constraint equations} of general relativity \cite{Bartnik-Isenberg, Carlotto, Choquet-Burhat-Isenberg-Pollack, Corvino-Pollack, Isenberg-constant-mean-curvature}. They are the relaxation of \eqref{constraintintro} with $k = 2$, $c = 1$, and $\ka = 0$ in which only $\div(\Pi) = 0$, and not $\klie(\Pi) = 0$, is imposed. In particular, every solution of \eqref{stressenergyintro} on a $3$-manifold with $k = 2$, $c = 1$, and $\ka = 0$ solves the constraint equations, so their solvability is a necessary, though not sufficient, condition for the existence of such solutions.

\subsubsection{Examples from mean curvature zero Lagrangian immersions}\label{kahlerexample}
Let $(N, \Om)$ be a $2n$-dimensional symplectic manifold and let $i:M \to N$ be a Lagrangian immersion. Tensors on $M$ and $N$ are labeled with lowercase and uppercase Latin indices, respectively.
A \emph{(para/pseudo)-Kähler structure} is a triple $(G, A, \Omega)$ comprising a (pseudo-)Riemannian metric $G$, an integrable (para)complex structure, $A$, and a symplectic form $\Omega$, which are \emph{compatible} in the sense that,  for all $X, Y \in \Ga(TN)$, $A(AX) = \ep X$, $\Om(X, Y) = G(AX, Y)$, $\Om(AX, Y) = \ep G(X, Y)$, where, when $\ep = 1$, the qualifier \emph{para} applies, and, when $\ep = -1$, the qualifier \emph{pseudo} is used if $G$ has indefinite signature. 

By definition, a (para/pseudo)-Kähler manifold $(N, G, A)$ has \emph{constant (para)-holomorphic sectional curvature $4\hat{c}$} if its metric curvature tensor $\widehat{\riem}$ has the form
\begin{align}\label{chsc}
\widehat{\riem} = -\hat{c}( G \kwedge G + 3 \ep \Om \cdot \Om),
\end{align}
where $\kwedge$ and $\cdot$ are the bilinear maps defined in \eqref{kwedgedefined} and \eqref{kformphi}.
In this case, its Ricci and scalar curvatures equal $2(n+1)\hat{c}G$ and $4n(n+1)\hat{c}$, respectively. 

Let $(N, G, A, \Omega)$ be a $2n$-dimensional (para/pseudo)-Kähler manifold with canonical (Levi-Civita) connection $\hnabla$. An immersion $i:M \to N$ is \emph{nondegenerate} if the induced tensor $h = i^{\ast}(G)$ is nondegenerate. Let $i:M \to N$ be a nondegenerate Lagrangian immersion. The second fundamental form equals the projection of $\hnabla_{X}Ti(Y)$ onto the normal bundle of $M$ and determines the tensor $\Pi \in \Ga(S^{3}\ctm)$ defined by
\begin{align}\label{pidefined}
\Pi(X, Y, Z) = \Omega(\hnabla_{X}Ti(Y), Ti(Z))= G(A\hnabla_{X}Ti(Y), Ti(Z)) = - G(\hnabla_{X}Ti(Y), ATi(Z)). 
\end{align}
The tensor $\Pi$ is symmetric because $\hnabla$ is torsion-free and the immersion is Lagrangian. Define $\Pis \in \Ga(S^{2}\ctm\tensor TM)$ by $h(\Pis(X, Y), Z) = \Pi(X, Y, Z)$. Although $\Pis$ is not the second fundamental form as such, as it is twisted by the (para)-complex structure, it carries the same information as the second fundamental form. In particular, by the complete symmetry of $\Pi$, the Lagrangian immersion has mean curvature zero (by definition meaning $\tr_{h}(\Pi) = 0$) if and only if $\tr(\Pis) =0$. 

\begin{theorem}\label{constantsecttheorem}
Let $i:M \to N$ be a nondegenerate mean curvature zero Lagrangian immersion of the $n$-manifold $M$ in the $2n$-dimensional (para/pseudo-)Kähler manifold $(N, G, \Om, A)$ with constant (para)-holomorphic curvature $4\hat{c}$. Let $h = i^{\ast}(G)$ be the induced metric and let $\Pi \in \Ga(S^{3}\ctm)$ be the tensor defined from the second fundamental form of $i$ by \eqref{pidefined}. Then $(h, \Pi)$ solves the coupled projectively flat equations \eqref{projectivehiggsintro} with $c = \ep$ and $\ka = \hat{c}n(n-1)$.
\end{theorem}

\begin{proof}
Because $i$ is nondegenerate, $Ti(TM)$ and $ATi(TM)$ are transverse, so the projection onto $Ti(TM)$ along $ATi(TM)$, $P \in \Ga(\eno(i^{\ast}(TN)))$, is well-defined. The torsion-free connection $D$ on $M$ defined by $Ti(D_{X}Y) = P\hnabla_{X}Ti(Y)$ is the Levi-Civita connection of $h$.
By the definition of $\Pis$, for all $X, Y, Z \in \Ga(TM)$ there holds
\begin{align}
\begin{aligned}
\Om(\hnabla_{X}Ti(Y), Ti(Z)) &= 
h(\Pis(X, Y), Z) = G(Ti(\Pis(X, Y)), Ti(Z)) = \ep\Om(ATi(\Pis(X, Y)), Ti(Z)),
\end{aligned}
\end{align}
and by the nondegeneracy of $\Om$ this shows the first identity of 
\begin{align}\label{pksplitting}
\begin{aligned}
&\hnabla_{X}Ti(Y) = Ti(D_{X}Y) + \ep ATi(\Pis(X, Y)),&&
\hnabla_{X}ATi(Y) = ATi(D_{X}Y) + Ti(\Pis(X, Y)).
\end{aligned}
\end{align}
The second identity of \eqref{pksplitting} follows from the first because $A$ is $\hnabla$-parallel.

Let $\hat{R}_{IJK}\,^{L}$ be the curvature of $\hnabla$, so that $\widehat{\riem}_{IJKL} = \hat{R}_{IJK}\,^{P}G_{PL}$ is the curvature tensor of $G$, and let $\riem$ be the metric curvature tensor of $h$. 
Let $X, Y, Z, W \in \Ga(TM)$. Applying \eqref{pksplitting} twice yields
\begin{align}
\begin{aligned}
\hnabla_{X}\hnabla_{Y}Ti(Z) &= \hnabla_{X}\left(Ti(D_{Y}Z) + \ep ATi(\Pis(Y, Z))\right)\\
&= Ti\left(D_{X}D_{Y}Z + \ep \Pis(X, \Pis(Y, Z))\right) + \ep ATi\left(\Pis(X, D_{Y}Z) + D_{X}(\Pis(Y, Z))\right),
\end{aligned}
\end{align}
in which there has been used $\ep^{2} = 1$. Antisymmetrizing in $X$ and $Y$, subtracting $\hnabla_{[X, Y]}Ti(Z) = Ti(D_{[X, Y]}Z) + \ep ATi(\Pis([X, Y], Z))$, and using that $D$ is torsion-free yields
\begin{align}\label{gausscodazzilag}
\begin{aligned}
\hat{R}(X, Y)Ti(Z) &= Ti\left(\sR(X, Y)Z + \ep\left(\Pis(X, \Pis(Y, Z)) - \Pis(Y, \Pis(X, Z))\right)\right)\\
&\quad + \ep ATi\left((D_{X}\Pis)(Y, Z) - (D_{Y}\Pis)(X, Z)\right).
\end{aligned}
\end{align}
Consequently there are tensors $\sT$ and $\sN$ on $M$ such that
\begin{align}\label{stntdefined}
\hat{R}(X, Y)Ti(Z)= Ti(\sT^{\sharp}(X, Y) Z) + \ep ATi(\sN^{\sharp}(X, Y)Z),
\end{align}
where $h(\sT^{\sharp}(X, Y)Z, W) = \sT(X, Y, Z, W)$ and similarly for $\sN^{\sharp}$. Pairing \eqref{gausscodazzilag} with $W$ using $h$ and using $h(\Pis(U, V), W) = \Pi(U, V, W)$ and the complete symmetry of $\Pi$ yields
\begin{align}
\begin{aligned}
\sT(X, Y, Z, W) &= \riem(X, Y, Z, W) + \ep\left(h(\Pis(Y, Z), \Pis(X, W)) - h(\Pis(X, Z), \Pis(Y, W))\right),\\
\sN(X, Y, Z, W) &= (D_{X}\Pi)(Y, Z, W) - (D_{Y}\Pi)(X, Z, W),
\end{aligned}
\end{align}
or, in abstract index notation,
\begin{align}\label{kicindex}
&\sT_{ijkl} = \riem_{ijkl} - \ep\left(\Pi_{ik}\,^{p}\Pi_{jlp} - \Pi_{jk}\,^{p}\Pi_{ilp}\right),&
&\sN_{ijkl} = 2D_{[i}\Pi_{j]kl}.
\end{align}
By \eqref{kwedgedefined} with $k = 3$, $(\Pi \kwedge \Pi)_{ijkl} = \Pi_{k[i}\,^{p}\Pi_{j]lp} - \Pi_{l[i}\,^{p}\Pi_{j]kp} = \Pi_{ik}\,^{p}\Pi_{jlp} - \Pi_{jk}\,^{p}\Pi_{ilp}$, so \eqref{kicindex} is the pair of identities (the Gauss and Codazzi equations of the Lagrangian immersion)
\begin{align}\label{kic}
&\sT= \riem - \ep\, \Pi \kwedge \Pi,&
&\sN=  2\cod(\Pi).
\end{align}
In particular $\sT$ and $\sN$ have metric curvature tensor type. Because there is supposed $\tr_{h}(\Pi) = 0$, tracing \eqref{kic} yields, by \eqref{rictralbe},
\begin{align}
\label{gic2b}&\rictr(\sT)= \ric -\ep\rictr(\Pi \kwedge \Pi)  , & &\tr_{h}\rictr(\sT) = \scal - \ep|\Pi|^{2},&&\rictr(\sN) = \div(\Pi),&&\tr_{h}\rictr(\sN) = 0.
\end{align}
By \eqref{kic} and \eqref{gic2b}, $\sN = 0$ if and only if $\cod(\Pi) = 0$, and in this case $\div(\Pi) = 0$, so that $\Pi \in \ker \klie \cap \ker \div$.

The tensors $\sT$ and $\sN$ are recovered from $\widehat{\riem}$ by
\begin{align}\label{stsnpullback}
\begin{aligned}
&\sT(X, Y, Z, W) = \widehat{\riem}(Ti(X), Ti(Y), Ti(Z), Ti(W)),&\\
&\sN(X, Y, Z, W) = -\widehat{\riem}(Ti(X), Ti(Y), Ti(Z), ATi(W)),
\end{aligned}
\end{align}
the second because $G(ATi(U), ATi(W)) = -\ep h(U, W)$ for $U, W \in \Ga(TM)$. Suppose now that $G$ has constant (para)-holomorphic sectional curvature $4\hat{c}$, so that $\widehat{\riem}$ has the form \eqref{chsc}. Because $i$ is Lagrangian, $\Om(Ti(U), Ti(V)) = 0$ and $G(Ti(U), ATi(V)) = \Om(Ti(V), Ti(U)) = 0$ for all $U, V \in \Ga(TM)$. Every term of $(\Om \cdot \Om)(Ti(X), Ti(Y), Ti(Z), Ti(W))$ contains a factor of the first kind, and every term of $(G \kwedge G)(Ti(X), Ti(Y), Ti(Z), ATi(W))$ and of $(\Om \cdot \Om)(Ti(X), Ti(Y), Ti(Z), ATi(W))$ contains a factor of the first or second kind. Hence \eqref{chsc} and \eqref{stsnpullback} yield
\begin{align}\label{stsnconst}
\begin{aligned}
&\sT = -\hat{c}\, h\kwedge h,&&\sN = 0,&
 &\rictr(\sT) = \hat{c}(n-1)h, && \tr_{h}\rictr(\sT) = \hat{c}n(n-1). 
\end{aligned}
\end{align}					
Together \eqref{kic}, \eqref{gic2b}, and \eqref{stsnconst} show that $\Pi \in \ker \klie \cap \ker \div$ and that $(h, \Pi)$ satisfies
\begin{align}
&\riem  = -\hat{c}\,h \kwedge h + \ep\, \Pi \kwedge \Pi, & &\ric  = \hat{c}(n-1)h + \ep \rictr(\Pi \kwedge \Pi), & &\scal = \hat{c}n(n-1) + \ep |\Pi|^{2},
\end{align}
so that $(h, \Pi)$ solves \eqref{projectivehiggsintro} with $c = \ep$ and $\ka = \hat{c}n(n-1)$.
\end{proof}

\subsection{Einstein statistical structures, flat projective structures, and cubic Codazzi tensors}\label{statisticalsection}
The notions described next are special cases of more general ones introduced in \cite{Fox-ahs, Fox-2dahs, Fox-crm} and described more completely in \cite[Sections $8-10$]{Fox-conelike}.

A \emph{statistical structure} $(\nabla, h)$ is a pair comprising a pseudo-Riemannian metric $h$ and a torsion-free affine connection $\nabla$ satisfying $\nabla_{[i}h_{j]k} = 0$. This means that the \emph{cubic form} $\bt_{ijk} = \nabla_{i}h_{jk}$ is symmetric. The torsion-free \emph{conjugate} connection $\bnabla$ defined by $\bnabla = \nabla + h^{kp}\nabla_{i}h_{jp}$ satisfies $\bnabla_{i}h_{jk} = -\nabla_{i}h_{jk} = -\bt_{ijk}$, so also $\bnabla_{[i}h_{j]k} = 0$. Hence $(\bnabla, h)$ is a statistical structure with cubic form $\bar{\bt}_{ijk} = -\bt_{ijk}$. The conjugate connection of $\bnabla$ is $\nabla$, so conjugacy is an involution on the space of statistical structures. A statistical structure is self-conjugate if and only if $\nabla$ is the Levi-Civita connection of $h$. 

\begin{remark}
For background on statistical structures see \cite{Amari, Ay-Jost-Le-Schwachhofer}
Statistical structures are also called \emph{Codazzi structures} \cite{Shima} and are essentially the same as what I have called \emph{exact AH structures} \cite{Fox-2dahs}. The \emph{statistical} terminology originates in the work of Amari \cite{Amari} applying differential geometric methods to study parametric statistics. While from a purely geometric point of view the terminology may seem poorly motivated, it has become a popular way of referring to these objects.
\end{remark}

A statistical structure is \emph{special} if $\nabla_{i}\det h = 0$. Because the conjugate connection $\bnabla$ satisfies $\bnabla_{i}\det h = -\nabla_{i}\det h$, the conjugate of a special statistical structure is special. Define $\ell_{i} = \bt_{ip}\,^{p}$, so that $(\nabla, h)$ is special if and only if $\ell_{i} = 0$.

By definition, the curvature of a statistical structure $(\nabla, h)$ is the curvature $R_{ijk}\,^{l}$ of $\nabla$. Its Ricci curvature is $R_{ij} = \rictr(R)_{ij} = R_{pij}\,^{p}$ and its scalar curvature is $\statscal = h^{ij}R_{ij}$.
Given a statistical structure $(\nabla, h)$, the difference tensor of $\nabla$ and the Levi-Civita connection $D$ of $h$ is $\nabla - D = -\tfrac{1}{2}\bt_{ij}\,^{k}$. Similarly $\bnabla = D + \tfrac{1}{2}\bt_{ij}\,^{k}$. More generally, for $t \in \rea$, the connection $\nabla^{t} = D - t\bt_{ij}\,^{k}$ constitutes with $h$ a statistical structure having cubic form $2t\bt_{ijk}$ and for which the conjugate statistical structure is $(\nabla^{-t}, h)$.

Let $\sR_{ijk}\,^{l}$ be the curvature of $D$. The curvature $R^{t}_{ijkl}= R^{t}_{ijk}\,^{p}h_{pl}$ of $(\nabla^{t}, h)$ satisfies
\begin{align}\label{statt}
&\begin{aligned}
R^{t}_{ijkl} & = R^{t}_{ijk}\,^{a}h_{al} = h_{al}\left(\sR_{ijk}\,^{a} - 2tD_{[i}\bt_{j]k}\,^{a} + 2t^{2}\bt_{p[i}\,^{a}\bt_{j]k}\,^{p}\right)\\
& = \sR_{ijkl} - 2t\cod(\bt)_{ijkl} - t^{2}(\bt \kwedge \bt)_{ijkl},\end{aligned}\\
\label{stattric}
&\begin{aligned}
\rictr(R^{t})_{ij} & = \rictr(\sR)_{ij} - t(\div(\bt)_{ij} - D_{i}\ell_{j}) - t^{2}\rictr(\bt \kwedge \bt)_{ij},
\end{aligned}\\
\label{stattscal}
&\begin{aligned}
\statscal^{t} & = \scal - t^{2}(|\bt|^{2} - |\ell|^{2}),
\end{aligned}\\
\label{statteinstein}
&\begin{aligned}
R^{t}_{(ij)} - \tfrac{1}{2}\statscal^{t}h_{ij} & = \efr_{ij}  -  t(\div(\bt)_{ij} - D_{(i}\ell_{j)}) - t^{2}\left(\rictr(\bt \kwedge \bt)_{ij} - \tfrac{1}{2}(|\bt|^{2} - |\ell|^{2})h_{ij}\right),
\end{aligned}\\
\label{stattrictracefree}
&\begin{aligned}
R^{t}_{(ij)} - \tfrac{1}{n}\statscal^{t}h_{ij} & = \rictr(\sR)_{ij} - \tfrac{1}{n}\scal h_{ij} -t(\div(\bt)_{ij} - D_{(i}\ell_{j)})- t^{2}\left(\rictr(\bt \kwedge \bt)_{ij} - \tfrac{1}{n}(|\bt|^{2} - |\ell|^{2})h_{ij}\right),
\end{aligned}
\end{align}
where $n = \dim M$. From \eqref{statt} there follow
\begin{align}\label{conjugatecurv}
&R^{t}_{ijlk} = -R^{-t}_{ijkl},& & R^{t}_{ij(kl)} = -2t\cod(\bt)_{ijkl},\\
\label{conjugatericcitrace}
&R^{t}_{ip}\,^{p}\,_{j} = \rictr(R^{-t})_{ij}, && \rictr(R^{-t})_{(ij)} - \rictr(R^{t})_{(ij)} = 2t (\div(\bt)_{ij} - D_{(i}\ell_{j)}).
\end{align}

A statistical structure has \emph{self-conjugate} curvature if its curvature tensor equals the curvature tensor of the conjugate statistical structure. More generally, a tensor associated with a statistical structure is \emph{self-conjugate} if it equals the corresponding tensor associated with the conjugate statistical structure. For example, it follows from \eqref{stattscal} that the scalar trace of the curvature of a statistical structure is self-conjugate, and by \eqref{conjugatericcitrace} a statistical structure has self-conjugate symmetrized Ricci tensor if and only if $\div(\bt)=D_{(i}\ell_{j)}$; in particular a special statistical structure has self-conjugate Ricci tensor if and only if its cubic form is divergence free.

\begin{lemma}\label{selfconjugatecurvaturelemma}
The following are equivalent for a statistical structure $(\nabla, h)$:
\begin{enumerate}
\item It has self-conjugate curvature.
\item Its cubic form is a Codazzi tensor.
\item Its curvature tensor satisfies $R_{ij(kl)} = 0$, so has the symmetries of a metric curvature tensor.
\end{enumerate}
A special statistical structure $(\nabla, h)$ has self-conjugate Ricci curvature if and only if its cubic form is divergence free.
\end{lemma}
\begin{proof}
The equivalence of the three statements follows from \eqref{statt} and \eqref{conjugatecurv} with $t = \pm 1/2$.
It follows from \eqref{stattric} that a special statistical structure has self-conjugate Ricci curvature if and only if its cubic form is divergence free. 
\end{proof}

\begin{lemma}\label{statisticalbianchilemma}
For a statistical structure $(\nabla, h)$ on a manifold of dimension $n$ there holds
\begin{align}\label{statisticalbianchi}
 \tfrac{n-2}{n}d\statscal_{i} + \tfrac{1}{4}\ell^{p}d\ell_{pi} - h_{ia}h^{pq}(\nabla_{p}S_{q}\,^{a} + \bnabla_{p}\bar{S}_{q}\,^{a})= \bt^{abc}R_{i(abc)} = \tfrac{1}{2}\div(\stp(\bt))_{i} - \tfrac{1}{2}\bt_{i}\,^{pq}\div(\bt)_{pq},
\end{align}
in which $S_{ij} = R_{(ij)} - \tfrac{\statscal}{n}h_{ij}$ is the trace-free part of the symmetrized Ricci tensor of $(\nabla, h)$, and $\bar{S}_{ij}$ is the trace-free part of the symmetrized Ricci tensor of the conjugate statistical structure.
\end{lemma}

\begin{proof}
By definition $R_{ij} = \tfrac{1}{n}\statscal h_{ij} + \tfrac{1}{4}d\ell_{ij} + S_{ij}$. Differentiating this and 
antisymmetrizing the result yields
\begin{align}\label{conserve1}
2\nabla_{[i}R_{j]k} &  = \tfrac{2}{n}d\statscal_{[i}h_{j]k} - \tfrac{1}{4}\nabla_{k}d\ell_{ij} + 2\nabla_{[i}S_{j]k}.
\end{align}
Contracting \eqref{conserve1} with $h^{jk}$ yields
\begin{align}\label{conserve2}
\begin{aligned}
2h^{jk}\nabla_{[i}R_{j]k} &= \tfrac{n-1}{n}d\statscal_{i} + \tfrac{1}{4}h^{pq}\nabla_{p}d\ell_{qi}  + h^{pq}\nabla_{i}S_{pq} - h^{pq}\nabla_{p}S_{qi}\\
& =  \tfrac{n-1}{n}d\statscal_{i} + \tfrac{1}{4}h^{pq}\nabla_{p}d\ell_{qi}  -S_{pq}\nabla_{i} h^{pq} - h^{pq}(h_{ai}\nabla_{p}S_{q}\,^{a} + S_{q}\,^{a}\nabla_{p}h_{ai})\\
& =  \tfrac{n-1}{n}d\statscal_{i} + \tfrac{1}{4}h^{pq}\nabla_{p}d\ell_{qi}   - h^{pq}h_{ai}\nabla_{p}S_{q}\,^{a}.
\end{aligned}
\end{align}
in which the last inequality follows from $\nabla_{i}h^{jk} = -\bt_{i}\,^{jk}$ and that $S_{ij}$ is trace-free. 

By \eqref{conjugatericcitrace}, $R_{ip}\,^{p}\,_{j} = \bar{R}_{ij} = \tfrac{1}{n}\statscal h_{ij} - \tfrac{1}{4}d\ell_{ij} + \bar{S}_{ij}$.
Using this together with the differential Bianchi identity yields
\begin{align}\label{conserve3}
\begin{aligned}
2h^{jk}\nabla_{[i}R_{j]k} &= h^{jk}\nabla_{p}R_{ijk}\,^{p} = \nabla_{p}R_{ia}\,^{ap} - R_{ijk}\,^{p}\nabla_{p}h^{jk}\\
&= \tfrac{1}{n}d\statscal_{i} - \tfrac{1}{4}\nabla_{p}d\ell_{i}\,^{p} + \nabla_{p}\bar{S}_{i}\,^{p} + \bt^{abc}R_{ibca} \\
&= \tfrac{1}{n}d\statscal_{i} + \tfrac{1}{4}h^{pq}\nabla_{p}d\ell_{qi} - \tfrac{1}{4}\ell^{q}d\ell_{qi} + h_{ia}h^{pq}\bnabla_{p}\bar{S}_{q}\,^{a} + \bt^{abc}R_{iabc}.
\end{aligned}
\end{align}
in which the last equality follows from $\nabla_{p}\bar{S}_{i}\,^{p} = h^{pq}h_{ia}\bnabla_{p}\bar{S}_{q}\,^{a}$ and the complete symmetry of $L^{abc}$. Combining \eqref{conserve2} and \eqref{conserve3} yields the first equality of \eqref{statisticalbianchi}.

Differentiating $\stp(L)$ yields $\div(\stp(L))_{i} = \bt_{i}\,^{pq}\div(\bt)_{pq} -2 \bt^{abc}\cod(\bt)_{iabc}$. Together with \eqref{conjugatecurv} this yields the second equality of \eqref{statisticalbianchi}.
\end{proof}

The following definition is a special case of one made in \cite{Fox-conelike}, specializing a more generl definition from \cite{Fox-ahs, Fox-2dahs}.
A special statistical structure $(\nabla, h)$ is \emph{Einstein} if it satisfies the following two conditions:
\begin{itemize}
\item (\emph{Naive Einstein}) $\rictr(R)$ and $\rictr(\bar{R})$ are multiples of $h$.
\item (\emph{Conservation}) The scalar curvature $\tr_{h}\rictr(R) = \tr_{h}\rictr(\bar{R})$ is constant. 
\end{itemize}

\begin{lemma}\label{naivelemma}
On manifold of dimension $n > 2$, a special statistical structure with self-conjugate curvature that is naive Einstein is Einstein.
\end{lemma}
\begin{proof}
Two different proofs are available. By Lemma \ref{selfconjugatecurvaturelemma}, that a special statistical structure have self-conjugate curvature means its cubic form is trace-free Codazzi. By \eqref{statteinstein} in this case $R_{(ij)} - \tfrac{1}{2}\statscal h_{ij}  = \efr_{ij}  -  \tfrac{1}{4}(\rictr(\bt \kwedge \bt)_{ij} - \tfrac{1}{2}|\bt|^{2}h_{ij})$ and by Theorem \ref{divrictrtheorem} this is divergence free. If the special statistical structure is naive Einstein, this implies $(n-2)d\statscal_{i} = 0$. Alternatively, for a special statistical structure with self-conjugate curvature, \eqref{statisticalbianchi} implies $(n-2)d\statscal_{i}= n h_{ia}h^{pq}(\nabla_{p}S_{q}\,^{a} + \bnabla_{p}\bar{S}_{q}\,^{a})$, and the naive Einstein condition implies the vanishing of the last expression. In either case, $(n-2)d\statscal_{i} = 0$, so $\statscal$ is constant, because $n > 2$.
\end{proof}

A statistical structure $(\nabla, h)$ is \emph{projectively flat} if $\nabla$ is projectively flat. A statistical structure is \emph{conjugate projectively flat} if the conjugate statistical structure is projectively flat.

\begin{lemma}\label{ssslemma}
For a special statistical structure $(\nabla, h)$ on an $n$-manifold the following are equivalent:
\begin{enumerate}
\item\label{sss1} The curvature of $(\nabla, h)$ satisfies $R_{ijkl} = \al (h\kwedge h)_{ijkl}$ for a constant $\al \in \rea$.
\item\label{sss2} $(\nabla, h)$ is projectively flat and has self-conjugate curvature.
\item\label{sss3} $(\nabla, h)$ is projectively flat and Einstein.
\end{enumerate}
If there hold \eqref{sss1}-\eqref{sss3}, then $-n(n-1)\al = \statscal$. If $n > 2$, the conditions \eqref{sss1}-\eqref{sss3} are equivalent to \eqref{sss4}.
\begin{enumerate}
\setcounter{enumi}{3}
\item\label{sss4} $(\nabla, h)$ is projectively flat and conjugate projectively flat. 
\end{enumerate}
\end{lemma}

\begin{proof}
A special statisical structure has symmetric Ricci tensor and, because the conjugate statistical structure is also special, symmetric conjugate Ricci tensor. Consequently, its projective Weyl tensor is $B_{ijk}\,^{l} = R_{ijk}\,^{l} +2\delta_{[i}\,^{l}P_{j]k} - 2 \delta_{k}\,^{l}P_{[ij]}$ where $(1-n)P_{ij} = R_{ij} - \tfrac{2}{n+1}R_{[ij]} = R_{ij}$.

If there holds \eqref{sss1}, then, by \eqref{conjugatericcitrace}, $\bar{R}_{ijkl} = -R_{ijlk} = -\al (h\kwedge h)_{ijlk} = \al(h\kwedge h)_{ijkl} = R_{ijkl}$, so $(\nabla, h)$ has self-conjugate curvature. Tracing yields $R_{ij} = (1-n)\al h_{ij}$ and $\statscal = -n(n-1)\al$, so that $(\nabla, h)$ is Einstein. Consequently $-n(n-1)P_{ij} = \statscal h_{ij}$ and $B_{ijkl} = R_{ijkl} + 2h_{l[i}P_{j]k} = R_{ijkl} - \tfrac{2\statscal}{n(n-1)}h_{l[i}h_{j]k} = 0$. If $n > 2$ this shows $\nabla$ is projectively flat. If $n = 2$, then $2\nabla_{[i}P_{j]k} = -d\statscal_{[i}h_{j]k} = 0$ because $\al$ is constant by assumption. This shows the vanishing of the projective Cotton tensor and hence that $\nabla$ is projectively flat. The preceding shows \eqref{sss1} implies \eqref{sss2}. 

Let $(\nabla, h)$ be a projectively flat special statistical structure. That $\nabla$ be projectively flat means that $R_{ijk}\,^{l} = -2\delta_{[i}\,^{l}P_{j]k}$ where $(1-n)P_{ij} = R_{ij}$. Tracing this in $jk$ yields $(1-n)\bar{R}_{ij} = (1-n)R_{ip}\,^{p}\,_{j} = R_{ij} - \statscal h_{ij}$. If the curvature of $\nabla$ is self-conjugate it follows that $n\bar{R}_{ij} = nR_{ij} = \statscal  h_{ij}$, so that $(\nabla, h)$ is naive Einstein. If $n >2$ this implies $(\nabla, h)$ is Einstein by Lemma \ref{naivelemma}. If $n = 2$, that $(\nabla, h)$ be projectively flat means $0 = 2\nabla_{[i}P_{j]k} = -d\statscal_{[i}h_{j]k}$. Tracing this in $jk$ yields $d\statscal_{i} =0$, so that $\statscal$ is constant and $(\nabla, h)$ is Einstein. This shows that \eqref{sss2} implies \eqref{sss3}.

If $(\nabla, h)$ is projectively flat and Einstein, then, as in the preceding paragraph, $R_{ijkl} = \tfrac{2}{n-1}h_{l[i}R_{j]k}$. If $(\nabla, h)$ is naive Einstein, then this last expression equals $-\tfrac{\statscal}{n(n-1)}(h\kwedge h)_{ijkl}$. If $(\nabla, h)$ is moreover Einstein, then $\statscal$ is constant. This shows that \eqref{sss3} implies \eqref{sss1}.

For a special statistical structure, 
\begin{align}\label{stattproj}
B_{ijkl} = R_{ijkl} + 2h_{l[i}P_{j]k}=R_{ijkl} - \tfrac{2}{n-1}h_{l[i}R_{j]k}.
\end{align}
By \eqref{conjugatericcitrace} and \eqref{stattproj} applied to the conjugate special statistical structure,
\begin{align}\label{constattproj}
\bar{B}_{ijkl} = \bar{R}_{ijkl} - \tfrac{2}{n-1}h_{l[i} \bar{R}_{j]k} = -R_{ijlk} - \tfrac{2}{n-1}h_{l[i} \bar{R}_{j]k} = -B_{ijlk} - \tfrac{2}{n-1}(h_{k[i}R_{j]l}+h_{l[i} \bar{R}_{j]k}).
\end{align}
If there holds \eqref{sss3}, then the right-hand side of \eqref{constattproj} vanishes, so $\bar{B}_{ijkl}$ vanishes. If $n > 2$ this implies $(\nabla, h)$ is conjugate projectively flat. This shows \eqref{sss3} implies \eqref{sss4} if $n > 2$. If $(\nabla, h)$ is both projectively flat and conjugate projectively flat, then \eqref{constattproj} implies $h_{k[i}R_{j]l} + h_{l[i} \bar{R}_{j]k} = 0$. Tracing this in $il$ and in $jk$ and relabeling the results yields $(n-1)\bar{R}_{ij} + R_{ij} = \statscal h_{ij} = (n-1)R_{ij} + \bar{R}_{ij}$. If $n > 2$ this implies $\bar{R}_{ij} = R_{ij}$ and $nR_{ij} = \statscal h_{ij} = n\bar{R}_{ij}$, so that $(\nabla, h)$ is naive Einstein. In \eqref{stattproj} and \eqref{constattproj}, this shows $R_{ijkl} = -\tfrac{\statscal}{n(n-1)}(h\kwedge h)_{ijkl} = \bar{R}_{ijkl}$ and, because $n > 2$, by Lemma \ref{naivelemma}, this shows $\statscal$ is constant. This proves \eqref{sss4} implies \eqref{sss1} when $n > 2$. 
\end{proof}

\begin{definition}\label{constantstatisticaldefinition}
A special statistical structure has \emph{constant curvature} if it satisfies the equivalent conditions \eqref{sss1}-\eqref{sss3} of Lemma \ref{ssslemma}.
\end{definition}

Theorem \ref{ahstheorem} shows that for a special statistical structure $(\nabla, h)$ with cubic form $\bt$ and self-conjugate curvature the hierarchy of coupled equations for $(h, \bt)$ corresponds with a geometrical natural hierarchy of equations for $(\nabla, h)$.

\begin{theorem}\label{ahstheorem}
Let $(\nabla, h)$ be a special statistical structure with cubic form $\bt$ and scalar curvature $\statscal$ on a manifold of dimension $n \geq 2$.
\begin{enumerate}
\item\label{sse1} If $n \geq 3$, $(\nabla, h)$ has constant curvature if and only if $(h, \bt)$ solves the coupled projectively flat equations \eqref{projectivehiggsintro} with constants $c = \tfrac{1}{4}$ and $\ka = \statscal$;
\item\label{sse2} If $n \geq 3$, $(\nabla, h)$ has self-conjugate curvature and is Einstein if and only if $(h, \bt)$ solves the coupled Einstein equations \eqref{stressenergyintro} with constants $c = \tfrac{1}{4}$ and $\ka = \statscal$;
\item\label{sse3} $(\nabla, h)$ has self-conjugate curvature and constant scalar curvature if and only if $(h, \bt)$ solves the coupled constraint equations \eqref{constraintintro} with constants $c = \tfrac{1}{4}$ and $\ka = \statscal$.
\end{enumerate}
\end{theorem}

\begin{proof}
Because the cubic form of a special statistical structure is trace-free, by Lemma \ref{selfconjugatecurvaturelemma}, the cubic form of a special statistical structure is a trace-free Codazzi tensor if and only if the special statistical structure has self-conjugate curvature. 

By the preceding and \eqref{statt}-\eqref{stattscal}, for a special statistical structure with self-conjugate curvature:
\begin{align}\label{stattsse}
&\sR_{ijkl} - \tfrac{1}{4}(\bt \kwedge \bt)_{ijkl} = R_{ijkl}, & &\sR_{ij} - \tfrac{1}{4}\rictr(\bt \kwedge \bt)_{ij} = R_{ij},& &\scal - \tfrac{1}{4}|\bt|^{2}= \statscal.
\end{align}
Claim \eqref{sse3} is immediate from \eqref{stattsse} and \eqref{constraintintro}. For claim \eqref{sse2}, if $(\nabla, h)$ has self-conjugate curvature and is Einstein, then $\bt$ is trace-free Codazzi, $R_{ij} = \tfrac{\statscal}{n}h_{ij}$, and $\statscal$ is constant, so by \eqref{stattsse} $(h, \bt)$ solves \eqref{stressenergyintrob}, and hence \eqref{stressenergyintro}, with $c = \tfrac{1}{4}$ and $\ka = \statscal$. Conversely, if $(h, \bt)$ solves \eqref{stressenergyintro} with $c = \tfrac{1}{4}$ and $\ka = \statscal$, then $\bt$ is trace-free Codazzi, so $(\nabla, h)$ has self-conjugate curvature; since $n \geq 3$, $\ka$ is constant by the remarks following \eqref{constraintintro}, and, by \eqref{stressenergyintrob} and \eqref{stattsse}, $R_{ij} = \tfrac{\statscal}{n}h_{ij}$ with $\statscal = \ka$ constant, so $(\nabla, h)$ is Einstein.

By Lemma \ref{ssslemma}, if a special statistical structure having self-conjugate curvature is projectively flat, then its curvature satisfies $R_{ijkl} = -\tfrac{\statscal}{n(n-1)}(h\kwedge h)_{ijkl}$ with $\statscal$ constant, so with \eqref{stattsse} this shows $(h, \bt)$ solves \eqref{projectivehiggsintro} with $c = \tfrac{1}{4}$ and $\ka = \statscal$. Conversely, if $(h, \bt)$ solves \eqref{projectivehiggsintro} with $c = \tfrac{1}{4}$ and $\ka = \statscal$, then $\bt$ is trace-free Codazzi, so $(\nabla, h)$ has self-conjugate curvature; since $n \geq 3$, $\ka = \statscal$ is constant by Lemma \ref{projectivehiggslemma}, and \eqref{stattsse} shows $R_{ijkl} = -\tfrac{\statscal}{n(n-1)}(h\kwedge h)_{ijkl}$, so $(\nabla, h)$ satisfies \eqref{sss1} of Lemma \ref{ssslemma} and has constant curvature.
\end{proof}

\begin{remark}
When $n \geq 3$ the constancy of $\ka$ in \eqref{sse1} and \eqref{sse2} is automatic, by Lemma \ref{projectivehiggslemma} and the remarks following \eqref{constraintintro}, but when $n = 2$ it is not. Indeed, if $n = 2$, then every $\sY \in \Ga(\mcurv(\ctm))$ satisfies $\sY = -\tfrac{1}{2}\scal(\sY)h \kwedge h$, and, because a trace-free symmetric tensor on a surface has only components of types $(k, 0)$ and $(0, k)$, every $\om \in \Ga(\symkt)$ satisfies $\rictr(\om \kwedge \om) = \tfrac{1}{2}|\om|^{2}h$, so that $\om \kwedge \om = -\tfrac{1}{2}|\om|^{2}h\kwedge h$ and $\stp(\om) = 0$. Consequently, for a trace-free Codazzi tensor $\om$ the curvature equation of \eqref{projectivehiggsintro} holds with $\ka = \scal - c|\om|^{2}$, a function that need not be constant, while the curvature equation of \eqref{stressenergyintro} reduces to $0 = 0$. On the other hand, by \eqref{stattsse}, a special statistical structure with self-conjugate curvature on a surface satisfies $R_{ijkl} = -\tfrac{\statscal}{2}(h \kwedge h)_{ijkl}$ and $R_{ij} = \tfrac{\statscal}{2}h_{ij}$, so it has constant curvature, equivalently is Einstein, if and only if $\statscal$ is constant. For example, if $h$ is a metric on a surface with nonconstant Gaussian curvature and $\bt = 0$, then $(h, 0)$ solves \eqref{projectivehiggsintro} with the nonconstant function $\ka = \scal$, and solves \eqref{stressenergyintro}, but $(D, h)$ is neither Einstein nor of constant curvature. Hence the hypothesis $n \geq 3$ in \eqref{sse1} and \eqref{sse2} cannot be dropped; when $n = 2$ they remain true provided $\ka$ is required to be constant, in which case each of \eqref{sse1}-\eqref{sse3} is equivalent to the requirement that $\bt$ be trace-free Codazzi and $\statscal$ be constant.
\end{remark}

The following paragraph summarizes the approach to the affine differential geometry of hypersurfaces  detailed in \cite{Fox-ahs, Fox-2dahs}.
An immersed hypersurface $M$ in $(n+1)$-dimensional flat affine space $\aff$ is nondegenerate if its second fundamental form is nondegenerate. Equivalently the conormal Gauss map $\cng:M \to \proj(\aff^{\ast})$ that associates with $p \in M$ the annihilator of the tangent space $T_{p}M$, viewed as an element of the projectivization $\proj(\aff^{\ast})$ of the vector space $\aff^{\ast}$ dual to $\aff$, is an immersion. In this case, the pullback via $\cng$ of the flat projective structure on $\proj(\aff^{\ast})$ induces on $M$ a flat projective structure $\ben$. A coorientation of $M$ means an orientation of its normal bundle and such identifies the second fundamental form of $M$ with a conformal structure $[h]$ on $M$. Any choice of vector field $W$ transverse to $M$ and consistent with the given coorientation determines a symmetric two tensor $h$ representing the second fundamental form with respect to the trivialization of the normal bundle determined by $W$, and $h$ is a metric by the assumed nondegeneracy of $M$. The metrics corresponding with different choices of cooriented $W$ are conformal. Each choice of $W$ determines a torsion-free affine connection $\nabla$ on $M$ by projecting the flat affine connection on $\aff$ along $W$. The resulting $\nabla$ is unchanged if $W$ is rescaled, but varies when the direction of $W$ is changed. The affine normal line bundle is the line field transverse to $M$ determined by the requirement that for any local section $W$ the induced $\nabla$ satisfying $h^{pq}\nabla_{i}h_{pq} = nh^{pq}\nabla_{p}h_{qi}$ with respect to any $h \in [h]$. If a parallel volume form $\Psi$ on $\aff$ is fixed, a unique cooriented transversal $W$ called the affine normal vector field is determined by requiring that it span the affine normal line bundle and that the volume densities on $M$ determined by $\imt(W)\Psi$ and $|\det h|^{1/2}$ agree, where $h \in [h]$ corresponds with $W$. The resulting $h$ is called the equiaffine or Blaschke metric of $M$. Because $\Psi$ is determined only up to multiplication by a constant by the requirement that it be parallel, $h$ is determined only up to positive homothety. Together $h$ and the connection $\nabla$ induced via the affine normal $W$ constitute a special statistical structure $(\nabla, h)$. The conjugate special statistical structure $(\bnabla, h)$ is such that $\bnabla$ represents the flat projective structure $\ben$ induced on $M$ via the conormal Gauss map.

In summary, a cooriented nondegenerate immersed hypersurface $M$ in flat affine space acquires a pair of conjugate special statistical structures one of which is necessarily projectively flat. 

The shape operator of a cooriented nondegenerate immersed hypersurface $M$ in flat affine space means the shape operator with respect to its affine normal $W$. Such a hypersurface is an \emph{affine sphere} if its affine normal line bundles all meet in a point (its center) or are all parallel (center at infinity). In the first case the affine sphere is \emph{proper}, in the second it is \emph{improper}. Equivalently, its affine shape operator is a (constant) multiple of the identity. It is an elementary matter to check that this is the case if and only if one of, and hence both, the induced special statistical structures on $M$ have constant curvature in the sense of Definition \ref{constantstatisticaldefinition}; in particular they determine a solution of the coupled projectively flat equations, as was observed with different terminology in \cite{Fox-ahs}; see also \cite{Fox-2dahs, Fox-schwarz}.  
	
These claims make sense whatever the signature of $[h]$, but more can be said when $[h]$ has definite signature. When $M$ is locally strictly convex so that $[h]$ has Riemannian signature, then a proper affine sphere is \emph{elliptic} or \emph{hyperbolic} as its center lies on the side to which $W$ points or on the side away from which $W$ points. In this setting improper affine spheres are called \emph{parabolic}. 
If the equiaffine metric of an elliptic affine sphere is complete, then the affine sphere is an ellipsoid, its cubic form vanishes, and its equiaffine metric has constant positive sectional curvature equal to its affine mean curvature \cite{Calabi}.
If the equiaffine metric of a parabolic affine sphere is complete, then the affine sphere is an elliptic paraboloid and its equiaffine metric is flat \cite{Jorgens, Calabi-improper}. If the equiaffine metric of a hyperbolic affine sphere is complete, then it has nonpositive Ricci curvature \cite[Theorem $5.1$]{Calabi}, and the cone over the boundary of the closure of its image in projective space is a pointed convex cone with vertex at the center of the affine sphere and to which the affine sphere is asymptotic, and all hyperbolic affine spheres with complete equiaffine metric arise in this way. These last affirmations summarize deep results due to Cheng and Yau, in particular in \cite{Cheng-Yau-affinehyperspheresI}; they are described in the form used here in the articles of J. Loftin \cite{Loftin-affine, Loftin-affinekahler, Loftin-survey} and further variants and discussion can be found in \cite{Fox-schwarz, Jia-Li-Simon-Xu, Klartag-elliptic}.

In general signature there are not available similarly general existence and characterization results but many examples are known. See for example \cite{Fox-prehom, Fox-afflsa, Hildebrand-analyticformulas,  Hildebrand-parallelcubicform}.

\begin{theorem}\label{properlyconvextheorem}
On an $n$-manifold $M$ that carries a properly convex flat real projective structure $\en$, there is a complete Riemannian metric $h$, unique up to homothety, that with the unique $\nabla \in \en$ satisfying $\nabla \det h =0$, constitutes a special statistical structure having constant curvature in the sense of Definition \ref{constantstatisticaldefinition}. In particular, the pair $(h, \bt)$ determined by $D - \nabla = \tfrac{1}{2}\bt_{ijp}h^{kp}$, where $D$ is the Levi-Civita connection of $h$, solves the coupled projectively flat equations \eqref{projectivehiggsintro} with $c = \tfrac{1}{4}$ and a nonpositive constant $\ka$. Moreover, the metric $h$ has nonpositive Ricci curvature.
\end{theorem}

\begin{proof}
The following sketch reformulates work of J. Loftin summarized in \cite{Loftin-affine, Loftin-survey} based on earlier work of E. Calabi and S.~T. Yau and S.~Y. Cheng.
Let $\ste$ be an $n$-dimensional vector space. That $M$ be a properly convex flat real projective $n$-manifold means that it is the quotient of a properly convex domain $\Om$ (having nonempty interior) in flat projective space $\proj(\ste)$ by a free and properly discontinuous action of a group $\Ga \subset PGL(\ste)$ on $\Om$. That $\Om$ be properly convex means that it is convex and its closure is contained in some affine chart; in particular it contains no projective line. The cone over $\Om$ means the union $\hat{\Om}= \cup_{p \in \Om}\pi^{-1}(p)$ where $\pi:\ste \setminus \{0\} \to \proj(\ste)$ is the defining projection. The closure of either of the two connected components of $\hat{\Om}$ is a pointed convex cone in $\ste$ with vertex at the origin. By a theorem of S.~T. Yau and S.~Y. Cheng there is a unique foliation of its interior by complete hyperbolic affine spheres. Any one of these affine spheres is diffeomorphic to $\Om$ via $\pi$ and their equiaffine metrics with respect to a parallel volume are all positively homothetic, so they determine on $\Om$ a positive homothety class of complete Riemannian metrics. The connections determined on these affine spheres via their affine normals all descend to the same connection on $\Om$ and with any one of these metrics it forms a special statistical structure having constant curvature. That these metrics have nonpositive Ricci curvature follows from \cite[Theorem $5.1$]{Calabi}
\end{proof}

\subsection{Solutions from invariant polynomials on compact simple Lie algebras}
There seems to be known no general existence result applicable to the coupled Einstein equations \eqref{stressenergyintro} for the case $k > 3$. Here there are given some specific examples based on algebraic constructions.

\begin{example}
Let $G$ be a connected compact simple Lie group of dimension greater than $3$ with Lie algebra $\g$, let $h = -B$ be the biinvariant metric determined by the negative of the Killing form of $\g$, and note that $h$ is Einstein with Ricci curvature $\tfrac{1}{4}h$.

Let $\om_{i_{1}\dots i_{k}} \in S^{k}(\g^{\ast})$ be the complete polarization of a homogeneous degree $k$ polynomial $P$ invariant under the adjoint action of $G$ on $\g$. More precisely, suppose that $P$ is one of a set of homogeneous generators of the ring of invariant polynomials on $\g$ and that $k \geq 3$. If $\g$ has rank $l$ then the ring of invariant polynomials on $\g$ is generated by $l$ algebraically independent homogeneous elements \cite[section VIII.$8$]{Bourbaki-lie}. The degrees $2 = u_{1} < \dots <  u_{l}$ of the homogeneous generators are given in terms of the exponents $m_{1}<  \dots < m_{l}$ of the Weyl group $W$ of $\g$ by $u_{i} = m_{i} + 1$, and satisfy $u_{1}\cdot\dots \cdot u_{l} = |W|$ and $2(u_{1} + \dots + u_{l}) = \dim \g + l$ (always, $u_{1} = 2$).
Since $G$ acts on $\g$ orthogonally, it fixes the invariant quadratic polynomial $E$ corresponding to $h$ and commutes with the $h$-Laplacian $\lap_{h}$. Consequently, if $P = \sum_{0 \leq s \leq k/2}E^{s}H_{k-2s}$ is the unique decomposition of $P$ with each $H_{k-2s}$ harmonic and homogeneous of degree $k - 2s$, then, by uniqueness, each $H_{k-2s}$ is $G$-invariant. In particular, the harmonic part $Q = H_{k}$ of $P$ is $G$-invariant. Were $Q = 0$, there would hold $P = E\cdot R$ with $R = \sum_{s \geq 1}E^{s-1}H_{k-2s}$ an invariant polynomial homogeneous of degree $k-2$. Because $\g$ is simple, $E$ is a nonzero multiple of the degree $2$ generator, and, because $k - 2 < k$, $R$ is a polynomial in the homogeneous generators of degree at most $k-2$, none of which is $P$. The identity $P = E\cdot R$ would then be a nontrivial polynomial relation among the homogeneous generators, linear in $P$, contradicting their algebraic independence. Hence $Q \neq 0$.

It follows that it may be supposed from the beginning that $P$ is $\lap_{h}$-harmonic, or, equivalently, that $\om_{i_{1}\dots i_{k}}$ is trace-free. In particular, this shows that on $\g$ there exists a $\lap_{h}$-harmonic homogeneous $G$-invariant polynomial of degree at least $3$.

\begin{theorem}\label{stressenergyexampletheorem}
Let $G$ be a connected compact simple Lie group of dimension greater than $3$ with Lie algebra $\g$ and let $h = -B$ be the biinvariant metric on $G$ determined by the negative of the Killing form $B$ of $\g$. Suppose $k \geq 3$ and let $\om_{i_{1}\dots i_{k}} \in S^{k}(\g^{\ast})$ be the extension as a left-invariant tensor of the complete polarization of a $\lap_{h}$-harmonic homogeneous $G$-invariant polynomial $P$ of degree $k$.
The pair $(h, \om)$ solves the coupled Einstein equations \eqref{stressenergyintro} on $G$. 
\end{theorem}
\begin{proof}
Let $D$ be the Levi-Civita connection of $h = -B$ and let $c_{ij}\,^{k}$ be the structure tensor of the Lie algebra (so $x^{i}y^{j}c_{ij}\,^{k} = [x, y]^{k}$).
The invariance of $P$ means that $0 = kc_{i(i_{1}}\,^{p}\om_{i_{2}\dots i_{k})p} = -2D_{i}\om_{i_{1}\dots i_{k}}$, so that $\om$ is parallel, and so in the kernel of $\klie$.
Because $\si_{ij} = \om_{ii_{1}\dots i_{k-1}}\om_{j}\,^{i_{1}\dots i_{k-1}}$ is a contraction of invariant tensors it is an invariant bilinear form.
By the simplicity of $\g$ there is a constant $\ka$ such that $\si_{ij} = \ka h_{ij}$, and tracing this equality shows that $\ka = |\om|^{2}/\dim \g$. It follows that $(h, \om)$ solves the coupled Einstein equations \eqref{stressenergyintro}.
\end{proof}
One expects that in general, and in some cases it can be proved, that the solutions of the coupled Einstein equations \eqref{stressenergyintro} obtained via Theorem \ref{stressenergyexampletheorem} are not solutions of the coupled projectively flat equations \eqref{projectivehiggsintro}. The simplest case in which this is easily shown is the following with $G = SU(m)$ for $m > 2$. Identify the Lie algebra $\g = \mathfrak{su}(m)$ with the real vector space of trace-free $m\times m$ anti-Hermitian matrices. The Killing form $B$ is $B(X, Y) = 2m \tr(XY)$ where $\tr$ is the usual matrix trace. Let $h$ be the biinvariant metric determined on $G$ by $-B$. The commutative product $X \mlt Y = \j(XY + YX - \tfrac{2}{m}\tr(XY)I)$ determines a completely symmetric trilinear form $\chi(X, Y, Z) = h(X\mlt Y, Z) = -2m\j\tr(XYZ + YXZ)$ with associated cubic form $P(X) = -4m\j \tr(X^{3})$ that is evidently invariant under the adjoint action of $G$. The reason for assuming $m > 2$ is that $P(X)$ vanishes identically when $m = 2$. On the other hand, when $m > 2$, $P$ is nontrivial.

By Theorem \ref{stressenergyexampletheorem}, for each $t \in \rea$ the pair $(h, 2t\chi)$ solves the coupled Einstein equations \eqref{stressenergyintro} on $G$. Define $L:\mathfrak{su}(m) \to \eno(\mathfrak{su}(m))$ by $L(X)Y = X\mlt Y$. The pair $(h, 2t\chi)$ determines the special statistical structure $(\nabla, h)$ with $\nabla^{t} = D - tL$ where $D$ is the Levi-Civita connection of $h$. Explicitly, 
\begin{align}\label{sunconnect}
\nabla^{t}_{X}Y = (\tfrac{1}{2}\ad(X) - tL(X))Y. 
\end{align}
The curvature $R^{t}(X, Y)Z$ of $\nabla^{t}$ can be computed using the invariance of $L$, $[\ad(X), L(Y)] = L([X, Y])$, and that 
\begin{align}
[L(X), L(Y)]Z = X\mlt (Y\mlt Z) -Y \mlt (X\mlt Z)  =  -\ad([X, Y])Z + \tfrac{2}{m^{2}}(h(X,Z)Y - h(Y, Z)X), 
\end{align}
and the result is
\begin{align}\label{suncurvature}
\begin{aligned}
R^{t}(X, Y)Z &= -\tfrac{1}{4}\ad([X, Y])Z + t^{2}[L(X), L(Y)]Z\\
& = -(\tfrac{1}{4} + t^{2})[[X, Y], Z] + \tfrac{2t^{2}}{m^{2}}( h(X, Z)Y - h(Y, Z)X).
\end{aligned}
\end{align}
so that 
\begin{align}
h(R^{t}(X, Y)Z, W) =  -(\tfrac{1}{4} + t^{2})h([X, Y], [Z, W])- \tfrac{2t^{2}}{m^{2}}(h(Y, Z)h(X,W) - h(X, Z)h(Y, W))
\end{align}
from which it is apparent that the statistical structure $(\nabla, h)$ does not have constant curvature, so, by Theorem \ref{ahstheorem}, $(h, 2t\chi)$ does not solve the coupled projectively flat equations \eqref{projectivehiggsintro}. The special statistical structures $(\nabla^{\pm t}, h)$ are conjugate and by \eqref{suncurvature} they have self-conjugate curvature.
Because
\begin{align}
R^{t}(\dum, X)Y = -(\tfrac{1}{4} + t^{2})\ad(Y)\ad(X)  + \tfrac{2t^{2}}{m^{2}}(X\tensor h(Y, \dum) - h(X, Y)\Id_{\mathfrak{su}(m)}),
\end{align}
the Ricci curvature of $\nabla^{t}$ is
\begin{align}
\ric^{t}(X, Y) = \tr R^{t}(\dum, X)Y = (\tfrac{1}{4} + (\tfrac{4}{m^{2}} - 1)t^{2})h(X, Y),
\end{align}
which is positive if $t^{2}< \tfrac{m^{2}}{4(m^{2} - 4)}$, zero when $t^{2} = \tfrac{m^{2}}{4(m^{2} - 4)}$, and negative otherwise. 

The one-parameter family of biinvariant affine connections $\nabla^{t}$ was found by T. Laquer who determined all biinvariant affine connections on compact Lie groups \cite{Laquer}. Laquer proved that the space of biinvariant affine connections on compact simple real Lie group $G$ with Lie algebra $\g$ is one-dimensional except when $\g = \mathfrak{su}(m)$ and $m \geq 3$, in which case there is a two-dimensional family of biinvariant connections (the case $\mathfrak{so}(6)$ is isomorphic to $\mathfrak{su}(4)$). Among these connections the torsion-free ones are exactly the connections $\nabla^{t}$. The geometric significance of these connections was not apparent from Laquer's results. The preceding results show that $(\nabla^{t}, h)$ is a one-parameter family of Riemannian signature Einstein special statistical structures having self-conjugate curvature but not having constant curvature. Moreover they exist on a compact manifold. In particular, these examples are not locally equivalent to the special statistical structures induced on an affine hypersurface because they are neither projectively nor conjugate projectively flat.

\begin{corollary}\label{suncharcorollary}
Let $G$ be a connected compact simple Lie group with Lie algebra $\g$ of dimension at least $3$. Suppose that $\nabla$ is a left-invariant affine connection that together with the Riemannian metric $h=-B$ determined by the negative of the Killing form generates a left-invariant Einstein special statistical structure having self-conjugate curvature. Then either $\nabla$ is the Levi-Civita connection of $h$, or $\g$ has type $A_{l}$ for some $l \geq 2$ and $(\nabla, h)$ is locally isomorphic to one of the one-parameter family of biinvariant Einstein special statistical structures $(\nabla^{t}, h)$ on $SU(m)$ defined by \eqref{sunconnect}.
\end{corollary}

In interpreting the conclusion of Corollary \ref{suncharcorollary}, recall that Dynkin diagrams of types $A_{3}$ and $D_{3}$ coincide, corresponding to the fact that $SU(4)$ is the connected double cover of $SO(6)$. 

\begin{proof}
Let $\chi \in S^{3}(\g^{\ast})$ be the cubic form of a left-invariant Einstein special statistical structure having self-conjugate curvature, $(\nabla, h)$. By Theorem \ref{ahstheorem}, $(h, \chi)$ solves the coupled Einstein equations \eqref{stressenergyintro}, so $\div(\chi) = 0$ and $\cod(\chi) = 0$. On the other hand, $-2D_{i}\chi_{jkl} = 3c_{i(j}^{p}\chi_{kl)p}$, so $D_{(i}\chi_{jkl)} = 0$. Hence $3c_{i(j}^{p}\chi_{kl)p} = -2D_{i}\chi_{jkl} = 0$, which shows the invariance of the cubic polynomial $P$ on $\g$ whose complete polarization is $\chi_{ijk}$. The only cases in which there is an invariant homogeneous polynomial of degree $u_{2} = 3$ are when $\g$ has type $A_{l}$ and $l \geq 2$ and when $\g$ has type $D_{3}$ (which is redundant since $D_{3} = A_{3}$). Note that it follows as well that in the cases where there is a nonzero invariant cubic polynomial on $\g$ it is unique up to scale. It also must be trace-free, for its trace corresponds to the $h$-Laplacian of $P$; since $G$ acts on $\g$ orthogonally and the Laplacian is orthogonally invariant, $\lap_{h}P$ is invariant, but there are no nonzero invariant homogeneous linear forms on $\g$, so $\lap_{h}P = 0$. Similarly, the invariance of $\chi_{ijk}$ implies the invariance of $\chi_{ip}\,^{q}\chi_{jq}\,^{p}$, and the simplicity of $\g$ then forces that $\chi_{ip}\,^{q}\chi_{jq}\,^{p}$ is a multiple of $h_{ij}$. That $(\nabla, h)$ must be locally isomorphic to one of the examples on $SU(m)$ defined by \eqref{sunconnect} follows from the uniqueness up to scale of the invariant cubic polynomial $P$. 
\end{proof}
The proof of Corollary \ref{suncharcorollary} does not use the classification of \cite{Laquer}; the result of \cite{Laquer} is used only to identify the connections $\nabla^{t}$ as the torsion-free biinvariant affine connections on $SU(m)$. What is new here is the interpretation of the pairs $(\nabla^{t}, h)$ as Einstein special statistical structures with self-conjugate curvature, which solve the coupled Einstein equations \eqref{stressenergyintro} but not the coupled projectively flat equations \eqref{projectivehiggsintro}, and the characterization of Corollary \ref{suncharcorollary} among left-invariant special statistical structures.

In \cite{Laquer-leftinvariant}, Laquer determined the space of left-invariant affine connections on a compact irreducible Riemannian symmetric space. Examples like that just described occur for the spaces $SU(m)/SO(m)$, $(SU(m)\times SU(m))/SU(m)$, $SU(2m)/Sp(m)$, all with $m \geq 3$, and $E_{6}/F_{4}$. The case $(SU(m)\times SU(m))/SU(m)$ is that already described. The connections are given essentially as in the preceding example, with the Levi-Civita connection of the underlying Riemannian symmetric space perturbed by adding to it a multiple of a commutative product that can be described as the trace-free Jordan part of the Jordan product on the algebra of $m \times m$ Hermitian matrices over one of the real Hurwitz algebras, $\rea$, $\com$, $\quat$, or $\cayley$ (in the last case, of the octonions, $m = 3$). These yield further examples of compact Riemannian signature solutions of the coupled Einstein equations that are not coupled projectively flat. These same connections arise in the classification of totally real parallel submanifolds in complex projective space obtained by H. Naitoh \cite{Naitoh}; in that context they arise as in Example \ref{kahlerexample}. An argument like that proving Corollary \ref{suncharcorollary} will show that these are essentially the unique left-invariant examples. The essence of the proof of this claim is contained in the proof of \cite[Theorem $4.5$]{Naitoh}. Because the details are formally quite similar to those given above, but writing them would take considerable space, they are omitted. 
\end{example}

\subsection{Solutions of other algebraic origins}
When the metric $h$ is flat, the coupled Einstein equations \eqref{stressenergyintro} admit purely algebraic solutions. Inwhat follows, $h$ is a flat Riemannian metric on an $n$-dimensional vector space $\ste$, although many of the claims make sense in other signatures. Let $x^{i}$ denote the radial Euler vector field. Let $E(x) = |x|^{2}$. Let $\lap = \lap_{h}$. If $F(x)$ is a function on $\ste$, let $F_{i_{1}\dots i_{k}} = D_{i_{1}}\dots D_{i_{k}}F$. Sometimes it is convenient to write $D^{(k)}F$ for the $k$-fold covariant derivative of $F$ (for example $D^{(2)}F$ is the Hessian of $F$). If $F$ is a polynomial homogeneous of degree $g$ then $x^{i_{1}}\dots x^{i_{j}}F_{i_{1}\dots i_{j}} = g(g-1)\dots(g-j+1)F$. In this case $F_{i_{1}\dots i_{g}}$ is a constant tensor, so parallel.

\begin{lemma}\label{epequivalencelemma}
Let $F$ be an $h$-harmonic polynomial homogeneous of degree $g\geq 2$ on the $n$-dimensional Euclidean vector space $(\ste, h)$. Let $\om_{i_{1}\dots i_{g}} = F_{i_{1}\dots i_{g}}$. The pair $(h, \om)$ solves the coupled Einstein equations \eqref{stressenergyintro} if and only if there is $c \in \rea$ such that $F$ solves any one of the following equivalent equations:
\begin{align}\label{sepoly}
&0   = F_{ip_{1}\dots p_{g-1}}F_{j}\,^{p_{1}\dots p_{g-1}} - ch_{ij},&
&0  = D_{i}D_{j}\left(|D^{(g-1)}F|^{2} - cE\right),&
&0  = |D^{(g-1)}F|^{2} - cE.
\end{align}
In this case $nc = |D^{(g)}F|^{2}$.
\end{lemma}
\begin{proof}
The equivalence of the first two equations of \eqref{sepoly} follows from the vanishing of $D^{(g+1)}F$. Because $|D^{(g-1)}F|^{2} - cE$ is a homogeneous quadratic polynomial, its Hessian vanishes if and only if it vanishes identically, so these equations are equivalent to the last equation of \eqref{sepoly}.
\end{proof}

\begin{example}
A hypersurface in a Riemannian space form is \emph{isoparametric} if its principal curvatures are constant. The question of classifying isoparametric hypersurfaces was posed and partially solved by E. Cartan in \cite{Cartan-cubic, Cartan-isoparametric, Cartan-isoparametricconstantcurvature}. See \cite{Chi,Siffert-isoparametric,Thorbergsson} for background. In \cite{Munzner-I, Munzner-II} it is shown that for an isoparametric hypersurface in a constant curvature $(n-1)$-dimensional sphere $\sphere^{n-1} = \{x \in \ste: E(x) = 1\}$:
\begin{itemize}
\item the number $g$ of distinct principal curvatures satisfies $g \in \{1, 2, 3, 4, 6\}$;
\item if the distinct principal curvatures are ordered $\la_{1} > \la_{2} > \dots > \la_{g}$, the multiplicities, $m_{i}$, of the $\la_{i}$ satisfy $m_{i} = m_{i+2}$ (indices modulo $6$), so that there are at most two distinct multiplicities $m_{1}$ and $m_{2}$ (moreover, if $g < 4$ then $m_{1} = m_{2}$ always); and 
\item every such hypersurface arises as a level set of the restriction to the sphere of a polynomial $P:\ste \to \rea$ homogeneous of degree $g$ and satisfying the equations
\begin{align}\label{munznerequations}
&|dP|^{2} = g^{2}E^{g - 1},& &\lap P = \tfrac{m_{2} - m_{1}}{2}g^{2}E^{\tfrac{g}{2} - 1}.
\end{align}
\end{itemize}
A polynomial $P$ solving \eqref{munznerequations} is called a \emph{Cartan-Münzner polynomial}. Examples of solutions for which the resulting hypersurfaces are not extrinsically homogeneous are known when $g = 4$; see for example \cite{Ferus-Karcher-Munzner}. 

\begin{theorem}\label{isoparametrictheorem}
Let $P$ be a Cartan-Münzner polynomial homogeneous of degree $g\geq 2$ and having multiplicities $m_{1}$ and $m_{2}$ on the $n$-dimensional Euclidean vector space $(\ste, h)$. Then the trace-free part $\om_{i_{1}\dots i_{g}}$ of $P_{i_{1}\dots i_{g}}$ solves \eqref{sepoly}, so the pair $(h, \om)$ solves the coupled Einstein equations \eqref{stressenergyintro}.
\end{theorem}

\begin{proof}
It suffices to show that $P$ solves an equation of the form \eqref{sepoly}. The following identity is needed. For any polynomial $F$ homogeneous of degree $g$ there holds
\begin{align}
\begin{aligned}
\label{lapeip}
\lap(E^{i}F) & = 2i(n + 2(g + i - 1))E^{i-1}F + E^{i}\lap F.
\end{aligned}
\end{align}
In particular, the special case $F = 1$ yields $\lap E^{i} = 2i(n+2(i-1))E^{i-1}$.

In the case $m_{1} = m_{2}$, the polynomial $P$ is harmonic and so $\om_{i_{1}\dots i_{g}} = P_{i_{1}\dots i_{g}}$. In this case applying $\lap^{g-2}$ to the first equation of \eqref{munznerequations} and simplifying the result using \eqref{lapeip} yields
\begin{align}
2^{g-2}g(g!)(n+2(g-2))\dots (n+ 2)E = \lap^{g-2}(g^{2}E^{g-1}) = \lap^{g-2}|dP|^{2} = 2^{g-2}|D^{(g-1)}P|^{2}.
\end{align}
Differentiating this yields
\begin{align}
D_{i}D_{j}|D^{(g-1)}P|^{2} = 2g(g!)(n+2(g-2))\dots (n+ 2) h_{ij},
\end{align}
which suffices to show that $P$ solves \eqref{sepoly}. The argument in the general case is similar, but more involved. Since $m_{1} = m_{2}$ whenever $g < 4$, it can be supposed that $g \in \{4, 6\}$. In particular, $g$ is even. Let $P = \sum_{i = 0}^{g/2}E^{i}Q^{(g-2i)}$ be the Lefschetz decomposition of $P$ into its harmonic components. Here $Q^{(g-2i)}$ is a harmonic polynomial homogeneous of degree $g - 2i$, and the decomposition is uniquely determined. Applying $\lap$ to both sides and using \eqref{lapeip} yields
\begin{align}
\tfrac{m_{2} - m_{1}}{2}g^{2}E^{\tfrac{g}{2} - 1} = \lap P = \sum_{i = 1}^{g/2}2i(n+2(g - i -1))E^{i-1}Q^{(g-2i)}.
\end{align}
By the uniqueness of the Lefschetz decomposition, this implies $Q^{(g-2i)} = 0$ if $0 < i < g/2$. Hence
\begin{align}\label{pqred1}
P = Q + \tfrac{(m_{2} - m_{1})g}{2(n+g - 2)}E^{g/2}
\end{align}
where $Q$ is a harmonic polynomial homogeneous of degree $g$. Note that the desired tensor $\om_{i_{1}\dots i_{g}}$ equals $Q_{i_{1}\dots, i_{g}}$. Calculating the differential of \eqref{pqred1} using the homogeneity of $Q$ yields
\begin{align}
g^{2}E^{g-1} = |dP|^{2} = |dQ|^{2} + g^{3}\tfrac{(m_{2} - m_{1})}{(n+g - 2)}E^{g/2 -1}Q + g^{4}\tfrac{(m_{2} - m_{1})^{2}}{4(n+g - 2)^{2}}E^{g-1},
\end{align}
so that
\begin{align}\label{pqred2}
0 =  |dQ|^{2} + g^{3}\tfrac{(m_{2} - m_{1})}{(n+g - 2)}E^{g/2 -1}Q + g^{2}\left(\tfrac{(m_{2} - m_{1})^{2}g^{2}}{4(n+g - 2)^{2}} - 1\right)E^{g-1}.
\end{align}
Applying $\lap^{g-2}$ to both sides of \eqref{pqred2} and simplifying using \eqref{lapeip} yields
\begin{align}
0 = 2^{g-2}\left(|D^{(g-1)}Q|^{2} + \left(\tfrac{(m_{2} - m_{1})^{2}g^{2}}{4(n+g - 2)^{2}} - 1\right)g(g!)(n+2(g-2))(n+2(g-3))\dots (n+2)E\right).
\end{align}
Hence
\begin{align}
Q_{ip_{1}\dots p_{g-1}}Q_{j}\,^{p_{1}\dots p_{g-1}} =  \left(1 - \tfrac{(m_{2} - m_{1})^{2}g^{2}}{4(n+g - 2)^{2}}\right)g(g!)(n+2(g-2))(n+2(g-3))\dots (n+2)h_{ij}.
\end{align}
Because $\om_{i_{1}\dots i_{g}}$ is parallel, this suffices to prove the claim.
\end{proof}
\end{example}

\begin{example}
Let $E = \{e_{1}, \dots, e_{n}\}$ be an enumeration of the edge set of a finite $k$-regular graph with vertex set $V$. The partial Steiner system $\B$ determined by the incidence of edges in the given graph is the collection of $k$-element subsets (blocks) of $\bar{n} = \{1, \dots, n\}$ such that $I = \{i_{1}, \dots, i_{k}\} \in \B$ if and only if the edges $e_{i_{1}}, \dots, e_{i_{k}}$ are incident at some vertex in $V$. Let $\ste$ be the $n$-dimensional real vector space generated by $E$ and equip $\ste$ with the flat Riemannian metric $h$ with respect to which $E$ is an ordered orthonormal basis. Let $x_{1}, \dots, x_{n}$ be the coordinates of $x \in \ste$ with respect to the ordered basis $E$. The quadratic form $Q(x)$ determined by $h$ is $Q(x) = \sum_{i = 1}^{n}x_{i}^{2}$. Associate with $I \in \B$ the monomial $x_{I}  = x_{i_{1}}\dots x_{i_{k}}$. Let $\ep \in \{\pm 1\}^{\B}$, so that $\ep_{I} \in \{\pm 1\}$ for each $I \in \B$.

\begin{theorem}\label{graphpolynomialtheorem}
Let $k \geq 3$ and let $\B$ be the partial Steiner system determined by the incidence of edges in a finite $k$-regular graph. For the homogeneity $k$ polynomial $P(x) = \sum_{I \in \B}\ep_{I}x_{I}$
associated with the $k$-regular graph and any choice of signs $\ep \in \{\pm 1\}^{\B}$, the pair $(h, D^{(k)}P)$ solves the coupled Einstein equation \eqref{stressenergyintro}.
\end{theorem}
\begin{proof}
The component $\tfrac{\pr^{k-1}P}{\pr x_{i_{1}}\dots \pr x_{i_{k-1}}}$ is nonzero if and only if $\{i_{1}, \dots, i_{k-1}\}$ is contained in some block $I = \{i_{1}, \dots, i_{k}\}$ in $\B$ (because any $k-1$ edges meet at most one vertex there is at most one such block). In this case $\tfrac{\pr^{k-1}P}{\pr x_{i_{1}}\dots \pr x_{i_{k-1}}} = \ep_{I}x_{i_{k}}$. Since there are $(k-1)!$ orderings of the distinct indices $i_{1}, \dots, i_{k-1}$ and since the index of a given edge appears in exactly two blocks there results $|D^{(k-1)}P|^{2} = 2(k-1)!Q$. Because no variable $x_{i}$ appears in any monomial of $P$ with a power higher than one, $P$ is $h$-harmonic. By Lemma \ref{epequivalencelemma} this shows that the pair $(h, D^{(k)}P)$ solves \eqref{stressenergyintro}.
\end{proof}

In Theorem \ref{graphpolynomialtheorem}, the restriction $k \geq 3$ is necessary. The polynomial $P(x) = x_{1}x_{2} + x_{2}x_{3} + x_{3}x_{1}$ associated with a triangle satisfies $|dP|^{2} = 2Q + 2P$, which is not a multiple of $Q$.
\end{example}

\section{Properties of solutions}\label{propertiessection}
This section describes qualitative properties of solutions to the coupled hierarchies for completely symmetric tensors. Section \ref{constraintsection} states the main theorems to be proved, Theorems \ref{scalarcurvaturetheorem} and \ref{simonstheorem}. The preliminary section \ref{curvatureoperatorsection} gives some definitions needed in their formulations and proofs. The subsequent sections give the technical material needed for their proofs: Section \ref{weitzenbocksection} reviews Weitzenböck formulas; Section \ref{katosection} presents refined Kato inequalities; Section \ref{boundsection} describes bounds on quantities such as $|\om \kwedge \om|^{2}$; and Section \ref{proofsection} assembles this material to give the proofs of Theorems \ref{scalarcurvaturetheorem} and \ref{simonstheorem}.

\subsection{Metric curvature tensors as operators}\label{curvatureoperatorsection}
Any $\sY \in \mcurv(\std)$ determines a self-adjoint endomorphism of $S^{2}\std$ defined by $a_{ij} \in S^{2}\std \to \sY_{ipjq}a^{pq} \in S^{2}\std$, but this endomorphism need not preserve the subspace $S^{2}_{0}\std$. The modified endomorphism $a_{ij} \to a^{pq}(\sY_{ipjq} + \rictr(\sY)_{p(i}h_{j)q})$ restricts to a self-adjoint endomorphism of $S^{2}_{0}\std$. This is the $k = 2$ special case of the following construction that goes back to A. Lichnerowicz \cite[section $10$]{Lichnerowicz-propagateurs}. 

\begin{lemma}\label{hycommutelemma}
For $\sY \in \mcurv(\std)$, the degree $0$ graded linear operator $\op{\sY}: S(\std) \to S(\std)$ defined by
\begin{align}\label{syom}
&\op{\sY}(\om)_{i_{1}\dots i_{k}} = k\rictr(\sY)_{p(i_{1}}\om_{i_{2}\dots i_{k})}\,^{p} - k(k-1)\sY_{p(i_{1}i_{2}}\,^{q}\om_{i_{3}\dots i_{k})q}\,^{p}, && \om \in S^{k}(\std),
\end{align}
commutes with the $\mathfrak{sl}(2, \rea)$- triple $\{\sraise, \slower, \sdeg\}$. In particular $\op{\sY}(h^{\sprod k}) = 0$ for all $k \geq 1$, and $[\op{\sY}, \tf] = 0$, so $\op{\sY}$ restricts to an endomorphism of $S_{0}(\std)$.
\end{lemma}

\begin{proof}
It is claimed that
\begin{align}\label{syomint}
\begin{aligned}
\rictr(\sY)_{p(i_{1}}\sraise(\om)_{i_{2}\dots i_{k+2})}\,^{p} & = \tbinom{k+1}{2}h_{(i_{1}i_{2}}\rictr(\sY)^{p}\,_{i_{3}}\om_{i_{4}\dots i_{k+2})p} + (k+1)\rictr(\sY)_{(i_{1}i_{2}}\om_{i_{3}\dots i_{k+2})},\\
\sY^{p}\,_{(i_{1}i_{2}}\,^{q}\sraise(\om)_{i_{3}\dots i_{k+2})pq} & = \tbinom{k}{2} h_{(i_{1}i_{2}}\sY^{p}\,_{i_{3}i_{4}}\,^{q}\om_{i_{5}\dots i_{k+2})pq} + \rictr(\sY)_{(i_{1}i_{2}}\om_{i_{3}\dots i_{k+2})}.
\end{aligned}
\end{align}
Combining \eqref{syomint} using \eqref{syom} yields $[\sraise, \op{\sY}] = 0$. The validity of \eqref{syomint} is shown as follows. Write
\begin{align}\label{syomint0}
\begin{aligned}
\sraise(\om)_{i_{1}\dots i_{k+1}p}& = h_{pi_{k+1}}\om_{i_{3}\dots i_{k}} + kh_{p(i_{1}}\om_{i_{2}\dots i_{k})i_{k+1}}+  kh_{i_{k+1}(i_{1}}\om_{i_{2}\dots i_{k})p} + \tbinom{k}{2}h_{(i_{1}i_{2}}\om_{i_{3}\dots i_{k})i_{k+1}p}.
\end{aligned}
\end{align}
Contracting \eqref{syomint0} with $\rictr(\sY)_{i_{k+2}}\,^{p}$ yields
\begin{align}\label{syomint1}
\begin{aligned}
\rictr(\sY)_{i_{k+2}}\,^{p}\sraise(\om)_{i_{1}\dots i_{k+1}p} &= \rictr(\sY)_{i_{k+1}i_{k+2}}\om_{i_{3}\dots i_{k}} + k\rictr(\sY)_{i_{k+2}(i_{1}}\om_{i_{2}\dots i_{k})i_{k+1}} \\&+  k\rictr(\sY)_{i_{k+2}}\,^{p}h_{i_{k+1}(i_{1}}\om_{i_{2}\dots i_{k})p} + \tbinom{k}{2}\rictr(\sY)_{i_{k+2}}\,^{p}h_{(i_{1}i_{2}}\om_{i_{3}\dots i_{k})i_{k+1}p}.
\end{aligned}
\end{align}
Symmetrizing \eqref{syomint1} over the uncontracted indices yields the first identity in \eqref{syomint}. Relabeling $i_{k+1}$ as $q$ in \eqref{syomint0} and contracting the result with $\sY^{p}\,_{i_{k+1}i_{k+2}}\,^{q}$ yields
\begin{align}\label{syomint2}
\begin{aligned}
&\sY^{p}\,_{i_{k+1}i_{k+2}}\,^{q}\sraise(\om)_{i_{1}\dots i_{k}pq} \\
&=\sY^{p}\,_{i_{k+1}i_{k+2}}\,^{q}\left(h_{pq}\om_{i_{3}\dots i_{k}} + kh_{p(i_{1}}\om_{i_{2}\dots i_{k})q} + kh_{q(i_{1}}\om_{i_{2}\dots i_{k})p} + \tbinom{k}{2}h_{(i_{1}i_{2}}\om_{i_{3}\dots i_{k})pq}\right)\\
& = \rictr(\sY)_{i_{k+1}i_{k+2}}\om_{i_{1} \dots i_{k}} - k\sY_{i_{k+2}}\,^{q}\,_{i_{k+1}(i_{1}}\om_{i_{2}\dots i_{k})q}\\
&\quad - k\sY_{i_{k+1}}\,^{q}\,_{i_{k+2}(i_{1}}\om_{i_{2}\dots i_{k})q}
+ \tbinom{k}{2}\sY^{p}\,_{i_{k+1}i_{k+2}}\,^{q}h_{(i_{1}i_{2}}\om_{i_{3}\dots i_{k})pq}
\end{aligned}
\end{align}
Symmetrizing \eqref{syomint2} over the uncontracted indices yields the second identity in \eqref{syomint}. Similarly, 
\begin{align}\label{ytrcom}
\begin{aligned}
&\slower( \op{\sY}(\om))_{i_{1}\dots i_{k-2}} = -\tfrac{1}{2}h^{i_{k-1}i_{k-2}}\left(k\rictr(\sY)_{p(i_{1}}\om_{i_{1}\dots i_{k})}\,^{p} + k(1-k)\sY_{p(i_{1}i_{2}}\,^{q}\om_{i_{3}\dots i_{k})q}\,^{p} \right)\\
& = -\rictr(\sY)^{pq}\om_{i_{1}\dots i_{k-2}pq}+  (k-2)\rictr(\sY)_{p(i_{1}}\slower(\om)_{i_{2}\dots i_{k-2})}\,^{p} \\
&\quad + \rictr(\sY)_{p}\,^{q}\om_{i_{1}\dots i_{k-2}q}\,^{p} -2 \tbinom{k-2}{2}\sY_{p(i_{1}i_{2}}\,^{q}\slower(\om)_{i_{3}\dots i_{k-2})q}\,^{p} \\
& = (k-2)\left(\rictr(\sY)_{p(i_{1}}\slower(\om)_{i_{2}\dots i_{k-2})}\,^{p} + (3-k)\sY_{p(i_{1}i_{2}}\,^{q}\slower(\om)_{i_{3}\dots i_{k-2})q}\,^{p}\right) 
= \op{\sY}(\slower(\om))_{i_{1}\dots i_{k-2}},
\end{aligned}
\end{align}
which shows $[\slower, \op{\sY}] = 0$.
That $\op{\sY}(h^{\sprod k}) = 0$ follows from $\op{\sY}(h) = 0$ and \eqref{sl2ops} by induction. 
Because $\op{\sY}$ commutes with the $\mathfrak{sl}(2, \rea)$-triple $\{\sraise, \slower, \sdeg\}$, it preserves the decomposition of symmetric tensors into their primitive parts. In particular, it commutes with $\tf$.
\end{proof}

It follows from \eqref{syom} and \eqref{rictralbe} that for $\al, \be \in  S^{k}\std$ and $\sY \in \mcurv(\std)$,
\begin{align}\label{qyalbe}
\begin{aligned}
\lb \op{\sY}(\al), \be\ra &= \tbinom{k}{2}\lb \al \kwedge \be, \sY\ra + k\lb \rictr(\al \kwedge \be), \rictr(\sY)\ra -k\left(\lb \imt(\rictr(\sY))\al, \slower(\be)\ra + \lb \slower(\al), \imt(\rictr(\sY))\be\ra \right).
\end{aligned}
\end{align}
If $\al, \be \in  S^{k}_{0}\std$ then the last two terms in \eqref{qyalbe} vanish.

Because the operator $\op{\sY}$ associated with $\sY \in \mcurv(\std)$ is self-adjoint it determines on any $\op{\sY}$-invariant subspace $E \subset \tensor^{k}\std$ a quadratic form defined by $\qY(\om) = \lb \om, \op{\sY}(\om)\ra$ for $\om \in E$. Taking $\be = \al = \om \in S^{k}_{0}\std$ in \eqref{qyalbe} yields $\qY( \om) = k\lb \rictr(\om \kwedge \om), \rictr(\sY)\ra + \tbinom{k}{2}\lb \om \kwedge \om, \sY\ra$. 

%%% SAVE: Following appears to be unused so can be omitted
\begin{example}
The metric curvature tensor $\sH = - \tfrac{1}{n(n-1)}h \kwedge h$ determines an operator $\op{\sH}$ and an associated quadratic form $\qH$. The coefficient is chosen so that $\rictr(\sH) = \tfrac{1}{n}h$ and $\scal(\sH) = 1$.    That a metric $h$ have constant sectional curvature is equivalent to the statement $\riem = \scal \sH$ with $\scal$ constant. By \eqref{qyalbe} there hold
\begin{align}\label{qhom}
&\op{\sH}(\om)= \tfrac{k(n+k-2)}{n(n-1)}\om, &&\qH(\om) =  \tfrac{k(n+k-2)}{n(n-1)}|\om|^{2}, && \om \in \Ga(\symkt).
\end{align}
In particular, $\qH$ is positive definite on $S^{k}_{0}\std$. Note that $\sH$ is \textit{negative} definite when viewed as an endomorphism of $\ext^{2}\std$. 
\end{example}
%SAVE: gives gyalbe with trace-free part
%When $n > 2$ and $\al, \be \in S^{k}_{0}\std$, substituting \eqref{tfweyl} into \eqref{qyalbe} yields
%\begin{align}\label{qyalbetr}
%\begin{aligned}
%\lb \al, \op{\sY}(\be)\ra  &=  \tbinom{k}{2}\lb \al \kwedge \be, \tf \sY\ra + \tfrac{k(n+2(k-2))}{n-2}\lb \rictr(\al \kwedge \be), \mr{\rictr(\sY)}\ra  + \tfrac{k(n+k - 2)}{n(n-1)}\scal(\sY)\lb \al, \be \ra. 
%\end{aligned}
%\end{align}

\subsection{Constraints on solutions: Calabi-Cheng-Yau growth estimates and Simons style pinching}\label{constraintsection}
This section describes a priori constraints on the growth of solutions of the coupled projectively flat equations \eqref{projectivehiggsintro} when $h$ is Riemannian. It would be interesting to extend, even partially, such constraints to solutions of the coupled Einstein equations \eqref{stressenergyintro}.

\begin{theorem}\label{scalarcurvaturetheorem}
Let $M$ be a manifold of dimension $n \geq 3$. Suppose $h$ is a complete Riemannian metric which with a trace-free Codazzi tensor $\om$ solves the coupled projectively flat equations \eqref{projectivehiggsintro} for \revnew{some $c > 0$; by Lemma \ref{projectivehiggslemma}, $\ka$ is then necessarily a constant.}
\begin{enumerate}
\item If $\ka \geq 0$ then $\om$ is identically zero, and $h$ is a metric of constant sectional curvature.
\item If $\ka < 0$ then 
\begin{itemize}
\item If $\ka < 0$ and $k \geq 3$, then $\sup_{M}|\om|^{2} \leq -\ka/c$, so the scalar curvature $\scal = c|\om|^{2} + \ka$ of $h$ is nonpositive.
\item If $\ka < 0$ and $k = 2$, then $\sup_{M}|\om|^{2} \leq -\tfrac{\ka}{c(n-1)}$, so the scalar curvature $\scal$ of $h$ satisfies $\scal = c|\om|^{2} + \ka \leq \tfrac{n-2}{n-1}\ka$ (which is strictly negative if $n > 2$).
\end{itemize}
\end{enumerate}
\end{theorem}

\begin{remark}
In the special case corresponding to the context of the cubic form of a complete hyperbolic affine sphere, Calabi \cite{Calabi} showed the nonpositivity of the Ricci curvature. It seems plausible that, perhaps with some additional conditions, the nonpositivity of the scalar curvature in Theorem \ref{scalarcurvaturetheorem} can be improved to nonpositivity of the Ricci curvature. The remaining technical issue is that certain tensorial identities used in Calabi's argument for $k = 3$ are special to that case, so some new ingredient is needed to extend these results to $k > 3$.
\end{remark}

The proof of Theorem \ref{scalarcurvaturetheorem} requires the Weitzenböck formulas described in Section \ref{weitzenbocksection}, the refined Kato inequalities described in Section \ref{katosection}, and a theorem of Cheng-Yau on the growth of solutions to a differential inequality of the form $\lap u \geq Bu^{1+\si} - Au$ recalled as Theorem \ref{cyestimatetheorem}.

These results can be understood as generalizations of results for holomorphic tensors on surfaces and as a counterparts to classical vanishing theorems for holomorphic symmetric tensors due to Kobayashi \cite{Kobayashi-holomorphicsymmetric, Kobayashi-holomorphictensor}.

For $\om \in S^{k}_{0}\std$ there holds 
\begin{align}\label{quadom}
\tfrac{1}{k}\lb \op{\om \kwedge \om}(\om), \om \ra = \tfrac{k-1}{2}|\om \kwedge \om|^{2} + |\rictr(\om \kwedge \om)|^{2}
\end{align} 
by \eqref{qyalbe}. The proofs of Theorem \ref{scalarcurvaturetheorem} and Corollary \ref{simonstheorem} below depend on estimating \eqref{quadom} from above and below. Because the unit sphere in $S^{k}_{0}\std$ is compact, the expression \eqref{quadom} assumes its maximum and minimum values on the sphere, so there are constants $\skmax_{n, k}, \skmin_{n, k} \in \rea$ defined by
\begin{align}\label{skdefined}
\begin{aligned}
\skmax_{n, k} &=\max\{\tfrac{k-1}{2}|\om \kwedge \om|^{2} + |\rictr(\om \kwedge \om)|^{2}: \om \in S^{k}_{0}\std, |\om|^{2} = 1\},\\
\skmin_{n, k} &=\min\{\tfrac{k-1}{2}|\om \kwedge \om|^{2} + |\rictr(\om \kwedge \om)|^{2}: \om \in S^{k}_{0}\std, |\om|^{2} = 1\},
\end{aligned}
\end{align} 
and for all $\om \in S^{k}_{0}\std$, there hold
\begin{align}\label{skineq}
\skmin_{n, k}|\om|^{4} \leq \tfrac{1}{k}\lb \op{\om \kwedge \om}(\om), \om \ra = \tfrac{k-1}{2}|\om \kwedge \om|^{2} + |\rictr(\om \kwedge \om)|^{2} \leq \skmax_{n, k}|\om|^{4}.
\end{align}
Note that $\skmax_{n, k}$ and $\skmin_{n, k}$ depend only on the homothety class of the chosen Euclidean metric. A basic technical issue is to obtain the values of $\skmax_{n, k}$ and $\skmin_{n, k}$ and to characterize the $\om$ for which there holds either of the equalities in \eqref{skineq}. In Section \ref{boundsection} it will be apparent that this is easier for $\skmin_{n, k}$ than for $\skmax_{n, k}$. 

Let $E$ and $F$ be bundles of tensors on $M$. A metric $h$ on $M$ determines a pairing $\ilp \al, \be \irp = \int_{M}\lb \al, \be \ra \,d\vol_{h}$ of sections $\al, \be \in \Ga(E)$, at least one of which is compactly supported. Write $\iln \om \irn^{2} = (\om, \om)$. 

Theorem \ref{simonstheorem} is the integral bound parallel to that of Theorem \ref{scalarcurvaturetheorem} for solutions of the coupled projectively flat theorem \eqref{projectivehiggsintro} for negative $c$.
\begin{theorem}\label{simonstheorem}
Let $M$ be a compact oriented manifold of dimension $n \geq 3$. Suppose the Riemannian metric $h$ and the trace-free Codazzi tensor $\om$ solve the coupled projectively flat equations \eqref{projectivehiggsintro} for \revnew{some $c < 0$; by Lemma \ref{projectivehiggslemma}, $\ka$ is then necessarily a constant.} Then
\begin{align}\label{general simonsinequality}
&0 \geq \int_{M}|\om|^{2}\left(\tfrac{n+k-2}{n(n-1)}\ka + c\skmax_{n, k})|\om|^{2}\right)d\vol_{h}.
\end{align}
\end{theorem}

Estimates of $\skmax_{n,k}$ obtained in Section \ref{boundsection} yield Corollary \ref{simonscorollary}.

\begin{corollary}\label{simonscorollary}
Let $M$ be a compact oriented manifold of dimension $n \geq 3$. Suppose the Riemannian metric $h$ and the trace-free Codazzi tensor $\om$ solve the coupled projectively flat equations \eqref{projectivehiggsintro} for \revnew{some $c < 0$; by Lemma \ref{projectivehiggslemma}, $\ka$ is then necessarily a constant.} Then
\begin{align}\label{simonsinequality}
&0 \geq \int_{M}|\om|^{2}\left(\tfrac{n+k-2}{n(n-1)}\ka + c(1 + \tfrac{(2n+1)(k-1)}{n})|\om|^{2}\right)d\vol_{h}, & &\text{if}\,\, k > 3.\\
\label{simonsinequalityk3}
&0 \geq \int_{M}|\om|^{2} \left( \tfrac{n+1}{n(n-1)}\ka + \tfrac{2n-1}{n}c|\om|^{2}\right)d\vol_{h},& &\text{if}\,\, k = 3.\\
\label{simonsinequalityk2}
&0 \geq \int_{M}|\om|^{2}\left(\tfrac{1}{n-1}\ka + c|\om|^{2}\right)d\vol_{h}, & &\text{if}\,\, k = 2.
\end{align}
\end{corollary}

\begin{example}\label{simonsexample}
The $k = 2$ case of Corollary \ref{simonscorollary} recovers an integral estimate of the scalar curvature of a compact mean curvature zero hypersurface in a round sphere due to J. Simons \cite{Simons}.
Applied with $h$ the induced metric and $\om$ the second fundamental form of a mean curvature zero compact immersed hypersurface in the $(n+1)$-dimensional round sphere of scalar curvature $n(n+1)$ as in Example \ref{constraintequationsexample}, Theorem \ref{simonstheorem} recovers the specialization to such hypersurfaces of a well known theorem of Simons \cite[Theorem $5.3.2$]{Simons} (which applies to compact submanifolds of arbitrary codimension). Concretely, in this case, $c = -1$ and $\scal - |\om|^{2} = \ka = n(n-1)$, so \eqref{simonsinequalityk2} becomes 
\begin{align}\label{simonsk2}
0 \geq \int_{M}|\om|^{2}\left(n -|\om|^{2}\right)d\vol_{h} = -n^{2}(n-1)^{2}\int_{M}\left(1 - \tfrac{\scal}{n(n-1)}\right)\left(\tfrac{n-2}{n-1} - \tfrac{\scal}{n(n-1)}\right)d\vol_{h}.
\end{align}
which recovers \cite[Theorems $5.3.2$ and $5.3.3$]{Simons}.
From \eqref{simonsk2} it follows that either $\scal = n(n-1)$ and $\om$ is identically zero, so that the hypersurface is a totally geodesic hypersphere; $\scal = n(n-2)$ and $|\om|^{2} = n$, in which case $\om$ is parallel; or $\inf_{M}\scal < n(n-2)$, which is \cite[Corollary $5.3.3$]{Simons}.
\end{example}

\begin{example}\label{chenogiueexample}
The $k = 3$ case of Corollary \ref{simonscorollary} recovers the analogous integral estimate for the scalar curvature of a compact mean curvature zero Lagrangian submanifold of a constant holomorphic sectional curvature Kähler manifold due to B.-Y. Chen and K. Ogiue \cite{Chen-Ogiue}. 
Applied with $h$ the induced metric and $\om$ the twisted second fundamental form of a mean curvature zero compact immersed Lagrangian submanifold in a $2n$-dimensional Kähler manifold of constant holomorphic sectional curvature $4\hat{c}$ as in Theorem \ref{constantsecttheorem}, Theorem \ref{simonstheorem} recovers the specialization to such submanifolds of a theorem of Chen-Ogiue \cite[Theorem $4.1$]{Chen-Ogiue}. Concretely, in this case, $c = -1$, $\ka = \hat{c}n(n-1)$,  and $|\om|^{2} = \hat{c}n(n-1) - \scal$, so \eqref{simonsinequalityk3} becomes 
\begin{align}\label{chenogiue}
\begin{aligned}
0 &\geq \int_{M}|\om|^{2}\left( \tfrac{n(n+1)}{2n-1}\hat{c} - |\om|^{2}\right)d\vol_{h}
= n^{2}(n-1)^{2}\int_{M}\left(\hat{c} - \tfrac{\scal}{n(n-1)}\right)\left(\tfrac{\scal}{n(n-1)} - \tfrac{2n(n-2)}{(2n-1)(n-1)}\hat{c}\right)d\vol_{h},
\end{aligned}
\end{align}
which recovers \cite[Theorem $4.1$]{Chen-Ogiue}. 
From \eqref{chenogiue} it follows that, if $\hat{c} > 0$, then either $\scal = n(n-1)\hat{c}$ and $\om$ is identically zero, so that the submanifold is totally geodesic and Lagrangian; $\scal = \tfrac{2n^{2}(n-2)}{2n-1}\hat{c}$ and $|\om|^{2} =\tfrac{n(n+1)}{2n-1}\hat{c} $, in which case $\om$ is parallel; or $\inf_{M}\scal < \tfrac{2n^{2}(n-2)}{2n-1}\hat{c}$.
\end{example}

One point here is that these examples are somehow the same result for different values of $k$. The corresponding statement obtained here for $k > 3$ is plainly not sharp, because certain inequalities for norms of tensors used in intermediate steps that are sharp for $k \leq 3$ can be improved when $k > 3$. However, even when $k \leq 3$, the method of proof has the virtues, when compared with previous approaches, that it is uniform in $k$ and that it does not suppose an immersion in an ambient space.

\subsection{Review of Weitzenböck formulas}\label{weitzenbocksection}
The operator $\klie$ defined in \eqref{kliedefined} maps the vector space $\Ga(\symkt)$ into $\Ga(\precod^{k+1}(\ctm)) $ where
\begin{align}\label{precoddefined}
\precod^{k+1}(\std) = \left\{
\phi_{iji_{1}\dots i_{k-1}}\in \tensor^{k+1}\std: \begin{aligned}&\phi_{iji_{1}\dots i_{k-1}} = \phi_{[ij]i_{1}\dots i_{k-1}} = \phi_{ij(i_{1}\dots i_{k-1})},\\
& \phi_{pji_{1}\dots i_{k-2}}\,^{p} =0, \phi_{[iji_{s}]i_{1}\dots \hat{i}_{s}\dots i_{k-1}} = 0 \,\,\text{for}\,\, 1\leq s \leq k-1.
\end{aligned}\right\}
\end{align}
comprises trace-free $(k+1)$-tensors having the symmetries determined by the Young projector given by symmetrization over the rows followed by antisymmetrization over the columns of the Young diagram corresponding to the partition $(k, 1)$. 
The formal adjoint $\kliea:\Ga(\precod^{k+1}(\ctm)) \to \Ga(\symkt)$ satisfies
\begin{align}
\label{klieaklie}
&\kliea(\phi)_{i_{1}\dots i_{k}} = -D^{p}\phi_{p(i_{1}\dots i_{k})},&&
\kliea\klie(\om)_{i_{1}\dots i_{k}} = -D^{p}\klie(\om)_{p(i_{1}\dots i_{k})}.&
\end{align}

\begin{lemma}\label{ellipticlemma}
Let $(M, h)$ be an $n$-dimensional Riemannian manifold.
\begin{enumerate}
\item The differential operator $\clie:\Ga(\symkt) \to \Ga(\symkpt)$ has injective symbol, so 
\begin{align}
\Ga(\symkpt) = \clie(\Ga(\symkt)) \oplus (\ker \div \cap \Ga(\symkpt)),
\end{align}
and $\div \clie:\Ga(\symkt) \to \Ga(\symkt)$ is an elliptic operator. If $M$ is compact, then $\div \clie$ is nonpositive and $\ker \div \clie = \ker \clie$.
\item For any $c \geq \tfrac{n+2(k-2)}{2(n+k-3)}$, the operator $\diamond_{c} = (\klie, \sqrt{c}\div):\Ga(\symkt) \to \Ga(\precod^{k+1}(\ctm) \oplus S^{k-1}_{0}(\ctm))$ has injective symbol, so 
\begin{align}
\diamond_{c}^{\ast}\diamond_{c} = -\kliea\klie + c\clie\div
\end{align}
is an elliptic operator. If $M$ is compact, then $\diamond_{c}$ is nonpositive and $\ker \diamond_{c} = \ker \klie \cap \ker \div$. 
\item For any $c \geq \tfrac{(k+1)(n+2(k-1))}{2k(n+k-3)}$, the operator $\kpc_{c} = (\klie, \sqrt{c}\clie):\Ga(\symkt) \to \Ga(\precod^{k+1}(\ctm) \oplus S^{k+1}_{0}(\ctm))$ has injective symbol, so 
\begin{align}
\kpc_{c}^{\ast}\kpc_{c} = -\kliea\klie + c\div\clie
\end{align}
is an elliptic operator. If $M$ is compact, then $\kpc_{c}$ is nonpositive and $\ker \kpc_{c} = \ker \klie \cap \ker \clie$. 
\end{enumerate}
\end{lemma}
\begin{proof}
Write $\sbl_{\clie}(Z)(\phi)$, $\sbl_{\klie}(Z)(\phi)$, and $\sbl_{\div}(Z)(\phi)$ for the symbols of $\clie$, $\klie$, and $\div$ applied to the vector $Z^{i}$ and $\phi \in \Ga(\symkt)$. Write $(i(Z)\phi)_{i_{1}\dots i_{k-1}} = Z^{p}\phi_{pi_{1}\dots i_{k-1}}$.  
Straightforward computations show
\begin{align}\label{cliesymbolnorm}
&\begin{aligned}
|\sbl_{\clie}(Z)(\phi)|^{2}& = \tfrac{1}{k+1}|Z|^{2}|\phi|^{2} + \tfrac{k(n+2(k-2))}{(k+1)(n+2(k-1))}|i(Z)\phi|^{2},
\end{aligned}\\
\label{kliesymbolnorm}
&\begin{aligned}
|\sbl_{\klie}(Z)(\phi)|^{2} 
& =  \tfrac{1}{2}\left(|Z|^{2}|\phi|^{2} -|i(Z)\phi|^{2}\right) - \tfrac{k-1}{2(n+k-3)}|i(Z)\phi|^{2} =  \tfrac{1}{2}\left(|\phi|^{2} - \tfrac{n+2(k-2)}{n+k-3}|i(Z)\phi|^{2}\right).
\end{aligned}
\end{align}
When $k = 1$ and $n = 2$ the coefficient of the pure trace terms in \eqref{kliesymbolnorm} should be understood in a limiting sense. 
By \eqref{cliesymbolnorm}, if $\sbl_{\clie}(Z)(\phi) = 0$ for nonzero $Z$, then $\phi = 0$, so $\clie$ has injective symbol. The remaining claims follow from standard elliptic operator theory as in \cite[section $4$]{Berger-Ebin}. If $M$ is compact, then $\ilp \div \clie \om, \om\irp = - \iln\clie(\om)\irn^{2} - c\iln \div \om \irn^{2} \leq 0$ and $\div\clie\om = 0$ if and only if $\clie(\om) = 0$.
Combining $|\sbl_{\div}(Z)(\phi)|^{2} = |\imt(Z)\phi|^{2}$ with \eqref{kliesymbolnorm} yields
\begin{align}
|\sbl_{\diamond_{c}}(Z)(\phi)|^{2} =  \tfrac{1}{2}|Z|^{2}|\phi|^{2} + \left(c - \tfrac{1}{2}\tfrac{n+2(k-2)}{n+k-3}\right)|i(Z)\phi|^{2},
\end{align}
from which the injectivity of $\sbl_{\diamond_{c}}(Z)$ is apparent. The ellipticity of $\diamond_{c}$ follows from standard elliptic operator theory as in \cite[section $6$]{Berger-Ebin}. If $M$ is compact, then $\ilp \diamond_{c}\om, \om\irp = - \iln\klie(\om)\irn^{2} - c\iln \div \om \irn^{2} \leq 0$ and $\diamond_{c}\om = 0$ if and only if $\klie(\om) = 0$ and $\div \om = 0$.
Combining \eqref{cliesymbolnorm} and \eqref{kliesymbolnorm} shows
\begin{align}
|\sbl_{\diamond_{c}}(Z)(\phi)|^{2} =  \tfrac{k+1 + 2c}{2(k+1)}|Z|^{2}|\phi|^{2} + \left(\tfrac{ck(n+2(k-2)}{(k+1)(n + 2(k-1))} - \tfrac{1}{2}\tfrac{n+2(k-2)}{n+k-3}\right)|i(Z)\phi|^{2},
\end{align}
from which the injectivity of $\sbl_{\kpc_{c}}(Z)$ is apparent. The ellipticity of $\kpc_{c}$ follows from standard elliptic operator theory. If $M$ is compact, then $\ilp \kpc_{c}\om, \om\irp = - \iln\klie(\om)\irn^{2} - c\iln \clie \om \irn^{2} \leq 0$ and $\kpc_{c}\om = 0$ if and only if $\klie(\om) = 0$ and $\clie(\om) = 0$.
\end{proof}

\begin{lemma}\label{culaplemma}
Let $(M, h)$ be a Riemannian manifold of dimension $n \geq 2$. For $\al \in \rea$ define a formally self-adjoint second order elliptic differential operator $\culap_{\al}: \Ga(S^{k}_{0}(\ctm)) \to \Ga(S^{k}_{0}(\ctm))$ by 
\begin{align}\label{culapdefined}
\culap_{\al}\om = \lap_{h}\om +\tfrac{\al}{k} \op{\sR}(\om).
\end{align}
\begin{enumerate}
\item If $\al = -1$, then
\begin{align}
\label{lapom3} \culap_{-1} = \lap_{h} \om - \tfrac{1}{k}\op{\sR}(\om)& =  \tfrac{n+2(k-2)}{n+k-3}\clie \div(\om) - 2\kliea\klie(\om),
\end{align}
is an elliptic operator on $\Ga(\symkt)$. If $M$ is compact, then $\culap_{-1}$ is nonpositive and $\ker \culap_{-1} \cap \Ga(S^{k}_{0}(\ctm)) = \ker \klie \cap \ker \div \cap \Ga(S^{k}_{0}(\ctm))$.
\item If $\al = \tau(k) = \tfrac{k}{n+k-2}$, then
\begin{align}
\label{lapom2}\begin{aligned}
\culap_{\tau(k)}\om =  \lap_{h} \om + \tfrac{1}{n+k-2}\op{\sR}(\om)& =  \tfrac{n+2(k-1)}{n+k-2}\div \clie(\om) - \tfrac{2k(n+k-3)}{(k+1)(n+k-2)}\kliea\klie(\om),
\end{aligned}
\end{align}
is an elliptic operator on $\Ga(\symkt)$. If $M$ is compact, then $\culap_{\tau(k)}$ is nonpositive and $\ker \culap_{\tau(k)} \cap \Ga(S^{k}_{0}(\ctm)) = \ker \klie \cap \ker \clie \cap \Ga(S^{k}_{0}(\ctm))$.
\item If $-1 < \al < \tau(k) = \tfrac{k}{n+k-2}$ then $\culap_{\al}$ is an elliptic operator on $\Ga(\symkt)$. If $M$ is compact, then $\culap_{\al}$ is nonpositive and $\ker \culap_{\al} \cap \Ga(S^{k}_{0}(\ctm)) = \ker D \cap \Ga(S^{k}_{0}(\ctm))$.
\end{enumerate}
\end{lemma}
\begin{proof}
For $\om \in \Ga(\symkt)$, straightforward computations using the Ricci identity and \eqref{cliedefined} show
\begin{align}
\label{liediv}
\begin{aligned}
\clie \div(\om) &= \tfrac{1}{2}\{h, \div \om\} - \tfrac{2}{k(n+2(k-2))}\sraise(\div^{2}(\om)),
\end{aligned}\\
\begin{aligned}
\label{divlie}
\lap_{h}\om + \op{\sR}(\om) & = (k+1)\div \clie(\om)  -  \tfrac{k(n+2(k-2))}{n+2(k-1)}\clie\div(\om).
\end{aligned}
\end{align}
Differentiating \eqref{domkl} and tracing the result using \eqref{liediv} and \eqref{klieaklie} gives
\dps{\label{lapom}
\lap_{h}\om& = \div \clie(\om) + \tfrac{k(n+2(k-2))}{(n+k-3)(n+2(k-1))}\clie \div(\om) - \tfrac{2k}{k+1}\kliea\klie(\om).
}
Solving \eqref{divlie} for $\lap_{h}\om$ and substituting the result into \eqref{lapom} yields
\begin{align}\label{klieweitzenbock}
\begin{aligned}
&\tfrac{1}{k}\op{\sR}(\om)  = \div \clie(\om) - \tfrac{(n+k-2)(n+2(k-2))}{(n+k-3)(n+2(k-1))}\clie\div(\om) + \tfrac{2}{k+1}\kliea\klie(\om).
\end{aligned}
\end{align}
Equation \eqref{divlie} and \eqref{klieweitzenbock} are the analogues of the corresponding identities for operators on antisymmetric forms, for example \cite[Equations $(2.8)$ and $(2.9)$]{Semmelmann}. Rewriting \eqref{lapom} in two different ways using \eqref{divlie} gives \eqref{lapom3} and \eqref{lapom2}. The ellipticity of $\culap_{\al}$ in these cases, and its nonpositivity when $M$ is compact, follow from Lemma \ref{ellipticlemma}. Being convex combinations of the elliptic operators $\culap_{-1}$ and $\culap_{\tau(k)}$, the operators $\culap_{\al}$ for $-1 < \al < \tau(k)$ are elliptic. If $M$ is compact, the same argument shows that $\culap_{\al}$ is nonpositive. By \eqref{lapom3},
\begin{align}\label{culap1}
\culap_{\al} \om = (1+\al)\div \clie(\om) + \tfrac{(n+2(k-2))(k - \al(n+k-2))}{(n+k-3)(n+2(k-1))}\clie \div(\om)  + \tfrac{2(\al - k)}{k+1}\kliea\klie(\om),
\end{align}
from which follows $\ker \culap_{\al} \cap \Ga(S^{k}_{0}(\ctm)) \supset \ker D \cap \Ga(S^{k}_{0}(\ctm))$.
If $M$ is compact and $\om \in \ker \culap \cap \Ga(S^{k}_{0}(\ctm))$, integrating the right side of \eqref{culap1} gives 
\begin{align}
0 =  (1+\al)\iln\clie(\om)\irn^{2} +  \tfrac{(n+2(k-2))(k - \al(n+k-2))}{(n+k-3)(n+2(k-1))}\iln\div(\om)\irn^{2} + \tfrac{2(k-\al)}{k+1}\iln\klie(\om)\irn^{2},
\end{align}
and together with \eqref{domkl} this implies $\ker \culap_{\al} \cap \Ga(S^{k}_{0}(\ctm)) \subset \ker D \cap \Ga(S^{k}_{0}(\ctm))$.
\end{proof}

\begin{corollary}\label{hypothesiscorollary}
If $h$ is a Riemannian metric on a compact manifold $M$ of dimension $n \geq 2$, and $\om \in \Ga(\symkt)\cap \ker(\lap_{h} - \tfrac{1}{k}\op{\sR})$, then $\om$ satisfies the hypotheses of Lemma \ref{tracefreeconstantlemma}.
\end{corollary}
\begin{proof}
Since $M$ is compact, and $\lap_{h} - \frac{1}{k}\op{\sR} = \culap_{-1}$, Lemma \ref{culaplemma} shows that $\ker(\lap_{h} -\frac{1}{k} \op{\sR}) \cap \Ga(S^{k}_{0}(\ctm)) = \ker \klie \cap \ker \div \cap \Ga(S^{k}_{0}(\ctm))$.
\end{proof}

\begin{example}
Let $M$ be compact. Define a functional $\cubic_{\al}$ with arguments a Riemannian metric $h$ and a tensor $\om \in \Ga(S^{k}_{0}(\ctm))$ by $\cubic_{\al}(h, \om) = -\ilp\om,\culap_{\al}\om\irp$. For fixed $h$ the first variation of $\cubic_{\al}(h, \om)$ in $\om$ yields the equation $\culap_{\al}\om = 0$. Lemma \ref{culaplemma} implies that for $-1 \leq \al \leq \tfrac{k}{n+k-2}$ the functional $\cubic_{\al}(h, \om)$ is nonnegative.
\end{example}

\begin{example}
The \emph{Lichnerowicz Laplacian} $\lich$ is the formally self-adjoint operator which acts on an arbitrary rank $k$ covariant tensor $\om_{i_{1}\dots i_{k}}$ by $-\lap_{h}\om + \op{\sR}(\om)$ (see \cite[page $27$]{Lichnerowicz-propagateurs} for the definition of $\op{\sR}(\om)$ for general tensors $\om$). The linearization of the Ricci curvature of the metric $h$ at the symmetric two-tensor $a_{ij}$ is $\tfrac{1}{2}\lich a_{ij} + D_{(i}D^{p}a_{j)p}$. On differential forms the Lichnerowicz Laplacian restricts to the usual Hodge Laplacian. The Lichnerowicz operator restricts to $-\culap_{-k}$ on $\Ga(S^{k}_{0}(\ctm))$.
\end{example}

\begin{example}
The special case of $\culap_{-1} = \lap_{h} - \tfrac{1}{2}\op{\sR}$ acting on $\Ga(S^{2}_{0}(\ctm))$ was studied by J. Simons in \cite{Simons}, and this case of Lemma \ref{culaplemma} is given in \cite[section $6$.c]{Berger-Ebin}.
\end{example}

\begin{example}
As is shown in \cite{Berger-Ebin}, an infinitesimal deformation of an Einstein metric $h$ on a compact manifold is identified with an $\om \in \Ga(S^{2}_{0}(\ctm))\cap \ker \div$ solving $\lap \om = \op{\sR}(\om) - \tfrac{2\sR}{n}\om$. As is summarized in \cite[section $12$.H]{Besse} (the notations there are different than those here), using this equation in conjunction with the positivity conditions given by integrating \eqref{divlie} and \eqref{lapom3} gives a proof of the criterion of N. Koiso \cite[Theorem $3.3$]{Koiso-nondeformability}, for the rigidity of an Einstein metric, in particular showing that an Einstein metric of negative sectional curvature is rigid provided $n \geq 3$.
\end{example}

The Weitzenböck formulas of Corollary \ref{weitzenbockcorollary} are in \cite[Section $6.4$]{Fox-ahs}, in \cite[Section $6$]{Heil-Moroianu-Semmelmann} (where they are obtained as special cases of the general machinery of \cite{Semmelmann-Weingart}), and in some equivalent form in \cite{Dairbekov-Sharafutdinov}.

\begin{corollary}\label{weitzenbockcorollary}
Let $(M, h)$ be a Riemannian manifold of dimension $n \geq 2$.  For $\om \in \Ga(\symkt)$ there hold
\begin{align}
\label{lapomsq}\tfrac{1}{2}\lap_{h}|\om|^{2} & = |D\om|^{2} + (k+1)\lb \om, \div \clie(\om)\ra  -  \tfrac{k(n+2(k-2))}{n+2(k-1)}\lb \om, \clie\div(\om)\ra - \qR(\om).\\
\label{lapomdivlie}&\begin{aligned}
\tfrac{1}{2}\lap_{h}|\om|^{2} & = |D\om|^{2} + \tfrac{n+2(k-1)}{n+k-2}\lb \om, \div \clie(\om)\ra  - \tfrac{2k(n+k-3)}{(k+1)(n+k-2)}\lb \om, \kliea\klie(\om)\ra%
- \tfrac{1}{n+k-2}\qR(\om).
\end{aligned}\\
\label{lapomliediv}\tfrac{1}{2}\lap_{h}|\om|^{2} & = |D\om|^{2} + \tfrac{n+2(k-2)}{n+k-3}\lb\om, \clie \div(\om)\ra - 2\lb\om, \kliea\klie(\om)\ra + \tfrac{1}{k}\qR(\om) . 
\end{align}
\end{corollary}
\begin{proof}
Contracting \eqref{divlie}, \eqref{lapom2}, and \eqref{lapom3} with $\om$ yields \eqref{lapomsq}-\eqref{lapomliediv}.
\end{proof} 

\subsection{Refined Kato inequalities and vanishing theorems for conformal Killing and divergence-free Codazzi tensors}\label{katosection}
The inequality \eqref{kato} of Lemma \ref{katolemma} generalizes the estimate for the second fundamental form of a minimal hypersurface proved in \cite{Schoen-Simon-Yau}. With a $1$ in place of $\tfrac{n+k-2}{n+2(k-1)}$, it follows from the Cauchy-Schwarz inequality, and is known as a \emph{Kato inequality}. The estimates \eqref{kato}-\eqref{kato3} of Lemma \ref{katolemma} are \emph{refined Kato inequalities} in the sense of \cite{Calderbank-Gauduchon-Herzlich} and \cite{Branson-kato} (see also \cite{Hitchin-vanishing}), and can be deduced from the results in either of those papers. In particular the results of \cite[Section $6$]{Calderbank-Gauduchon-Herzlich} include Lemma \ref{katolemma}, and the discussion at the very end of \cite[Section $6$]{Calderbank-Gauduchon-Herzlich} gives the explicit constants of \eqref{kato}-\eqref{kato3} for the cases $k = 1, 2$. (The $k = 2$ case of \eqref{kato} had earlier been stated in \cite[section $4$]{Bourguignon-magic}.) As considerable work is required to translate general representation theoretic statements into the concrete contexts here, and to keep the exposition self-contained, it is simpler to give here direct proofs following the general procedure described in the introduction of \cite{Branson-kato}, and not utilizing general representation theoretic machinery. These proofs were given previously in \cite{Fox-ahs}.

It is convenient to call a tensor $\om_{i_{1} \dots i_{k}} \in \Ga(S^{k}(\ctm))$ a \emph{dual Killing tensor} if $\om^{i_{1}\dots i_{k}} \in \Ga(S^{k}(TM))$ is a Killing tensor.
\begin{lemma}\label{katolemma}
Let $(M, h)$ be a Riemannian $n$-manifold with Levi-Civita connection $D$, and let $k \geq 1$. 

For a trace-free Codazzi tensor $\phi \in \Ga(\symkt)\cap \ker \klie \cap \ker \div$, where $|\phi|^{2} \neq 0$ there holds
\begin{align}\label{kato}
|d|\phi||^{2} \leq \tfrac{n+k-2}{n+2(k-1)}|D\phi|^{2}.
\end{align}
For a trace-free dual Killing tensor $\phi \in \Ga(\symkt)\cap \ker \clie \cap \ker \div$, where $|\phi|^{2} \neq 0$ there holds
\begin{align}\label{kato2}
|d|\phi||^{2} \leq \tfrac{k}{k+1}|D\phi|^{2}.
\end{align}
For $\phi \in \Ga(\symkt)\cap \ker \clie \cap \ker \klie$, where $|\phi|^{2} \neq 0$ there holds
\begin{align}\label{kato3}
|d|\phi||^{2} \leq \tfrac{k}{n+2(k-1)}|D\phi|^{2}.
\end{align}
\end{lemma}

\begin{proof}
Write $\sbl_{\clie}(Z)(\phi)$ for the symbol of $\clie$ applied to the vector $Z$ and $\phi \in \Ga(\symkt)$, and similarly for $\klie$ and $\div$. Write $(i(Z)\phi)_{i_{1}\dots i_{k-1}} = Z^{p}\phi_{pi_{1}\dots i_{k-1}}$ and suppose $Z$ has unit norm. The nonnegativity of $|\sbl_{\klie}(Z)(\phi)|^{2}$ in \eqref{kliesymbolnorm} yields
\begin{align}\label{klieinequality}
 \tfrac{n+2(k-2)}{n+k-3}|i(Z)\phi|^{2} \leq |\phi|^{2}.
\end{align}
(When $k = 1$ and $n = 2$ the coefficient of the pure trace terms in \eqref{kliesymbolnorm} should be understood in a limiting sense.)
Together \eqref{cliesymbolnorm} and \eqref{klieinequality} give 
\begin{align}\label{sblclie}
|\sbl_{\clie}(Z)(\phi)|^{2} \leq  \tfrac{1}{k+1}|\phi|^{2} + \tfrac{k(n+k-3)}{(k+1)(n+2(k-1))}|\phi|^{2} = \tfrac{n+k-2}{n+2(k-1)}|\phi|^{2}.
\end{align}
Contracting \eqref{domkl} with $\sbl_{D}(Z)(\om)$ and using the Cauchy-Schwarz inequality shows
\begin{align}\label{katoineq1}
\begin{aligned}
&\tfrac{1}{2}\imt(Z)d|\phi|^{2} = Z^{i}\phi^{i_{1}\dots i_{k}}D_{i}\phi_{i_{1}\dots i_{k}} = \lb \sbl_{D}(Z)(\phi), D\phi\ra \\
& = \lb\sbl_{\clie}(Z)(\phi), \clie(\phi)\ra + \tfrac{2k}{k+1} \lb \sbl_{\klie}(Z)(\phi), \klie(\phi)\ra + \tfrac{k(n+2(k-2))}{(n+k-3)(n+2(k-1))}\lb i(Z)\phi, \div(\phi)\ra,\\
& \leq |\sbl_{\clie}(Z)(\phi)||\clie(\phi)| + \tfrac{2k}{k+1}| \sbl_{\klie}(Z)(\phi)|| \klie(\phi)| + \tfrac{k(n+2(k-2))}{(n+k-3)(n+2(k-1))}|i(Z)\phi|| \div(\phi)|.
\end{aligned}
\end{align}
Suppose $\phi \in \Ga(\symkt) \cap \ker \klie\cap \ker \div$. By \eqref{normdom}, $|D\phi|^{2} = |\clie(\phi)|^{2}$. Substituting this and \eqref{sblclie} into \eqref{katoineq1} gives
\begin{align}
\begin{aligned}
|\phi|^{2}|\imt(Z)d|\phi||^{2} = \tfrac{1}{4}|\imt(Z)d|\phi|^{2}|^{2} &\leq |\sbl_{\clie}(Z)(\phi)|^{2}|\clie(\phi)|^{2} %
= |\sbl_{\clie}(Z)(\phi)|^{2}|D\phi|^{2} \leq   \tfrac{n+k-2}{n+2(k-1)}|\phi|^{2}|D\phi|^{2}.
\end{aligned}
\end{align}
This holds for all unit norm $Z$, so shows \eqref{kato}.

Suppose $\phi \in \Ga(\symkt) \cap \ker \clie\cap \ker \div$. By \eqref{normdom}, $|D\phi|^{2} = \tfrac{2k}{k+1}|\klie(\phi)|^{2}$, and, by \eqref{kliesymbolnorm}, $2|\sbl_{\klie}(Z)(\phi)|^{2} \leq |\phi|^{2}$. In \eqref{katoineq1} these give
\begin{align}
\begin{aligned}
|\phi|^{2}|\imt(Z)d|\phi||^{2} = \tfrac{1}{4}|\imt(Z)d|\phi|^{2}|^{2} &\leq (\tfrac{2k}{k+1})^{2}|\sbl_{\klie}(Z)(\phi)|^{2}|\klie(\phi)|^{2} \\& = \tfrac{2k}{k+1}|\sbl_{\klie}(Z)(\phi)|^{2}|D\phi|^{2} \leq   \tfrac{k}{k+1}|\phi|^{2}|D\phi|^{2}.
\end{aligned}
\end{align}
This holds for all unit norm $Z$, so shows \eqref{kato2}.

Suppose $\phi \in \Ga(\symkt) \cap \ker \clie\cap \ker \klie$. By \eqref{normdom}, $|D\phi|^{2} =\tfrac{k(n+2(k-2))}{(n+k-3)(n+2(k-1))}|\div(\phi)|^{2}$. With \eqref{klieinequality} in \eqref{katoineq1} this gives
\begin{align}
\begin{aligned}
|\phi|^{2}|\imt(Z)d|\phi||^{2}& = \tfrac{1}{4}|\imt(Z)d|\phi|^{2}|^{2}     \leq\left(\tfrac{k(n+2(k-2))}{(n+k-3)(n+2(k-1))}\right)^{2}|i(Z)\phi|^{2}| \div(\phi)|^{2} \\
&= \tfrac{k(n+2(k-2))}{(n+k-3)(n+2(k-1))}|i(Z)\phi|^{2}|D\phi|^{2} \leq  \tfrac{k}{n+2(k-1)}|\phi|^{2}|D\phi|^{2}.
\end{aligned}
\end{align}
This holds for all unit norm $Z$, so shows \eqref{kato3}.
\end{proof}

\begin{remark}
When $n = 2$ the inequalities \eqref{kato}-\eqref{kato3} are in fact equalities \cite[Section $3$]{Fox-2dahs}. 
\end{remark}

\begin{remark}\label{katoremark}
The proof of Lemma \ref{katolemma} shows that if $\phi \in \Ga(\symkt)$ then the largest eigenvalue $\mu$ of the symmetric two-tensor $\rictr(\phi \kwedge \phi)$ satisfies $\mu \leq \tfrac{n+k-3}{n+2(k-2)}|\phi|^{2}$. By \eqref{klieinequality}, for any vector field $X$ there holds $X^{i}X^{j}\rictr(\phi \kwedge \phi)_{ij} = |i(X)\rictr(\phi)|^{2} \leq \tfrac{n+k-3}{n+2(k-2)}|\phi|^{2}|X|^{2}$, which suffices to show the claim. This means $\rictr(\phi \kwedge \phi) \leq \tfrac{n+k-3}{n+2(k-2)}|\phi|^{2}h$. 
\end{remark}

\begin{lemma}\label{swlemma}
Let $h$ be a Riemannian metric on a manifold $M$ of dimension $n > 2$ and let $\om \in \Ga(\symkt)$. Wherever $\om \neq 0$ there hold
\begin{align}
\label{sharplapom}& |\om|^{(n+2(k-1))/(n+k-2)}\lap_{h}|\om|^{(n-2)/(n+k-2)}\geq \tfrac{n-2}{k(n+k-2)}\qR(\om), && \text{if}\,\, \om \in \ker \klie \cap \ker \div,\\
\label{sharplapom2}& |\om|^{(k+1)/k}\lap_{h}|\om|^{(k-1)/k}\geq -\tfrac{2(k-1)}{k+1}\lb \om, \kliea\klie(\om)\ra, && \text{if}\,\, \om \in \ker \clie \cap \ker \div,\\
\label{sharplapom3}& |\om|^{(n+2(k-1))/k}\lap_{h}|\om|^{(2-n)/k}\leq \tfrac{n-2}{k(n+k-2)}\qR(\om),&& \text{if}\,\, \om \in \ker \klie \cap \ker \clie.
\end{align}
\end{lemma}
\begin{proof}
Let $\om \in \Ga(\symkt)$. Wherever $|\om| > 0$ there holds
\begin{align}\label{lapnormla}
\tfrac{1}{2\la}|\om|^{2(1-\la)}\lap_{h}|\om|^{2\la} = \tfrac{1}{2}\lap_{h}|\om|^{2} + 2(\la - 1)|d|\om||^{2}.
\end{align}
Combining \eqref{lapnormla} with \eqref{katolemma}, \eqref{lapom}, \eqref{lapomdivlie}, \eqref{lapomliediv}, and Lemma \ref{katolemma} yields \eqref{sharplapom}-\eqref{sharplapom3}.
\end{proof}

Lemma \ref{swlemma} recovers a number of apparently unrelated classical results.

\begin{example}
If $h$ is flat and $f \in \cinf(M)$ is harmonic, then $\om_{i_{1}\dots i_{k}}= D_{i_{1}}\dots D_{i_{k}}f \in \Ga(\symkt) \cap \ker \klie \cap \ker \div$. By Lemma \ref{swlemma}, the function $|\om|^{p}$ is subharmonic for all $p \geq (n-2)/(n+k-2)$. For the flat Euclidean connection on $\rea^{n}$ and $k = 1$ this is \cite[Theorem $A$]{Stein-Weiss-harmonic}, and for $k >1$ it is \cite[Theorem $1$]{Calderon-Zygmund}. In the opposite direction, \cite[Theorem $2(b)$]{Stein-Weiss} shows that on flat Euclidean space the best $p$ for which $|\om|^{p}$ is subharmonic is $(n-2)/(n+k-2)$, and \cite[Theorem $2(a)$]{Stein-Weiss} shows that on flat Euclidean space, given any section $\om$ of $\symkt$ there is around every point a neighborhood $U$ and a harmonic function $f \in \cinf(U)$ such that on $U$ there holds $\om_{i_{1}\dots i_{k}}= D_{i_{1}}\dots D_{i_{k}}f$.

For general $h$, if $f \in \cinf(M)$ is harmonic then $df \in \ker \klie \cap \ker \div$, and Lemma \ref{swlemma} shows that 
\begin{align}
(n-1) |df|^{n/(n-1)}\lap_{h}|df|^{(n-2)/(n-1)}\geq (n-2)\sric(df, df).
\end{align} 
If $h$ has nonnegative Ricci curvature it follows that $|df|^{p}$ is subharmonic for any $p \geq (n-2)/(n-1)$. 
\end{example}

The rest of this section records some vanishing theorems for conformal Killing and trace-free Codazzi tensors. These are relevant here for identifying contexts in which the equations \eqref{projectivehiggsintro} or \eqref{stressenergyintro} admit only solutions $(h, \om)$ that are trivial in the sense that $\om$ vanish identically. 

These vanishing theorems for conformal Killing tensors and trace and divergence free Codazzi tensors are analogous to the somewhat stronger vanishing theorems for symmetric tensors on Kähler manifolds obtained by S. Kobayashi \cite{Kobayashi-holomorphicsymmetric, Kobayashi-holomorphictensor}. For compact Riemann surfaces such tensors are the real parts of sections of powers of the canonical bundle so more precise statements are available; see \cite{Fox-2dahs}. 

The statements of these vanishing theorems have several subtle aspects. First, the Weitzenböck identities lead naturally to hypotheses involving the definiteness of the quadratic form determined by the curvature operator on $\Ga(\symkt)$. Different authors define the curvature operator slightly differently; in particular it may or may not involve a term in the Ricci curvature (as it does here). Second, it is desirable to formulate hypotheses in terms of sectional curvature and some extra work is required to relate a condition on sectional curvature with the definiteness of the quadratic form determined by the curvature operator (however defined). Third, when the hypothesis is positivity or negativity of the sectional curvature, this implies the same condition for the Ricci curvature, and implies the definiteness of the curvature operator on trace-free symmetric tensors, whatever the particular definition used for the latter. Finally, when the manifold is not compact the arguments require the Kato inequalities, and the sharpest results require the refined Kato inequalities. 

The first vanishing theorem for trace-free Codazzi tensors on compact manifolds is due to M. Berger and D. Ebin \cite{Berger-Ebin} for rank two trace-free Codazzi tensors. Its extension to higher rank Codazzi tensors is due in the compact case to S. Stepanov \cite[Theorem $2$]{Stepanov}. Subsequent refinements can be found in \cite{Fox-ahs, Heil-Moroianu-Semmelmann, Shandra-Stepanov-Mikes, Stepanov-Tsyganok}. A good summary of previous work is in \cite{Stepanov-Tsyganok}. The strongest result in this direction is Theorem \ref{completecodazzitheorem} below, which is the slight sharpening of \cite[Theorem $5$]{Stepanov-Tsyganok} and \cite{Shandra-Stepanov-Mikes} that results from using the refined Kato inequality \eqref{kato} in place of the usual Kato inequality. The analogous vanishing theorem for conformal Killing tensors is in \cite[Theorem $6.2$]{Fox-ahs}, while the improved result with a hypothesis on sectional curvature was proved in \cite[Theorem $1.6$]{Dairbekov-Sharafutdinov}.

It is convenient to say that $\qR$ is positive or negative (semi-)definite on $\symk$ or $\symkt$ if it is positive (semi-)definite or negative (semi-)definite as a quadratic form on $\Ga(\symk)$ or $\Ga(\symkt)$. Since, by Lemma \eqref{hycommutelemma}, $\qR(h^{\sprod k}) = 0$ for any $k \geq 1$, $\qR$ is not definite on $\Ga(S^{2k}(\ctm))$ for any $k \geq 1$.

\begin{theorem}[\cite{Fox-ahs}]\label{bochnerliedivtheorem}
Let $(M, h)$ be a compact Riemannian manifold of dimension $n > 2$.
\begin{enumerate}
\item\label{bcc1} If $\qR$ is nonnegative on $S^{k}_{0}(\ctm)$ then any $\om \in \Ga(\symkt) \cap \ker \klie \cap \ker \div$ is parallel. If moreover $\qR$ is at some point of $M$ strictly positive on $S^{k}_{0}(\ctm)$ then $\Ga(\symkt) \cap \ker \klie \cap \ker \div = \{0\}$. 
\item\label{bcc2} If $\qR$ is nonpositive on $S^{k}_{0}(\ctm)$ then any rank $k$ conformal Killing tensor is parallel, and if, moreover, $\qR$ is at some point of $M$ strictly negative on $S^{k}_{0}(\ctm)$, then any rank $k$ conformal Killing tensor is identically zero.
\end{enumerate}
\end{theorem}

\begin{proof}
For a compactly supported $\om \in \Ga(\symkt)$, integrating any of \eqref{lapomsq}-\eqref{lapomliediv} by parts against the Riemannian volume $\vol_{h}$ and simplifying the result using \eqref{normdom} yields
\begin{align}\label{bigbochner}
\begin{aligned}
\tfrac{2}{k+1}\iln|\klie(\om)\irn^{2} &  + \tfrac{(n+k-2)(n+2(k-2))}{(n+k-3)(n+2(k-1))}\iln\div(\om)\irn^{2} - \iln\clie(\om)\irn^{2} =   \tfrac{1}{k}\int_{M}\qR(\om) \,d\vol_{h}.
\end{aligned}	
\end{align}
The identity \eqref{bigbochner} generalizes the usual integrated Bochner identities for harmonic one-forms and conformal Killing vector fields. 
If $\om \in \ker \klie \cap \ker \div$ and $\qR \geq 0$ then \eqref{bigbochner} shows that $\om \in \ker \clie$ and from \eqref{normdom} it follows that $D\om = 0$. If moreover $\qR$ is somewhere positive then \eqref{lapomliediv} shows $\om = 0$. If $\om \in \ker \clie$ and $\qR \leq 0$ then \eqref{bigbochner} shows that $\om \in \ker \klie \cap \ker \div$ and from \eqref{normdom} it follows that $D\om = 0$. If, moreover, $\qR$ is somewhere negative then \eqref{lapomdivlie} shows $\om = 0$.
\end{proof}

\begin{remark}
A result very similar to \eqref{bcc1} of Theorem \ref{bochnerliedivtheorem} was obtained in \cite[Theorem $2$]{Stepanov}. The difference is that the operator $\op{\sR}$ and corresponding quadratic form $\qR$ considered in \cite{Stepanov} are slightly different (they differ by a term involving the Ricci curvature). The claim here is very slightly more general, but, when the curvature is assumed to have a definite sign, since this sign is inherited by the Ricci tensor, the ambit of application of the claims is the same. 
\end{remark}

\begin{corollary}\label{bochnercorollary}
Let $h$ be a Riemannian metric on a compact manifold $M$ of dimension $n > 2$.
\begin{enumerate}
\item\label{bcc1b}(Stepanov, \cite[Theorem $2$]{Stepanov}; see also \cite[Corollary $1$]{Shandra-Stepanov-Mikes}) If $h$ has nonnegative sectional curvature, then a trace-free Codazzi tensor $\om \in \Ga(\symkt) \cap \ker \klie \cap \ker \div$ is parallel. If, moreover, the sectional curvature is strictly positive at some point of $M$ then $\Ga(\symkt) \cap \ker \klie \cap \ker \div = \{0\}$. 
\item(Dairbekov-Sharafutdinov, \cite[Theorem $1.6$]{Dairbekov-Sharafutdinov}; see also \cite[Proposition $6.6$]{Heil-Moroianu-Semmelmann})\label{bcc2b} If $h$ has nonpositive sectional curvature, then a rank $k$ conformal Killing tensor is parallel, and if the sectional curvature is moreover strictly negative at some point of $M$, then a rank $k$ conformal Killing tensor is identically zero.
\end{enumerate}
\end{corollary}
\begin{proof}
Both claims follow from Theorem \ref{bochnerliedivtheorem} once it is known that a sign condition on the sectional curvature of $h$ implies the same sign condition for $\qR$ on $\Ga(\symkt)$. When $k = 2$ this follows from the proof of \cite[Proposition $6.1$]{Berger-Ebin}, but that argument (which is direct) does not extend straightforwardly to the case $k > 2$. For conformal Killing tensors, claim \eqref{bcc2} is \cite[Theorem $1.6$]{Dairbekov-Sharafutdinov}. A proof of the result of Dairbekov-Sharafutdinov for all $k$ based on Weitzenböck formulas as in Theorem \ref{bochnerliedivtheorem} was given as \cite[Proposition $6.6$]{Heil-Moroianu-Semmelmann}; they show via an integration in the fibers argument that the condition \eqref{bcc2} of Theorem \ref{bochnerliedivtheorem} follows from the nonpositivity of the sectional curvature. Their argument works equally well assuming nonnegativity of the sectional curvature, and combined with Theorem \ref{bochnerliedivtheorem}, this yields corollary \ref{bochnercorollary}. Alternatively, in the case of trace-free Codazzi tensors, the proof of \cite[Corollary $1$]{Shandra-Stepanov-Mikes} shows how to deduce the required nonnegativity for $k > 2$ from the Berger-Ebin argument.
\end{proof}

\begin{remark}
The $n = 2$ case of Corollary \ref{bochnercorollary} is proved in \cite[section $3$]{Fox-2dahs}. 
\end{remark}

Theorem \ref{completecodazzitheorem} refines \cite[Theorem $5$]{Stepanov-Tsyganok}, where it is shown there holds the same statement for all $p > q$ for some $q > 0$ that is not determined. The specific $q$ obtained here follows from the refined Kato inequality. 
\begin{theorem}\label{completecodazzitheorem}
Let $(M, h)$ be a complete Riemannian manifold of dimension $n > 2$ and let $k\geq 2$. If $\qR$ is nonnegative on $\symkt$ (in particular if $h$ has nonnegative sectional curvature) and a trace-free Codazzi tensor $\om \in \ker \klie \cap \ker \div \cap \Ga(\symkt)$ is in $L^{p}(M, \vol_{h})$ for $p > \tfrac{n-2}{n+k-2}$, then:
\begin{enumerate}
\item If $\vol_{h}(M) = \infty$ then $\om$ vanishes identically.
\item If $\vol_{h}(M)  < \infty$ then $\om$ is parallel and $\qR(\om)$ vanishes identically. In particular if at some point of $M$ $\qR$ is strictly positive on $\symkt$, then $\om$ vanishes identically. 
\end{enumerate}
\end{theorem}
\begin{proof}
Because $\qR$ is nonnegative on $\symkt$ , by Lemma \ref{swlemma}, $u = |\om|^{(n-2)/(n+k-2)}$ is a nonnegative subharmonic function on $M$. Because $h$ is complete, by a theorem of Yau \cite{Yau-function-theoretic}, if $u \in L^{q}(M, \vol_{h})$ for $q \in (1, \infty)$, then $u$ must be constant. Hence if $\om \in L^{p}(M, \vol_{h})$ for $p > \tfrac{n-2}{n+k-2}$, then $|\om|$ is constant. If $\vol_{h}(M)  = \infty$ then no constant function is in $L^{p}(\vol_{h})$ for any $p > 0$, so $\om$ vanishes identically. Suppose $\vol_{h}(M) < \infty$. By \eqref{lapomliediv}, $\qR(\om) = 0$ and $D\om = 0$.
\end{proof}

The result of Theorem \ref{completecodazzitheorem} is sharp in the following sense. For $k \geq 2$, on flat Euclidean space let $P$ be a harmonic polynomial homogeneous of degree $k$. Then $\om = D^{(k)}P$ is a parallel trace-free Codazzi tensor. Because it is parallel it descends to any flat torus. Evidently in this case $\om$ is in $L^{p}$ for all $p > 0$.

\subsection{Bounding norms of Kulkarni-Nomizu products}\label{boundsection}
What follows obtains estimates of the quantities $\skmax_{n, k}$ and $\skmin_{n, k}$ defined in \eqref{skdefined}.

A tensor $\om \in S^{k}_{0}\std$ is \emph{critical} if it is a critical point of the restriction to the unit sphere $\{\om \in S^{k}_{0}\std: |\om|^{2} = 1\}$ of the expression \eqref{quadom}. By Lemma \ref{criticalomlemma}, $\skmax_{n, k}$ and $\skmin_{n, k}$ are eigenvalues of the operator $ \op{\om \kwedge \om}$ determined by a critical $\om$.
\begin{lemma}\label{criticalomlemma}
A critical $\om \in S^{k}_{0}\std$ satisfies $\op{\om \kwedge \om}(\om) = \theta \om$ for some $\theta \in \rea$. 
\end{lemma}

\begin{proof}
Consider a one-parameter family $\om(t) \in S^{k}_{0}\std$ and let $\dot{\om} = \tfrac{d}{dt}\om$. By \eqref{qyalbe} applied twice,
\begin{align}
\begin{aligned}
\tfrac{d}{dt}\lb \op{\om \kwedge \om}(\om), \om \ra & = \tfrac{d}{dt}\left( \binom{k}{2}|\om \kwedge \om|^{2} + k|\rictr(\om \kwedge \om)|^{2}\right)\\
& = 4\left( \binom{k}{2}\lb \om \kwedge \om, \om \kwedge \dot{\om}\ra + k\lb \rictr(\om \kwedge \om), \rictr(\om \kwedge \dot{\om})\ra\right)= 4\lb \op{\om \kwedge \om}(\om), \dot{\om} \ra,
\end{aligned}
\end{align}
while $\tfrac{d}{dt}|\om|^{4} = 4|\om|^{2}\lb \om, \dot{\om}\ra$. The claim follows.
\end{proof}

The entire $O(n)$ orbit of a critical $\om$ comprises critical points. It would be interesting to characterize these critical orbits algebraically.

The $k  = 2$ case is particularly simple.

\begin{lemma}\label{2dsymlemma}
Let $(\ste, h)$ be a Euclidean vector space of dimension $n \geq 2$. Any $\om \in S^{2}_{0}\std$ satisfies
\begin{align}
\label{lijklnormbk2b}
&|\om|^{4}  = \tfrac{1}{2}\lb \op{\om \kwedge \om}(\om), \om \ra=  \tfrac{1}{2}|\om\kwedge \om|^{2} + |\rictr(\om\kwedge \om)|^{2} .
\end{align}
In particular $\skmax_{n, 2} = 1 = \skmin_{n, 2}$ for all $n \geq 2$.
\end{lemma}

\begin{proof}
A straightforward computation using the definition of $\op{\om \kwedge \om}$ shows that $\op{\om \kwedge \om}(\om) = 2|\om|^{2} \om$. With \eqref{qyalbe} this yields \eqref{lijklnormbk2b}.
\end{proof}

Let $(\ste, h)$ be a Euclidean vector space of dimension $n \geq 2$. A trace-free symmetric tensor $\om \in S^{k}_{0}\std$ is \emph{projectively flat} if $\om \kwedge \om$ is a multiple of $h \kwedge h$. 

\begin{lemma}\label{lowerkwedgelemma}
Let $(\ste, h)$ be a Euclidean vector space of dimension $n \geq 3$. For $\om \in S^{k}_{0}\std$,
\begin{align}\label{loweromtriple}
\tfrac{1}{k}\lb \op{\om \kwedge \om}(\om), \om \ra =\tfrac{k-1}{2}|\om \kwedge \om|^{2} + |\rictr(\om \kwedge \om)|^{2} \geq \tfrac{n+k-2}{n(n-1)}|\om|^{4},
\end{align}
with equality if and only if $\om$ is projectively flat. In particular, $\skmin_{n, k} \geq  \tfrac{n+k-2}{n(n-1)}$. %
\end{lemma}
\begin{proof}
By \eqref{tfweyl}, for $\sY \in \mcurv(\std)$, 
\begin{align}\label{tfweylnorm}
\begin{aligned}
|\sY|^{2}&= |\tf \sY|^{2}+ \tfrac{4}{n-2}|\tf \rictr(\sY)|^{2}+ \tfrac{2}{n(n-1)}(\scal(\sY))^{2}.
\end{aligned}
\end{align}
By \eqref{tfweylnorm}, for $\om \in S^{k}_{0}\std$ there holds
\begin{align}\label{tfweylomom}
\begin{aligned}
|\om \kwedge \om|^{2}&=  |\tf (\om \kwedge \om)|^{2}+ \tfrac{4}{n-2}|\tf \rictr(\om \kwedge \om)|^{2} + \tfrac{2}{n(n-1)}|\om|^{4}_{h}.
\end{aligned}
\end{align}
From \eqref{tfweylomom} and $|\rictr(\om \kwedge \om)|^{2} = |\tf \rictr(\om \kwedge \om)|^{2} + n^{-1}|\om|^{4}$ there follows
\begin{align}\label{tfomom}
\tfrac{k-1}{2}|\om \kwedge \om|^{2} + |\rictr(\om \kwedge \om)|^{2} = \tfrac{k-1}{2}|\tf (\om \kwedge \om)|^{2}+ \tfrac{n+2(k-2)}{n-2}|\tf \rictr(\om \kwedge \om)|^{2} + \tfrac{n+k-2}{n(n-1)}|\om|^{4},
\end{align}
for $\om \in S^{k}_{0}\std$. The equality \eqref{tfomom} implies \eqref{loweromtriple} and equality holds in \eqref{loweromtriple} if and only if both $\tf (\om \kwedge \om) = 0$ and $\tf \rictr(\om \kwedge \om) =0$, which by \eqref{tfweyl} imply that $\om \kwedge \om$ is a multiple of $h \kwedge h$. 
\end{proof}

Note that a priori it is not evident that there exist nonzero projectively flat elements of $S^{k}_{0}\std$. Examples \ref{k2pfexample}-\ref{k4pfexample} address this for $k \in \{2, 3, 4\}$.

\begin{example}\label{k2pfexample}
By Lemma \ref{2dsymlemma}, $\skmin_{n, 2} = 1$ for $n \geq 2$, so by Lemma \ref{lowerkwedgelemma} no nonzero element of $S^{2}_{0}\std$ is projectively flat provided $n = \dim \ste \geq 3$. A direct proof of this can be given as follows. Suppose $\om \in S^{2}_{0}\std$ were projectively flat so that there be $\ka \in \rea$ such that $\om \kwedge \om = \ka h \kwedge h$. Let $\{e_{1}, \dots, e_{n}\}$ be an $h$-orthonormal basis of $\ste$ with respect to which $\om$ is diagonal.  Let $w_{i} = \om(e_{i}, e_{i})$. Then $w_{i}w_{j} = (\om \kwedge \om)(e_{i}, e_{j}, e_{i}, e_{j}) = \ka$ for all $1 \leq i \neq j \leq n$. Because $\om$ is trace-free, $0 = w_{i}\sum_{j = 1}^{n}w_{j} =w_{i}^{2} + (n-1)\ka$, so that $0 \geq \ka = w_{i}w_{j}$ and $w_{i}^{2} = (1-n)\ka$. Were $\ka < 0$ then, for distinct $i$ and $j$, $w_{i}$ and $w_{j}$ would be nonzero with opposite signs but equal squares. This is impossible if $n > 2$.
\end{example}

\begin{example}\label{k3pfexample}
It follows from \cite[Theorem $1.10$]{Fox-cubicpoly} that, when $n = \dim \ste > 1$, in $S^{3}_{0}\std$ there is a unique $CO(h)$ orbit comprising nonzero projectively flat tensors. An example of a nonzero projectively flat tensor is as follows. Let $\{e_{1}, \dots, e_{n}\}$ be a basis of $\ste$. 
Define $h \in S^{2}\std$ by $h(e_{i}, e_{i}) = n$ and $h(e_{i}, e_{j}) = -1$ for $i \neq j$ and extending bilinearly.
By construction the Gram matrix in the given basis is positive definite, so $h$ is a Euclidean metric.
Define a commutative nonassociative product on $\ste$ by $e_{i}\mlt e_{i} = (n-1)e_{i}$ and $e_{i}\mlt e_{j} = -e_{i} - e_{j}$ for $i \neq j$ and extending bilinearly. Define $\om \in \tensor^{3}\std$ by $\om(x, y, z) = h(x\mlt y, z)$, so that $\om(e_{i}, e_{i}, e_{i}) = n(n-1)$, $\om(e_{i}, e_{i}, e_{j}) = -(n-1)$, and $\om(e_{i}, e_{j}, e_{k}) = 2$ for distinct indices $i$, $j$, and $k$. It is straightforward to check that $\om$ is  completely symmetric and $h$-trace-free so that $\om \in S^{3}_{0}\std$. Define $L:\ste \to \eno(\ste)$ by $L(x)y = x\mlt y$. The complete symmetry of $\om$ is equivalent to $h(L(x)u, v)) = h(u, L(x)v)$. 
It can be checked that $[L(a), L(b)]c = h(a, c)b - h(b, c)a$. It follows that
\begin{align}
\begin{aligned}
(\om \kwedge \om)(a, b, c, d) &= h(c\mlt a, b \mlt d) - h(c \mlt b, a \mlt d) = h([L(b), L(a)]c, d)\\
&= h(b, c)h(a, d) - h(a, c)h(b, d) = -(h\kwedge h)(a, b, c, d),
\end{aligned}
\end{align}
which shows that $\om \kwedge \om = - h\kwedge h$ so that $\om$ is projectively flat. There follow $\rictr(\om \kwedge \om) = (n-1)h$, $|\om|^{2} = \scal(\om\kwedge \om) = n(n-1)$, $|\om \kwedge \om|^{2} = 2n(n-1)$, and $|\rictr(\om \kwedge \om) |^{2} = n(n-1)^{2}$, from which it follows that equality holds in \eqref{loweromtriple}. This shows that $\skmin_{n, 3} = \tfrac{n+1}{n(n-1)}$ when $n >1$.
\end{example}

\begin{example}\label{k4pfexample}
Here is an example of a projectively flat tensor in $S^{4}_{0}\std$ when $n = \dim \ste > 2$. This shows that $\skmin_{n, 4} = \tfrac{n+2}{n(n-1)}$ when $n > 2$.  Let $\{e_{1}, \dots, e_{n}\}$ be a basis of $\ste$. Define a Euclidean metric $h$ as in Example \eqref{k2pfexample}. In what follows, $i$, $j$, $k$, and $l$ denote distinct indices from $\{1, \dots, n\}$. Define a commutative ternary product $[\dum, \dum, \dum]:\ste \times \ste \times \ste \to \ste$ by extending multilinearly
\begin{align}
&[e_{i}, e_{i}, e_{i}] = (n-1)(n-2)e_{i},&& [e_{i}, e_{i}, e_{j}] = -(n-2)(e_{i} + e_{j}), && [e_{i}, e_{j}, e_{k}] = 2(e_{i} + e_{j} + e_{k}).
\end{align}
Define $\om \in\tensor^{4}\std$ by $\om(a, b, c, d) = h([a, b, c], d)$ for $a, b, c, d\in \ste$, so that $\om(e_{i}, e_{i}, e_{i}, e_{i}) = n(n-1)(n-2)$, $\om(e_{i}, e_{i}, e_{i}, e_{j}) = -(n-1)(n-2) = \om(e_{i}, e_{i}, e_{j}, e_{j})$, $ \om(e_{i}, e_{i}, e_{j}, e_{k}) = 2(n-2) $, and $ \om(e_{i}, e_{j}, e_{k}, e_{l})= -6$. It is straightforward to check that $\om$ is completely symmetric and $h$-trace-free, so $\om \in S^{4}_{0}\std$. Define $L: \ste \times \ste \to \eno(\ste)$ by $L(a, b)c = [a, b, c]$. The complete symmetry of $\om$ is equivalent to $h(L(x, y)u, v) = h(u, L(x, y)v)$. Using $(\om \kwedge \om)(a, b, c, d) = \tr (L(c, a)L(b, d) - L(c, b)L(a, d))$ it is straightforward to calculate that $\om \kwedge \om = -\tfrac{(n+2)(n-2)}{n+1}h \kwedge h$, showing that $\om$ is projectively flat. Whether every projectively flat element of $S^{4}_{0}\std$ is in the $CO(h)$ orbit of the $\om$ constructed here has not been addressed, although it seems likely this is the case.
\end{example}

For $k > 4$, I do not know if there exist projectively flat tensors in $S^{k}_{0}\std$. There is an evident pattern in the construction of such tensors in the $k = 3$ and $k = 4$ cases that it appears probable continues for $k > 4$, but the direct computations used to check these examples need to be replaced by some more conceptual argument involving $n$-ary operations, and developing this in the detail necessary would not be germane here.

\begin{lemma}\label{k3normlemma}
Let $(\ste, h)$ be a Euclidean vector space of dimension $n \geq 2$. For $\om \in S^{3}_{0}\std$,
\begin{align}
\label{lijklnormk3}
&\tfrac{2n-1}{n}|\om|^{4} \geq |\om \kwedge \om|^{2} + |\rictr(\om \kwedge \om)|^{2}.
\end{align}
so that $\skmax_{n, 3} \leq \frac{2n-1}{n}$.
\end{lemma}

\begin{proof}
Let $\{e(1), \dots, e(n)\}$ be an $h$-orthonormal basis of $\ste$. In terms of the endomorphisms $\om(i)_{j}\,^{k} = e(i)^{p}\om_{pj}\,^{k} \in \eno(\ste)$, $1 \leq i \leq n$, 
\begin{align}\label{co0}
\begin{aligned}
[\om(i), \om(j)]_{kl} &= \om(i)_{pl}\om(j)_{k}\,^{p} - \om(j)_{pl}\om(i)_{k}\,^{p} = e(i)^{a}e(j)^{b}(\om_{pla}\om_{kb}\,^{p} - \om_{plb}\om_{ka}\,^{p})\\
& = -2e(i)^{a}e(j)^{b}\om_{k[a}\,^{p}\om_{b]lp} = -e(i)^{a}e(j)^{b}(\om \kwedge \om)_{abkl}.
\end{aligned}
\end{align}
(It is because of the last two equalities of \eqref{co0} that this proof works only for the $k = 3$ case.)
By \cite[Lemma $1$]{Chern-Docarmo-Kobayashi}, for symmetric endomorphisms $A_{i}\,^{j}$ and $B_{i}\,^{j}$ of $\ste$ there holds $|[A, B]|^{2}\leq 2|A|^{2}|B|^{2}$, and applied with \eqref{co0} this yields
\begin{align}\label{co1}
\begin{aligned}
|\om \kwedge \om|^{2}& = \sum_{i= 1}^{n}\sum_{j = 1}^{n}| [\om(i), \om(j)]|^{2}  = \sum_{i= 1}^{n}\sum_{j \neq i}| [\om(i), \om(j)]|^{2}\leq 2\sum_{i= 1}^{n}\sum_{j\neq i}|\om(i)|^{2}|\om(j)|^{2}.
\end{aligned}
\end{align}
There holds $\rictr(\om \kwedge \om)_{ab}e(i)^{a}e(j)^{b}  = \lb \om(i), \om(j)\ra$. Because $\rictr(\om \kwedge \om)$ is symmetric, the $h$-orthonormal basis $\{e(1), \dots, e(n)\}$ can be chosen to be also orthogonal with respect to $\rictr(\om\kwedge \om)$, which entails that $ \lb \om(i), \om(j)\ra = 0$ if $i \neq j$. In this case,
\begin{align}\label{co5}
|\rictr(\om \kwedge \om)|^{2} & = \sum_{i = 1}^{n}\sum_{j = 1}^{n}\lb \om(i), \om(j)\ra^{2} = \sum_{i = 1}^{n}|\om(i)|^{4}.
\end{align}
Combining \eqref{co1} and \eqref{co5} yields
\begin{align}\label{co6}
\begin{aligned}
|\om \kwedge \om|^{2} & + |\rictr(\om \kwedge \om)|^{2} \leq 2\sum_{i= 1}^{n}\sum_{j\neq i}|\om(i)|^{2}|\om(j)|^{2} + \sum_{i = 1}^{n}|\om(i)|^{4}.
\end{aligned}
\end{align}
There hold
\begin{align}\label{co2}
\begin{aligned}
\sum_{i = 1}^{n}|\om(i)|^{4} & = \left(\sum_{i = 1}^{n}|\om(i)|^{2}\right)^{2} - \sum_{i= 1}^{n}\sum_{j\neq i}|\om(i)|^{2}|\om(j)|^{2}= |\om|^{4} - \sum_{i= 1}^{n}\sum_{j\neq i}|\om(i)|^{2}|\om(j)|^{2},\\
\sum_{i = 1}^{n}|\om(i)|^{4} & =\tfrac{1}{2(n-1)}\sum_{i = 1}^{n}\sum_{j \neq i}\left(|\om(i)|^{2} - |\om(j)|^{2}\right)^{2} + \tfrac{1}{n-1}\sum_{i = 1}^{n}\sum_{j \neq i}|\om(i)|^{2}|\om(j)|^{2}.
\end{aligned}
\end{align}
The observations \eqref{co2}, which are key to the proof, appear in some equivalent form in the proof of \cite[Theorem $4.2$]{Chen-Ogiue}.
Summing $\tfrac{2n-1}{n}$ times the first equation of \eqref{co2} with $\tfrac{1-n}{n}$ times the second equation of \eqref{co2} yields
\begin{align}\label{co3}
\sum_{i = 1}^{n}|\om(i)|^{4} + 2 \sum_{i= 1}^{n}\sum_{j\neq i}|\om(i)|^{2}|\om(j)|^{2}= \tfrac{2n-1}{n} |\om|^{4} - \tfrac{1}{2n}\sum_{i = 1}^{n}\sum_{j \neq i}\left(|\om(i)|^{2} - |\om(j)|^{2}\right)^{2}.
\end{align}
Combining \eqref{co6} and \eqref{co3} yields
\begin{align}\label{co7}
\begin{aligned}
|\om \kwedge \om|^{2} & + |\rictr(\om \kwedge \om)|^{2} \leq 2\sum_{i= 1}^{n}\sum_{j\neq i}|\om(i)|^{2}|\om(j)|^{2} + \sum_{i = 1}^{n}|\om(i)|^{4}\\
& = \tfrac{2n-1}{n} |\om|^{4} - \tfrac{1}{2n}\sum_{i = 1}^{n}\sum_{j \neq i}\left(|\om(i)|^{2} - |\om(j)|^{2}\right)^{2}\leq  \tfrac{2n-1}{n} |\om|^{4},
\end{aligned}
\end{align}
which proves \eqref{lijklnormk3}. 
\end{proof}

I do not know the characterization of tensors saturating the bound \eqref{lijklnormk3}.

%%% SAVE: Following is correct bu not used anywhere
%%%%%% Probably can omit following
%\begin{corollary}
%Let $(\ste, h)$ be a Euclidean vector space of dimension $n \geq 2$. For $\om \in S^{3}_{0}\std$,
%\begin{align}
%\label{lijklnormbk3}
%&2\tfrac{n-2}{n-1}|\om|^{4} \geq |\tf(\om \kwedge \om)|^{2} + \tfrac{n+2}{n-2}|\mr{\rictr(\om\kwedge \om)}|^{2} \geq |\tf(\om \kwedge \om)|^{2}.
%\end{align}
%\end{corollary}
%\begin{proof}
%Combining \eqref{lijklnormk3} with \eqref{tfweylomom} yields \eqref{lijklnormbk3}.
%\end{proof}

\begin{lemma}
Let $(\ste, h)$ be a Euclidean vector space of dimension $n \geq 2$. For $k \geq 2$ and $\om \in S^{k}_{0}\std$ there hold
\begin{align}
\label{lijnorm}
&\tfrac{n+k-3}{n + 2(k-2)}|\om|^{4}\geq |\rictr(\om \kwedge \om)|^{2}\geq \tfrac{1}{n}|\om|^{4},&\\
\label{lijnormb}
&\tfrac{(n-2)(n+k-2)}{n(n + 2(k-2))}|\om|^{4}\geq |\tf\rictr(\om \kwedge \om)|^{2}.
\end{align}
\end{lemma}
\begin{proof}
For $x_{i} \in \std$, the nonnegativity of the squared norm of 
\begin{align}
x_{[i}\om_{j]i_{1}\dots i_{k-1}} - \tfrac{1}{n+k-3}\sum_{s = 1}^{k-1}h_{i_{s}[i}\om_{j]i_{1}\dots \hat{i}_{s} \dots i_{k-1}p}x^{p},
\end{align}
(where $\hat{i}_{s}$ denotes the omission of the index $i_{s}$) yields
\begin{align}
\begin{aligned}
0 &\leq 2x^{i}\om^{ji_{1}\dots i_{k-1}}\left(x_{[i}\om_{j]i_{1}\dots i_{k-1}} - \tfrac{1}{n+k-3}\sum_{s = 1}^{k-1} h_{i_{s}[i}\om_{j]i_{1}\dots \hat{i}_{s} \dots i_{k-1}p} x^{p} \right) \\
&= |x|^{2}|\om|^{2}- |\imt(x)\om|^{2} - \tfrac{k-1}{n+k-3}|\imt(x)\om|^{2}= |x|^{2}|\om|^{2}- \tfrac{n + 2(k-2)}{n+k-3}|\imt(x)\om|^{2},
\end{aligned}
\end{align}
so that
\begin{align}
\rictr(\om \kwedge \om)_{ij}x^{i}x^{j} = |\imt(x)\om|^{2}\leq \tfrac{n+k-3}{n+2(k-2)}|x|^{2}|\om|^{2}.
\end{align}
This shows $\tfrac{n+k-3}{n+2(k-2)}|x|^{2}|\om|^{2}h_{ij} - \rictr(\om \kwedge \om)_{ij}$ is positive semidefinite. Because $\rictr(\om \kwedge \om)_{ij}$ is also positive semidefinite and the endomorphisms $\tfrac{n+k-3}{n+2(k-2)}|\om|^{2}\delta_{i}\,^{j} - \rictr(\om \kwedge\om)_{i}\,^{j}$ and $\rictr(\om \kwedge \om)_{i}\,^{j}$ commute, contracting with $\rictr(\om \kwedge \om)^{ij}$ yields
\begin{align}\label{tfricnorm}
\tfrac{n+k-3}{n+2(k-2)}|\om|^{4}\geq |\rictr(\om \kwedge \om)|^{2}= |\tf\rictr(\om\kwedge \om)|^{2}+ \tfrac{1}{n}|\om|^{4}_{h},
\end{align}
from which the inequalities \eqref{lijnorm} and \eqref{lijnormb} follow.
\end{proof}

\begin{lemma}
Let $(\ste, h)$ be a Euclidean vector space of dimension $n \geq 2$. For $k \geq 2$ and $\om \in S^{k}_{0}\std$ there holds
\begin{align}
\label{lijklnormupper} |\om \kwedge \om|^{2} \leq 4|\om|^{4}.
\end{align}
\end{lemma}
\begin{proof}
For $\tau_{ij} \in S^{2}_{0}\std$, 
\begin{align}
\begin{aligned}
0 &\leq (\tau_{ij}\om_{klQ} - \tau_{kl}\om_{ijqQ})(\tau^{ij}\om^{klQ} - \tau^{kl}\om^{ijQ})
 = |\tau|^{2}|\om|^{2}- |\imt(\tau)\om|^{2} \\
 &= \tau^{ij}\tau^{kl}\left( |\om|^{2}h_{k(i}h_{j)l} - \om_{ijQ}\om_{kl}\,^{Q}\right),
\end{aligned}
\end{align}
where $\imt(\tau)\om$ is as in \eqref{imtdefined} and $Q$ stands for $q_{1}\dots q_{k-2}$. This means that the symmetric quadratic form defined on $S^{2}_{0}\std$ by contracting $\tau^{ij}\tau^{kl}$ with the tensor $|\om|^{2}h_{k(i}h_{j)l} - \om_{ijQ}\om_{kl}\,^{Q}$ is positive semidefinite. Since the quadratic form defined on $S^{2}_{0}\std$ by contracting $\tau^{ij}\tau^{kl}$ with $\om_{ijQ}\om_{kl}\,^{Q}$ is also positive semidefinite and the corresponding endomorphisms commute it follows that
\begin{align}\label{om4}
0 \leq (|\om|^{2}h_{k(i}h_{j)l} - \om_{ijQ}\om_{kl}\,^{Q})\om^{ij}\,_{P}\om^{kl}\,^{P} = |\om|^{4} -  \om_{ijQ}\om_{kl}\,^{Q}\om^{ij}\,_{P}\om^{kl}\,^{P},
\end{align}
(where $P$ stands for $p_{1}\dots p_{k-2}$). 
Define $\mt = \om_{j}\,^{kQ}\om_{k}\,^{iP}\om_{i}\,^{l}\,_{Q}\om_{l}\,^{j}\,_{P}$. 
From 
\begin{align}
\begin{aligned}
0 &\leq 2\om_{k(i}\,^{P}\om_{j)lP}\om^{k(i}\,_{Q}\om^{j)lQ} = \om_{kiP}\om_{jl}\,^{P}\om^{ki}\,_{Q}\om^{jlQ} + \mt
\end{aligned}
\end{align}
it follows that
\begin{align}\label{lijkllijkl}
\begin{aligned}
|\om \kwedge \om|^{2}&= 2\om_{kiP}\om_{jl}\,^{P}\om^{ki}\,_{Q}\om^{jlQ} - 2\mt \leq 4\om_{kiP}\om_{jl}\,^{P}\om^{ki}\,_{Q}\om^{jlQ} \leq 4|\om|^{4},
\end{aligned}
\end{align}
the last inequality by \eqref{om4}. This shows \eqref{lijklnormupper}. 
\end{proof}

%%% SAVE: correct - is used!
\begin{corollary}
Let $(\ste, h)$ be a Euclidean vector space of dimension $n \geq 2$. For $k \geq 2$ and $\om \in S^{k}_{0}\std$ there hold
\begin{align}
\label{lijklnormb}
&\left(4 - \tfrac{2}{n(n-1)}\right)|\om|^{4} \geq |\tf(\om \kwedge \om)|^{2}+ \tfrac{4}{n-2}|\mr{\rictr(\om\kwedge \om)}|^{2} \geq |\tf(\om \kwedge \om)|^{2}.
\end{align}
\end{corollary}
\begin{proof}
From \eqref{tfweylnorm} there follows
\begin{align}\label{lijklnorm2}
\begin{aligned}
|\om \kwedge \om|^{2}&=  |\tf (\om \kwedge \om)|^{2}+ \tfrac{4}{n-2}|\mr{\rictr(\om\kwedge \om)}|^{2} + \tfrac{2}{n(n-1)}|\om|^{4}_{h}\geq \tfrac{2}{n(n-1)}|\om|^{4}.
\end{aligned}
\end{align}
Combining \eqref{lijklnormupper} with \eqref{lijklnorm2} yields \eqref{lijklnormb}.
\end{proof}

%%% SAVE: correct, but used only in preceding
%\begin{lemma}
%Let $(\ste, h)$ be a Euclidean vector space of dimension $n \geq 2$. For $k \geq 2$ and $\om \in S^{k}_{0}\std$ there holds
%\begin{align}
%\label{lijklnormlower}
%&|\om \kwedge \om|^{2}
%\geq \tfrac{2}{n(n-1)}|\om|^{4}.
%\end{align}
%\end{lemma}
%\begin{proof}
%From \eqref{tfweylnorm} there follows
%\begin{align}\label{lijklnorm2}
%\begin{aligned}
%|\om \kwedge \om|^{2}&=  |\tf (\om \kwedge \om)|^{2}+ \tfrac{4}{n-2}|\mr{\rictr(\om\kwedge \om)}|^{2} + \tfrac{2}{n(n-1)}|\om|^{4}_{h}\geq \tfrac{2}{n(n-1)}|\om|^{4}.
%%\\%&= |\tf (\om \kwedge \om)|^{2}+ \tfrac{4}{n-2}|\rictr(\om\kwedge \om)|^{2} - \tfrac{2}{n(n-1)(n-2)}|\om|^{4}.
%\end{aligned}
%\end{align}
%\end{proof}

\begin{corollary}\label{k4plusnormcorollary}
Let $(\ste, h)$ be a Euclidean vector space of dimension $n \geq 2$. For $k \geq 4$ and $\om \in S^{k}_{0}\std$ there holds
\begin{align}\label{sk4plusest}
\tfrac{k-1}{2}|\om \kwedge \om|^{2} + |\rictr(\om \kwedge \om)|^{2} &\leq\left(1 + \tfrac{(2n+1)(k-1)}{n}\right)|\om|^{4}.
\end{align}
In particular, $\skmax_{n, k} \leq 1 + \tfrac{(2n+1)(k-1)}{n}$.
\end{corollary}
\begin{proof}
By \eqref{lijklnorm2}, \eqref{lijnormb}, and \eqref{lijklnormb},
\begin{align}\label{skmaxest}
\begin{aligned}
\tfrac{k-1}{2}|\om \kwedge \om|^{2} &+ |\rictr(\om \kwedge \om)|^{2} =  \tfrac{k-1}{2}|\tf(\om \kwedge \om)|^{2} + \tfrac{n+2(k-2)}{n-2}|\tf \rictr(\om\kwedge  \om)|^{2} + \tfrac{n+k-2}{n(n-1)}|\om|^{4}\\
&\leq \left(\tfrac{k-1}{2}\left(4 - \tfrac{2}{n(n-1)}\right) + \tfrac{n+k-2}{n} + \tfrac{n+k-2}{n(n-1)}\right) |\om|^{4}\\
&=\left( (k-1)\left(2 - \tfrac{1}{n(n-1)}\right) + \tfrac{n+k-2}{n-1}\right)|\om|^{4} =\left(1 + \tfrac{(2n+1)(k-1)}{n}\right)|\om|^{4},
\end{aligned}
\end{align}
which shows the claim.
\end{proof}

\subsection{{Proofs of Theorems \ref{scalarcurvaturetheorem} and \ref{simonstheorem}}}\label{proofsection}
Theorem \ref{cyestimatetheorem} is needed in the proof of Theorem \ref{scalarcurvaturetheorem}. Its statement can be found in the form given here, in the context of Hermitian manifolds, as \cite[Theorem $4.1$]{Tosatti-schwarz}.

\begin{theorem}[{S.~Y. Cheng and S.~T. Yau, \cite[Corollary $1$ on p. $857$]{Cheng-Yau-affinehyperspheresI}, \cite[Corollary to Theorem $8$ on p. $353$]{Cheng-Yau-differentialequations}}]\label{cyestimatetheorem}
Let $(M, g)$ be a complete $n$-dimensional Riemannian manifold with Ricci curvature bounded from below by $-\ka(n-1)g$ for a real constant $\ka \geq 0$. Suppose $u \in C^{2}(M)$ is nonnegative %
and satisfies $\lap u \geq Bu^{1 + \si} - Au$ for constants $B > 0$, $\si > 0$, and $A \in \rea$. 
If $A \leq 0$, then $u$ is identically zero, while, if $A > 0$, there holds $\sup_{M}u \leq |A/B|^{1/\si}$.
\end{theorem}

\begin{proof}[Proof of Theorem \ref{scalarcurvaturetheorem}]
By assumption, $\om \in \Ga(\symkt) \cap \ker \div \cap \ker \klie$ solves \eqref{projectivehiggsintro} for \revnew{some $c>0$, with $\ka$ constant by Lemma \ref{projectivehiggslemma}}. By \eqref{projectivehiggsintro}, $\rictr(\sR) = c\rictr(\om \kwedge \om) + \tfrac{\ka}{n}h$, 
and $\scal(\sR) = c|\om|^{2} +\ka$. Together with \eqref{qyalbe} these observations yield
\begin{align}\label{phscal1}
\begin{aligned}
\tfrac{1}{k}\qR(\om) &= \tfrac{k-1}{2}\lb \sR, \om \kwedge \om \ra + \lb \rictr(\sR), \rictr(\om\kwedge  \om)\ra \\
& = c\left(\tfrac{k-1}{2}|\om \kwedge \om|^{2} + |\rictr(\om \kwedge \om)|^{2} \right) + \tfrac{\ka(n+k-2)}{n(n-1)}|\om|^{2}.
\end{aligned}
\end{align}
Since $c > 0$, it follows from \eqref{phscal1} and \eqref{loweromtriple} of Lemma \ref{lowerkwedgelemma} that
\begin{align}
\label{phscal2}
\begin{aligned}
\tfrac{1}{k}\qR(\om) \geq  \left(c\skmin_{n, k}|\om|^{2} + \tfrac{\ka(n+k-2)}{n(n-1)}\right)|\om|^{2} \geq   \tfrac{n+k-2}{n(n-1)}\left(c|\om|^{2} + \ka\right)|\om|^{2}.
\end{aligned}
\end{align}
Substituting \eqref{phscal2} into \eqref{sharplapom} of Lemma \ref{swlemma} yields
\begin{align}
\label{phscal3}
\begin{aligned}
\lap_{h}|\om|^{\tfrac{n-2}{n+k-2}} &\geq \tfrac{n-2}{k(n+k-2)}\qR(\om) |\om|^{- \tfrac{n+2(k-1)}{n+k-2}}\geq  \tfrac{n-2}{n(n-1)}\left(c|\om|^{2} + \ka\right)|\om|^{\tfrac{n-2}{n+k-2}} 
 \\&=  \tfrac{n-2}{n(n-1)}\left( c\left(|\om|^{\tfrac{n-2}{n+k-2}}\right)^{1 + \tfrac{2(n+k-2)}{n-2}} + \ka|\om|^{\tfrac{n-2}{n+k-2}}\right).
\end{aligned}
\end{align}
Because $h$ is complete, Theorem \ref{cyestimatetheorem} applies. If $\ka \geq 0$, it implies $\om$ vanishes. In this case $h$ is a metric of constant sectional curvature. If $\ka < 0$, Theorem \ref{cyestimatetheorem} yields $\sup_{M}|\om|^{2} \leq -\ka/c$. This implies $\sup_{M}\scal \leq 0$.

When $k = 2$, combining \eqref{phscal1} with \eqref{lijklnormbk2b} yields
\begin{align}\label{phscalk2}
\begin{aligned}
\tfrac{1}{2}\qR(\om) & = c\left(\tfrac{1}{2}|\om \kwedge \om|^{2} + |\rictr(\om \kwedge \om)|^{2} \right) + \tfrac{\ka}{n-1}|\om|^{2} =c|\om|^{4}+ \tfrac{\ka}{n-1}|\om|^{2} = |\om|^{2}\left( \tfrac{\ka}{n-1} + c|\om|^{2}\right).
\end{aligned}
\end{align}
Substituting \eqref{phscalk2} into \eqref{sharplapom} of Lemma \ref{swlemma} yields
\begin{align}
\label{phscalk23}
\begin{aligned}
\tfrac{n}{n-2}\lap_{h}|\om|^{\tfrac{n-2}{n}} &\geq \tfrac{1}{2}\qR(\om) |\om|^{- \tfrac{n+2}{n}}\geq 
 \left(c|\om|^{2} + \tfrac{\ka}{n-1}\right)|\om|^{\tfrac{n-2}{n}}
 =   c\left(|\om|^{\tfrac{n-2}{n}}\right)^{1 + \tfrac{2n}{n-2}} + \tfrac{\ka}{n-1}|\om|^{\tfrac{n-2}{n}}.
\end{aligned}
\end{align}
The rest of the argument is as that following \eqref{phscal3}, but now yields $\sup_{M}|\om|^{2} \leq -\tfrac{\ka}{c(n-1)}$. This implies $\sup_{M}\scal \leq \tfrac{n-2}{n-1}\ka < 0$.
\end{proof}

\begin{remark}\label{ishihararemark}
Consider a pseudo-Riemannian space form $(N, g)$ with constant negative sectional curvature $-c$ (in Lorentzian signature such a space form is often called \emph{anti de Sitter space}). Combining the $k =2$ case of Theorem \ref{hypersurfacetheorem} with Theorem \ref{scalarcurvaturetheorem} shows that the squared-norm $|\sff|^{2}$ of the second fundamental form of a mean curvature zero spacelike hypersurface complete in the induced metric satisfies $|\sff|^{2} \leq -\tfrac{1}{n+1}\scal_{g} = -nc$. This recovers a theorem of Ishihara \cite[Theorem $1.2$]{Ishihara-maximal}. 
\end{remark}

\begin{proof}[Proof of Theorem \ref{simonstheorem} and Corollary \ref{simonscorollary}]
By assumption $\om \in \Ga(\symkt) \cap \ker \div \cap \ker \klie$ solves \eqref{projectivehiggsintro} for \revnew{some $c < 0$, with $\ka$ constant by Lemma \ref{projectivehiggslemma}}. The equation \eqref{phscal1} remains valid. For $k \geq 2$, since $c < 0$, it follows from \eqref{phscal1} and \eqref{loweromtriple} of Lemma \ref{lowerkwedgelemma} that
\begin{align}
\label{phscal2b}
\begin{aligned}
\tfrac{1}{k}\qR(\om) \geq  \left(c\skmax_{n, k}|\om|^{2} + \tfrac{\ka(n+k-2)}{n(n-1)}\right)|\om|^{2}.
\end{aligned}
\end{align}
Combining \eqref{lapomliediv} and \eqref{phscal2b} yields
\begin{align}\label{presimons}
\begin{aligned}
0 &= \tfrac{1}{2}\int_{M}\lap_{h}|\om|^{2}d\vol_{h}  = \int_{M}\left(|D\om|^{2} +\tfrac{1}{k} \qR(\om)\right)d\vol_{h}
\\ &\geq \int_{M}\left(|D\om|^{2} +  |\om|^{2}\left(  \tfrac{n+k-2}{n(n-1)}\ka + c\skmax_{n, k}|\om|^{2} \right)\right)d\vol_{h},
\end{aligned}
\end{align}
which shows Theorem \ref{simonstheorem}.
Taking $k = 3$ in \eqref{presimons} and using \eqref{lijklnormk3} of Lemma \ref{k3normlemma} yields \eqref{simonsinequalityk3}. When $k > 3$, \eqref{sk4plusest} of Lemma \ref{k4plusnormcorollary} and \eqref{presimons} yield \eqref{simonsinequality}. 

When $k = 2$, combining \eqref{lapomliediv} and \eqref{phscal1} yields
\begin{align}
\begin{aligned}
0 &= \tfrac{1}{2}\int_{M}\lap_{h}|\om|^{2}d\vol_{h}  = \int_{M}\left(|D\om|^{2} + \tfrac{1}{k}\qR(\om)\right)d\vol_{h} \\
&\quad = \int_{M}\left(|D\om|^{2} +  |\om|^{2}\left( \tfrac{\ka}{n-1} + c|\om|^{2}\right)\right)d\vol_{h},
\end{aligned}
\end{align}
which shows \eqref{simonsinequalityk2}. 
\end{proof}

\bibliographystyle{amsplain}
\def\polhk#1{\setbox0=\hbox{#1}{\ooalign{\hidewidth
  \lower1.5ex\hbox{`}\hidewidth\crcr\unhbox0}}} 
\providecommand{\bysame}{\leavevmode\hbox to3em{\hrulefill}\thinspace}
\providecommand{\MR}{\relax\ifhmode\unskip\space\fi MR }
\providecommand{\MRhref}[2]{%
  \href{http://www.ams.org/mathscinet-getitem?mr=#1}{#2}
}
\providecommand{\href}[2]{#2}

\end{document}